\documentclass[12pt]{amsart}
\usepackage[utf8]{inputenc}
\usepackage{amsmath, latexsym, amsfonts, amssymb, amsthm, amscd,amsaddr}
\usepackage[left=3cm, right=2.5cm, top=2.5cm, bottom=2.5cm]{geometry}
\usepackage{color}
\usepackage{hyperref}
\usepackage{stmaryrd}
\usepackage{enumerate}
\usepackage{float}
\usepackage{subcaption}
\usepackage{graphicx}
\usepackage{mfirstuc}
\usepackage[section]{placeins}
\usepackage{adjustbox}
\newtheorem{proposition}{Proposition}
\newtheorem{lemma}{Lemma}
\newtheorem{rem}{Remark}

\newtheorem{theorem}{Theorem}

\hypersetup{
    colorlinks=true,
    citecolor=blue,
    linkcolor=blue,
}

\usepackage[
backend=biber,
style=numeric,
citestyle=numeric,
maxbibnames=99,
maxcitenames=2
]{biblatex}       
\AtEveryBibitem{\clearfield{url}}

\usepackage{comment}

\newcommand{\R}{\mathbb{R}}%Reales

\newcommand{\E}{\mathbb{E}}%Valor esperado
\newcommand{\Var}{\text{Var}} % Variance
\newcommand{\totvar}{\emph{V}} % V majuscule pour la variation totale de la fonction

\def\Bea{\begin{eqnarray*}}
\def\Eea{\end{eqnarray*}} 
\def\bea{\begin{eqnarray}}
\def\eea{\end{eqnarray}} 
\usepackage{tikz}

\begin{document}
 \title[]{Estimation of the lifetime distribution from fluctuations in Bellman-Harris processes}
%\author{Jules {\sc Olayé}, Hala {\sc Bouzidi}, Andrey {\sc Aristov}, Antoine {\sc Barizien}, Salomé {\sc Gutiérrez Ramos}, Charles {\sc Baroud}, Vincent {\sc Bansaye}}

\author{J\MakeLowercase{ules} {O\MakeLowercase{layé}}$^{1,*}$, H\MakeLowercase{ala} {B\MakeLowercase{ouzidi}}$^{2}$, A\MakeLowercase{ndrey} {A\MakeLowercase{ristov}}$^{3,4}$, A\MakeLowercase{ntoine} {B\MakeLowercase{arizien}$^{3,4}$}, \\ S\MakeLowercase{alomé} {G\MakeLowercase{utiérrez} R\MakeLowercase{amos}}$^{3,4}$, C\MakeLowercase{harles} {B\MakeLowercase{aroud}}$^{3,4}$, V\MakeLowercase{incent} {B\MakeLowercase{ansaye}}$^{1}$}
%\address{$^1$CMAP, INRIA, École polytechnique, Institut Polytechnique de Paris, 91120 Palaiseau, France}
%\address{$^2$ENSTA Paris, Institut Polytechnique de Paris, 91120 Palaiseau, France}
%\address{$^3$Physical Institut Pasteur, Université Paris Cité, Microfluidics and Bioengineering Unit,75015 Paris, France}

\let\thefootnote\relax\footnotetext{$^{1}$ CMAP, INRIA, École polytechnique, Institut Polytechnique de Paris, 91120 Palaiseau, France}
\let\thefootnote\relax\footnotetext{$^{2}$ ENSTA Paris, Institut Polytechnique de Paris, 91120 Palaiseau, France}
\let\thefootnote\relax\footnotetext{$^{3}$ Institut Pasteur, Université Paris Cité, Physical microfluidics and Bioengineering, F-75015 Paris, France}
\let\thefootnote\relax\footnotetext{$^{4}$ LadHyX, CNRS, Ecole Polytechnique, Institut Polytechnique de Paris, 91120 Palaiseau, France}
\let\thefootnote\relax\footnotetext{$^{*}$ Corresponding author: \href{mailto:jules.olaye@polytechnique.edu}{jules.olaye@polytechnique.edu}}
%\affiliation[adresseF]{organization = {Université Paris Saclay, Inria, Centre Inria de Saclay}, postcode = {91120}, city = {Palaiseau}, country ={France}}
%\affiliation[adresseJ]{organization = {Institut Polytechnique de Paris, Inria, Centre de Mathématiques Appliquées}, postcode = {91120}, city = {Palaiseau}, country ={France}}
 
%\author{Vincent Bansaye}
%\address{CMAP, \'Ecole Polytechnique, Route de Saclay, F-91128 Palaiseau Cedex, France}
%\email{vincent.bansaye@polytechnique.edu}
% \vspace{0.2in}
% \noindent {\sc Key words and phrases}:  
% \bigskip
% \noindent MSC 2000 subject classifications: 60G17, 60G51, 60J80.
% \vspace{0.5cm}

\begin{abstract}  
%We are interested in the estimation of the parameters of the lifetime of individuals when the population is modeled by a Bellman-Harris process. This means that the lifetime of each individual is assumed to be independent and identically distributed. The population starts from small initial values and the full population is observed at successive times. The exponential growth of the population size at the beginning offers a first easy estimation linked to the mean of the lifetime. Going farther and describing lifetime variability is a challenging issue, due to the complexity of the fluctuations of non-Markovian branching processes. 
The growth of populations without interactions can often be modeled by branching processes where each individual evolves independently and with the same law. In Bellman-Harris processes, each individual lives a random time and is then replaced by a random number of offspring. We are interested in the estimation of the parameters of this model. Our motivation comes from the estimation of cell division time and we focus on Gamma distribution for lifetime and binary reproduction. The mean of the lifetime is closely related to the growth rate of the population. Going farther and describing lifetime variability from fixed time observations is a challenging task, due to the complexity of the fluctuations of non-Markovian branching processes. Using fine results on these fluctuations, we describe two time-asymptotic regimes and explain how to discriminate between them and estimate the parameters. Then, we consider simulations and biological data to validate and discuss our method. It allows to determine single-cell parameters from time-resolved measurements of populations without the need to track each individual or to know the details of the initial condition. The results can be extended to more general branching processes.
\end{abstract}
\maketitle
% \tableofcontents
\noindent\textit{\small{Keywords: Bellman-Harris processes, asymptotic fluctuations, estimation, inverse problem,  cell division}}

\section{Introduction}
%\ju{Je me suis permis de modifier l'encodement des caractères qui était en "latin1" et je l'ai mis en "utf8". Cela permet notamment de pouvoir écrire les accents. Si cela change quelquechose, ne pas oublier de remettre à la fin en latin1.}

Branching processes are widely used for modeling populations where individuals may reproduce or die, and evolve independently.  The simplest Markovian branching process is the Galton-Watson process. In continuous time, each  individual  lives during an exponential time and is then replaced by a random number of offspring.  This model and its extensions have been used and applied 
in population dynamics and evolution \cite{durrett_branching_2015,haccou_branching_2005,kimmel_branching_2015,meleard_stochastic_2015}, epidemiology
\cite{ball_strong_1995, britton_stochastic_2019},  queuing systems like polling \cite{vatutin_polling_2011}, nuclear  physics \cite{cox_multi-species_2019} etc. The exponential distribution of lifetimes corresponds to a memory less property and absence of aging of individuals. 

For many models and issues in life sciences, such a distribution is not  relevant and does not fit with data and observations. For instance, the time for cell division  rather looks like an unimodal distribution, more or less concentrated around its mean, see e.g. \cite{billy_synchronisation_2014,stukalin_age-dependent_2013,taheri-araghi_cell-size_2015} and references therein. The  variability of the lifespan  can be attributed to different sources.  Various models have been considered to describe it,  including a trait driving the division (the time from birth, the size, the increment of size from birth) or taking into account environmental variability, or individual variability, see e.g. ~\cite{barizien_growing_2019, cooper_1991, taheri-araghi_cell-size_2015}.  Similarly, in epidemiology, the time of infection for an individual is not considered exponentially distributed \cite{lloyd_realistic_2001}; rather, it is more accurately described by an unimodal time with a varying transmission rate. Another statistical limitation of exponential law is that the value of the mean characterizes the full distribution and  thus forces the variance of the distribution and the pattern of  the variability.
% S. Cooper. Bacterial Growth and Division : Biochemistry and Regulation of Prokaryotic and Eukaryotic Division Cycles. Elsevier, 1991.
% S. Taheri-araghi, S. Bradde, M. Vergassola, S. Jun, J. T Sauls, N. Hill, P. A. Levin, and J. Paulsson. Cell-Size Control and Homeostasis in Bacteria. Current Biology, 25(3):385–391, 2015.
%autre plus recents ? voir avec Marie qui connait bien, ou Ignacio

Various extensions of the Galton-Watson process in continuous time enable  to go beyond the exponential lifetimes and consider an age structure. In the Bellman-Harris process, the individuals live during independent random times that follow a common but general distribution. We are interested in inferring this distribution.  In biology or ecology or epidemiology, many data (experimental data or observations in wild life) consist of measuring the total population size along time, with no access to the values of the lifespan between two counts. 
However, up to our knowledge, the estimation of the parameters of the lifetime distribution from such population-level monitoring has not attracted a lot of attention so far.

Our work aims at proposing an efficient method for estimating the parameters of the distribution of lifetimes from such data set.  Going from the population sizes along a given time sequence to the growth rate of the population is direct by looking at the slope observed from data plotted at the log scale.  But going farther and obtaining  the variance and more generally the quantification of the variability of the lifespan of individuals 
 is much more delicate in the non-exponential case. This is due to  several reasons we are explaining in this work. Roughly for now, 
fluctuations of the population size  around its predicted value can have different and subtle behaviors depending on the time distribution and the time of observations.\\
%lifetime in non exponential cas is delicate, as we will see [dire que different regime en multitype avec deux vitesses engages]

%\charles{Add reference to Barizien et al. to say that it is not possible, with our microfluidic approach, to distinguish different division control mechanisms.}

Our motivation for this work and our first application is the estimation of the  time  for cell division using  microfluidics experiments. Indeed, single cell approaches have emerged as an important new way to address biological questions, with many formats to produce data that show heterogeneity of biological processes~\cite{wang2010robust}. We are particularly interested in experiments based on anchored droplets, where the contents of each microfluidic drop can be followed in time~\cite{amselem2016universal}. These experiments allow to obtain many parallel realizations of the growth of cell population starting from a small number of cells~\cite{barizien_growing_2019}. We want to exploit such data to infer the variability of cell division time, using Bellman-Harris processes as a simple statistical framework to model individual variability without heredity, or with negligible heredity. 

We choose here to focus on inference of the two parameters of Gamma lifetime distribution, with binary division. Indeed, Gamma distribution provides a convenient unimodal two parameters family which allows to cover realistic lifetime distribution. It is also a rare case among non-exponential laws which yields  some  explicit and useful computations. As we will see, fluctuations have a  complex behavior, which is a general phenomenon for branching processes.  It can be fully described here, thanks in particular to recent results on fluctuations of Crump-Mode-Jagers processes \cite{iksanov_asymptotic_2024}. This result can be seen as a starting point for more complex explorations to link the fluctuations at the level of population to the individual variability. In particular, in Section \ref{subsect:extension_other_bellman_harris}, we deal with extension of the results to more general lifetimes and non-binary reproduction events. \\

Let us be more explicit now on the setting and the results. We consider  two positive real numbers $k\geq1, \theta >0$. The process starts from one single individual and each individual lives, independently, during 
a random time distributed as a Gamma distribution $\Gamma(k,\theta)$. The density  $g_{k,\theta}$ of this law, denoted $g$ for short,   
is defined for all $t\in\R_+$ by
\begin{equation}
\label{defdens}
g_{k,\theta}(t)=g(t) = \frac{t^{k-1}e^{-\frac{t}{\theta}}}{\Gamma(k)\theta^k},
\end{equation}
where $\Gamma(y) = \int_0^{\infty} s^{y-1}e^{-s} ds$ for all $y\geq0$ is the Gamma function.
We denote by $(N_t)_{t\geq 0}$ the number of individuals at time $t$ and the observed quantity is
the number of individuals
$(N_{t_i})_{i=1...n}$ for a given sequence of times $(t_i)_{i=1...n}$.
Our objective is to determine the two unknown values $k,\theta$ as precisely as possible. 
It is equivalent  to determine the mean $\mu$ and variance $\sigma^2$ of the lifetime. For a  $\Gamma(k,\theta)$ law, they are   explicitly given by  %\ju{on a plutot fait une approche "estimation moyenne et coefficient de variation". Meme si cela est equivalent, ne faut-il pas plutot presenter la chose en "inference moyenne-coeff.var. plutot que moyenne-variance afin d'avoir quelque chose d'homogene entre l'intro et les simus-applications aux donnees ?}
$$\mu=k\theta, \qquad \sigma^2=k\theta^2.$$
Equivalently, 
we choose to  infer the pair $(\mu, \sigma/\mu)=(k\theta, 1/\sqrt{k})$ gathering the mean and coefficient of variation. % will be relevant and 
%k = \frac{\mu^2}{\sigma^2},  \quad \theta = \frac{\sigma^2}{\mu}.
%$$
Classically, in such setting, a first information comes from the observed Malthusian growth. Indeed, for Bellman-Harris processes under some general conditions \cite[Theorem $17.1$ and $21.1$]{Harris_1963},
\begin{align}\label{equivN}
\E(N_t) \underset{t\rightarrow \infty}{\sim} n_{1} e^{\alpha t}, \qquad  N_t\underset{t\rightarrow \infty}{\sim}\E(N_t) W \underset{t\rightarrow \infty}{\sim}  n_{1} e^{\alpha t}W\, \text{ a.s.},
\end{align}
where $n_1$ is inherited from the initial condition, $\alpha>0$ is the growth rate and  $W$ is a non-negative finite random variable.
%$$n_{1}=1 /\left[4 \alpha \int_{0}^{\infty} t e^{-\alpha t} g(t)dt\right]$$
%$$ 1=2 \int_{0}^{\infty} e^{-\alpha t} g(t)dt \text {. } $$
Besides, for $\Gamma(k,\theta)$ lifetime, $\alpha$ is explicitly known in function of $k,\theta$. Plotting the number of individuals observed at the log scale provides then the following first  estimation 
$$\frac{\log(N_t)}{t}\underset{t\rightarrow \infty}{\longrightarrow}  \alpha = \frac{1}{\theta}\left(2^{\frac{1}{k}} - 1\right) \qquad \text{a.s.}.$$
With real data of biological growth, a stable exponential growth  may indeed be observed during some suitable time window, namely after the biological lag phase and before the cell number approaches the carrying capacity of the droplet~\cite{barizien_growing_2019}. As such, this time window must avoid the early times for which the biological processes are not yet stationary. It also stops at large times when the independence and branching property fail due to saturation and competition. \\

The challenge is then to extract additional pertinent information from the data and to capture the two parameters  $(k, \theta)$. We assume that we have many observations of the process, i.e. many realizations of the Bellman-Harris process and the values of the population size at different times. One may  thus consider variance \cite{stukalin_age-dependent_2013} of the number of individuals. For Gamma distribution, 
\begin{equation}\label{eq:ratio_var_esp}
\lim_{t\rightarrow \infty} \frac{\text{Var}(N_t)}{\mathbb{E}[N_t]^2} =  q(\sigma/\mu),
\end{equation}
where the function $q$ is explicit, see~\eqref{eq:q_function}. This gives a theoretical way to conclude for estimation, and we refer to Section \ref{subsect:variancegamma_sensitivity} for details. 

However, various limitations exist for such an approach in practical applications. In particular, the microfluidic data involve several sources of variability that are not accounted for in this description~\cite{barizien_growing_2019}. First, the initial distribution of number of cells per droplet is not generally known exactly. Instead, it is usually assumed that cells distribute according to a Poisson process, which leads to a distribution of number of cells initially. Second, as already mentioned,  the initial divisions of the bacteria may happen with a different rate than the steady-state process, due to the cells adapting their biological mechanisms to the new environment. Taken together, these new sources of randomness play a strong role on the evolution of the cell number in the droplets, since they take place at the early stages of the exponential process~\cite{stukalin_age-dependent_2013}. 

Indeed, $\text{Var}(N_t)$ is very sensible to the initial number of individuals and the first steps of the process, i.e. the first lifespans in Bellman-Harris process. We are thus bound to forget a first time period of the trajectories of the Bellman-Harris process. We need  estimators after this time, which are the most sensible as possible to the parameters, and in particular to the variance $\sigma^2$ or the coefficient of variation $\sigma/\mu$. We also want to exploit our data set with successive observations as best as possible. Consequently, we need to use an alternative quantity. Whatever the initial condition of Bellman-Harris, the ratio $N_{t+\delta}/N_{t}$ converges almost surely to $\exp(\delta \alpha)$ as $t$ tends to infinity.  This leads us to consider the asymptotic fluctuations $$R_t^{\delta}=N_{t+\delta} -e^{\delta  \alpha} \, N_t,$$ %{thm:main_result}
where $\alpha$ is the Malthusian growth and  has been estimated in the first step.
Our issue is now to relate the “observed" distribution of $R_t^{\delta}$ for large $t$ to the parameters $(k,\theta)$
we want to determine. This link is delicate.  One may expect that  the order of magnitude of 
$R_t^{\delta}$ is $\sqrt{N_t}=\mathcal O(\exp(t \alpha/2))$. This regime
 corresponds to the classical Gaussian fluctuations in central limit theorem and the fact that for large time, the age of  individuals  can be seen as independent and picked
 according to the  limiting age distribution.
It indeed happens, but only when $k < k_c$, where $k_c$ is the unique solution in $[2,+\infty)$ of 
\begin{equation}\label{eq:definition_thresold_kc}
2\cos\left(2\pi/k_c\right)  = 2^{-1/k_c} + 1, 
\end{equation}
and is approximately equal to $57.24 \pm 0.01$, see~Figure~\ref{fig:thresold_spectral_gap}. In that case,  we prove the following convergence  in law as $t$ tends to infinity
\begin{equation}
\label{regGauss}
\frac{R_t^{\delta}}{\sqrt{N_t}} \overset{\mathcal{L}}{\underset{t\longrightarrow+\infty}{\Longrightarrow}} \mathcal N(0,\sigma_{\delta}^2),
\end{equation}
where  $\mathcal{N}$ is a Gaussian law, and its variance  $\sigma_{\delta}^2$ is given in \eqref{eq:expression_variance_first}. When $k > k_c$, the order of magnitude of $R_t^{\delta}$
is larger. Indeed, in that case, the speed of convergence of the age profile among the population
is too slow.  This convergence 
is quantified 
by the spectral gap $\alpha-\lambda$, i.e. the gap between the first and second eigenvalue of the mean operator, see Section~\ref{sect:asympt}. Indeed, $k > k_c$ implies that $\lambda>\alpha/2$. Hence, the speed of convergence $\alpha-\lambda$ becomes smaller than the  coefficient $\alpha/2$ quantifying fluctuations due to population sizes.
The leading term in $R_t^{\delta}$ 
%is no longer given by Gaussian fluctuations coming from an (asymptotically) i.i.d. sampling.  It 
comes then from the convergence 
of the age profile to its limiting distribution.
The renormalized process has then asymptotically deterministic oscillations along time. These oscillations are due to the lack of variability in division times, implying too much synchronicity in the division times (see Figure \ref{fig:oscillations_explication}). These oscillations involve the time step $\delta$ and a non-Gaussian, finite and complex random variable $M_{\delta}$, that may be equal to $0$: %linked to $W$  
\begin{equation}
\label{Regimeaussilent}
\frac{R_t^{\delta}}{\exp(\lambda t)}-2 \vert M_{\delta} \vert \cos(\tau t +\arg ( M_{\delta})) %\stackrel{t\rightarrow 0}
\overset{\mathbb{P}}{\underset{t\longrightarrow+\infty}{\longrightarrow}} 0,  %\quad \text{in probability},
\end{equation}
where  we use  the classical notation for modulus and argument of complex numbers and  %\ju{j'ai limpression que l'ajout de ma figure a entraîné du blanc à ces endroits. Comment les enlever ?}\vinc{je ne sais pas mais je suggere de reduir la légende de la figure grandement, ou differer la figure a la section concernee en fait, voir plus bas.}
\begin{equation}
\label{expcst}
\lambda=\frac{2^{\frac{1}{k}}\cos\left(\frac{2\pi}{k}\right) - 1}{\theta}  >\alpha/2, \qquad  \tau=\frac{2^{\frac{1}{k}}\sin\left(\frac{2\pi}{k}\right)}{\theta}.
\end{equation}
%and for $z\in \mathbb C$,  : 
%\begin{equation}\label{eq:power_complex_number}
%\ju{z^x = |z|^x\exp\left(ix\arg(z)\right), \qquad  \arg(z) \in(-\pi,\pi], \qquad |z|\in \mathbb R_+.}
%\end{equation}
%\alpha = \frac{1}{\theta}\left(2^{\frac{1}{k}} - 1\right)
%The parameter $\lambda$ corresponds to the spectral gap of the mean operator of the population structured by age. It is involved in, when   time going to infinity. In the first regime, the gap $\lambda$ is larger than $\alpha/2$ and  the speed of convergence of the age profile is "fast enough". The Gaussian  fluctuations around the stationary age profile, which corresponding to classical central limit theorem,  then yield the leading term of $R_t^{\delta}$. %This term is of order $\sqrt{N_t}=.
%In the second regime, the spectral gap $\lambda$ is smaller than $\alpha/2$ and the term corresponding to the convergence to the limiting  age profil prevails in  $R_t^{\delta}$. It is of order $\exp(\lambda t)$ and the fluctuations are non-longer Gaussian and  deterministic oscillations arise.
%This  depends on  the position of the spectral gap $\alpha-\lambda$ compared to the half of the Malthusian growth $\alpha/2$, which gives explicit conditions on the parameters for Gamma distribution. 
This classification in two regimes involving the spectral gap   is known for multitype
processes from the works of Athreya \cite{athreya_limit_1969_1,athreya_limit_1969_2}. In our study, the key ingredient  is the asymptotic behavior of $Y_t^{\delta}=\E(N_{t+\delta}\vert \mathcal F_t)-e^{\delta  \alpha} \, N_t$, where $(\mathcal F_t)_{t\geq 0}$ is the filtration of the process. For Bellman-Harris processes, in the  Gaussian regime, it has been obtained in \cite{kang_central_1999}. We add  then the Gaussian contribution of $X_t^{\delta}=N_{t+\delta}-\E(N_{t+\delta}\vert \mathcal F_t)$, which is at the same order of magnitude and asymptotically independent. This yields the long time behavior of $R_t^{\delta}=X_t^{\delta}+Y_t^{\delta}$ and the expression of the variance of the limiting Gaussian law $\sigma_{\delta}^2$.  For the  oscillating and non-Gaussian regime, we use recent and fine results of \cite{iksanov_asymptotic_2024} to get the asymptotic behavior of $Y_t^{\delta}$. The second contribution $X_t^{\delta}$ still behaves with Gaussian fluctuations of order $\sqrt{N_t}$ but is now negligible.     \\

We then exploit these results  to infer the value of the parameters $(k,\theta)$
from the observations of $R_t^{\delta}$.
The initial step is to discern which of the two regimes, \eqref{regGauss} or \eqref{Regimeaussilent},  we find ourselves. This can be achieved by quantifying the order of magnitude of $R_t^{\delta}$. For this step, one has to be careful to the time parameter $\delta$, since $\vert M_{\delta}\vert$  vanishes for some values of
$\delta$, see Section \ref{subsect:determination_regime}. 
%\vinc{[OK Jules, je te laisse ecrire le passage la dessus, deux ou trois phrases me sembleraient bien, idealement faire le rapport entre cette  synchronicite et la vitesse de convergence plus lente de la distribition des ages mentionne avant pour ce regime]} \ju{J'ai mis une explication en dessous de la figure 1, avec une référence à la figure 1 lorsque ça parle des oscillations dans l'introduction. Je ne sais pas par contre si la figure est bien placée, ni si ce genre d'explication en dessous d'une figure se fait beaucoup (je l'ai déjà vu personnellement). Il est dur d'expliquer la chose sans support visuel.}.\\
%A VOIR, je ne suis pas sur de faire le rapport avec ce qui est fait apres]}\ju{ C'est le terme $Y_{t}^{\delta}$ qui oscille. Qualitativement, ces oscillations sont liees a la synchronicite des temps de divisions. En effet, le regime oscillatoire apparait lorsqu'il y a peu de variabilite dans les temps de divisions. Du coup, les cellules vont se diviser plus ou moins "en meme temps", et on va voir s'enchainer en s'alternant "des periodes ou aucune cellule ne se divise", et "des periodes ou toutes les cellules se divisent". Ces enchainements vont faire apparaitre les oscillations. Le mieux est surement que cela moi qui ecrive un truc dessus ? (supprime le message apres l'avoir lu il prend de la place, au moins ce qu'il y a avant ma derniere question)}.  
The other difficulty comes from identifiability of the parameters since in the Gaussian regime, the limiting variance $\sigma_{\delta}^2$ is not injective with respect to the parameter $k$. We propose a procedure for the estimation which take into account these issues. 
 We evaluate our method 
by using  simulations, respectively in  Sections \ref{subsect:inference_Gaussian_regime} and   \ref{subsect:inference_oscillating_regime}
for the  Gaussian regime \eqref{regGauss}  and the oscillating regime \eqref{Regimeaussilent}. We recover in any case the parameters of the lifetime from the population size at given times and study the speed of convergence. This shows the efficiency of our procedure in our setting. Finally, in Section  \ref{desdata}, we use our approach for inference  on  two data set. This allows to estimate the heterogeneity for the cell division time from monitoring at the population level, which  is our original motivation for
this work. 
\\

\section{Regimes of asymptotic fluctuations}
\label{sect:asympt}

To present the different regimes of convergence of the fluctuations, we need to introduce some notations. 
Recall that $k \in [1,\infty)$ and $\theta\in (0,\infty)$ are fixed and the density $g$ of the Gamma law with parameter $(k,\theta)$ has been given in \eqref{defdens}. Recall also that the power of a complex number is defined for all $z\in\mathbb{C}^*$, $x\in\mathbb{R}$ as
\begin{equation}\label{eq:power_complex_number}
z^x = \left|z\right|^x\exp\left(ix\text{arg}(z)\right),
\end{equation}
%$g(t) = t^{k-1}\exp(-\frac{t}{\theta})/({\Gamma(k)\theta^k})$ is the density of the Gamma law. We restrict ourselves to the Gamma distribution with $k \geq 1$, because the case where $k < 1$ concerns distributions whose standard deviation is greater than the mean, and is therefore not relevant to our motivation. \ju{In all this work, the power of a complex number is defined for all $z\in \mathbb{C}^*$ and $x\in\mathbb{R}$ as
where $\text{arg}(z)\in(-\pi,\pi]$. We introduce the cumulative distribution function $G$, and the Laplace transform $\mathcal{L}g$ related to our distribution,
%Let us denote for all $a\geq0$ \ju{Pas trop inspir\'e sur la maniere d'introduire $p(a)$ et la CDF. A revoir je pense...}
$$
G(a) = \int_0^a g(s) ds= \int_0^a \frac{s^{k-1}e^{-\frac{s}{\theta}}}{\Gamma(k)\theta^k}ds,\qquad \mathcal{L}g(\rho)=\int_0^{+\infty} g(u) e^{-\rho u} du = \frac{1}{(1+\rho\theta)^k},$$
for respectively $a\geq 0$, and $\rho \in \mathbb{C}$ such that $Re(\rho) >-1/\theta$. With an abuse of notation, for all $\rho \in \mathbb{C}\backslash\{-1/\theta\}$ such that $Re(\rho) \leq -1/\theta$, we also denote $\mathcal{L}g(\rho) = 1/(1+\rho\theta)^k$ the analytic continuation of $\mathcal{L}g$. Finally, we introduce the following age distribution
\begin{equation}\label{eq:stationary_age_distribution}
p(a) := \frac{e^{-\alpha a}(1-G(a))}{\int_0^{\infty} e^{-\alpha u}(1-G(u)) du}.
\end{equation}
The Bellman-Harris process is defined inductively by starting from an initial number of individuals $N_0\in \mathbb{N}^*$ and initial ages $(A_i)_{i=1, \ldots N_0}$. 
 The individuals live during random times, which are independent and  distributed as a Gamma law $\Gamma(k,\theta)$. At the end of their  life, they are replaced by two individuals with age $0$. When we start from one single individual with age $a\in\mathbb{R}_+$, we denote by $\mathbb{E}_a$ and $\text{Var}_a$ the associated expectation and the associated variance. Unless otherwise specified, we start from one single individual with age $0$.
 %If nothing is specified, it means that we study an expectation or a variance associated with a dynamic starting from one individual with age $0$.\\

We now give the different regimes of convergence of the fluctuations. The result holds for a Bellman-Harris process starting from one individual with age $0$. Classical arguments for branching processes enable the extension of the results to more general initial conditions.

\begin{theorem}
\label{thm:main_result}
Consider a Bellman-Harris $N$ where the lifespan of individuals is distributed as $\zeta \sim \Gamma(k,\theta)$, with $k\geq 1, \theta >0$. The following statements hold.
\begin{enumerate}[$i)$]
    \item If $k < k_c$, then for any $\delta>0$,     $$
    \frac{R_t^{\delta}}{\sqrt{N_t}} \overset{\mathcal{L}}{\underset{t\longrightarrow+\infty}{\Longrightarrow}} \mathcal{N}\left(0, \sigma_{\delta}^2 \right),
    $$
  where, denoting for all $x\geq0$,
$$
j^{(\delta)}(x) = \mathbb{E}_x\left[N_{\delta}\right] - e^{\alpha\delta} \text{ and } h^{(\delta)}(x)=\mathbb{E}\left[N_{x+\delta}\right]- e^{\alpha\delta}\mathbb{E}\left[N_x\right],
$$ 
we have 
\begin{equation}\label{eq:expression_variance_first}
\begin{aligned}
\sigma_{\delta}^2 &= \int_{\mathbb{R}^+} \emph{Var}_a(N_{\delta})p(a)da + 2\alpha\int_{\mathbb{R}_+} \emph{Var}\Big(j^{(\delta)}(x)1_{[0,\zeta[}(x) \\
&\qquad \qquad \qquad \qquad +2h^{(\delta)}(x - \zeta)1_{[\zeta, +\infty[}(x)\Big) e^{-\alpha\,x}dx < +\infty.
\end{aligned}
\end{equation}
   %$$\sigma_{\delta}^2= \int_{\mathbb{R}^+} \emph{Var}_a(N_{\delta})p(a)da
   %+ 2\alpha\int_{\mathbb{R}_+} \emph{Var}\left(\mathbb{E}\left[Y_{t}^{\delta} | \mathcal{F}_{T_1}\right]\right)e^{-\alpha t} dt.$$
   
\item  If  $k = k_c$, 
then for any $\delta>0$, 
    $$
    \frac{R_t^{\delta}}{\sqrt{tN_t}} \overset{\mathcal{L}}{\underset{t\longrightarrow+\infty}{\Longrightarrow}} \mathcal{N}\left(0, {\sigma_{\delta}}^{2}\right),
    $$
   where, recalling the definitions of $\lambda$ and $\tau$ given in~\eqref{expcst}, we have 
   %\vinc{J'ai retiré $\zeta \sim \Gamma(k,\theta)$ et l'ai mis en début d'énoncé, ca me paraissait plus logique, car il n'apparait pas ici }
   %now  writing $\zeta \sim \Gamma(k,\theta)$, 
$$
\sigma_{\delta}^2 = \frac{\alpha}{k^2}\frac{2^{\frac{2}{k}}}{2^{\frac{2}{k}}-2^{\frac{1}{k}}}\left|e^{(\lambda+i\tau)\delta} - e^{\alpha\delta}\right|^2.
$$

\item If $k > k_c$,
then for any $\delta>0$,
% il existe $W'$ variable aléatoire complexe (indépendante de $\delta$) et $z_{\delta}\in\mathbb{C}$:
    $$
    \left| \frac{R_t^{\delta}}{\exp\left(\lambda t\right)} -  2\left|M_{\delta} \right| \cos\left[\tau t +  \arg\left(M_{\delta}\right)\right]\right|  \overset{\mathbb{P}}{\underset{t\longrightarrow+\infty}{\longrightarrow}} 0,$$
where 
\begin{equation}\label{eq:random_modulus_oscillating_regime}
M_{\delta}= \left(e^{(\lambda + i\tau)\delta}-e^{\alpha \delta}\right)M,
\end{equation}
and $M$ is a complex random variable that does not depend on $\delta$.
\end{enumerate}
\end{theorem}
%\ju{\begin{rem}
%The variance $\sigma_{\delta}^2$ defined in~\eqref{eq:expression_variance_first} corresponds in fact to the sum of the variances of two Gaussian distributions $\sigma_{X,\delta}^2$ and $\sigma_{Y,\delta}^2$, defined respectively in~\eqref{eq:variance_asymptotic_Xt} and~\eqref{eq:variance_asymptotic_Yt}.
%\end{rem}}
%\vinc{Dire  un mot sur $k\geq 2$?}; \ju{Je pense que le theoreme est encore vrai pour $k\in]1,2]$ juste cela alambique la preuve car on a pas de derivee bornees pour la fonction pour laquelle on applique Iksanov-Kolesko-Meiners dans ce cas.} \\
%A discuter mais je me demande si l'expression qui vient de \eqref{eq:to_change_variance_fluctuations} ne serait pas mieux, a la place de  $\mathbb{E}\left[Y_{t}^{\delta} | \mathcal{F}_{T_1}\right]$. Ca eviterait d'introduire a nouveau $Y_{t}^{\delta}$ en plus.  %\ju{Je comprends ton point, je pense que cela est peut-etre mieux. Reste a voir la place prend cette expression... A discuter.} \vinc{ok, en plus faudrait definir $T_1$ sinon, et j'ai l'impression qu'on peut faire plus simple, voir le passage concerne, on en parle de vive voix}\\
%$W_1$ is complex r.v. : can we say more, can it be expressed in function of $W$? \ju{Pour moi, on ne peut pas.} \vinc{Surprenant pour moi, il faut en parler alors. Si leur theoreme ne s'applique pas disons que c'est la raison. Est ce genant pour la suite ? peut on justifier qu'on se moque de $k\leq 2$ pour nos motivations ? \\}
This theorem shows that two different regimes may be observed, excluding the very specific critical behavior. They depend on the value 
of $k$ and correspond respectively to a Gaussian regime, with the expected order of magnitude $\sqrt{N_t}$, and an oscillatory regime. The threshold value $k_c$ between these two regimes has been defined in~\eqref{eq:definition_thresold_kc}, and is approximatively $57.24$, as illustrated in Figure~\ref{fig:thresold_spectral_gap}.

Recalling  that $k = 1/(\sigma/\mu)^2$, the threshold $k_c$ is equivalent to a coefficient of variation $\sigma/\mu$ of approximately $0.1322$. The oscillating regime (Theorem \ref{thm:main_result} $iii)$) then corresponds to a small variability of the lifespan, making the convergence to the age profile slow, while a large coefficient of variation helps for mixing and leads to the Gaussian regime (Theorem~\ref{thm:main_result}~$i)$). We illustrate this in Figure \ref{fig:oscillations_explication}.
%\begin{figure}[!ht]
%    \centering
%    \includegraphics[scale = 0.25]{images/divers/illustration_two_regimes_dynamics.jpg}
%    \caption{Illustration of the two regimes. The oscillatory regime corresponds to dynamics with lifetime distribution $\Gamma\left(k_{\text{oscillatory}},\theta_{\text{oscillatory}}\right)$, where $\left(k_{\text{oscillatory}},\theta_{\text{oscillatory}}\right) = (400,0.05)$. The Gaussian regime corresponds to dynamics with lifetime distribution $\Gamma\left(k_{\text{Gaussian}},\theta_{\text{Gaussian}}\right)$, where if we denote $\alpha_{\text{oscillatory}} = (2^{1/k_{\text{oscillatory}}} - 1)/\theta_{\text{oscillatory}}$, we have $\left(k_{\text{Gaussian}},\theta_{\text{Gaussian}}\right) = \left(6.25,(2^{k_{\text{Gaussian}}} - 1)/\alpha_{\text{oscillatory}}\right)$. Figure A is the mean of the evolution of $500$ Bellman-dynamics, in the two different regimes. Figure B represents several dynamics in the oscillatory regime, and Figure C represents several dynamics in the Gaussian regime.}
%        \label{fig:oscillations_explication}
%\end{figure}

\begin{figure}[!ht]
    \centering
%    \begin{subfigure}[t]{0.45\textwidth}
%        \centering
%    \includegraphics[width = \textwidth]{images/simus/gaussian-inference/study_time_step/step_time_without_identifiability_1over4.jpg}
%    \caption{For $\delta = \log(2)/(4\widehat{\alpha})$.}
%    \end{subfigure}
%    \hfill 
    \begin{subfigure}[t]{0.45\textwidth}
        \centering
    \includegraphics[width = \textwidth]{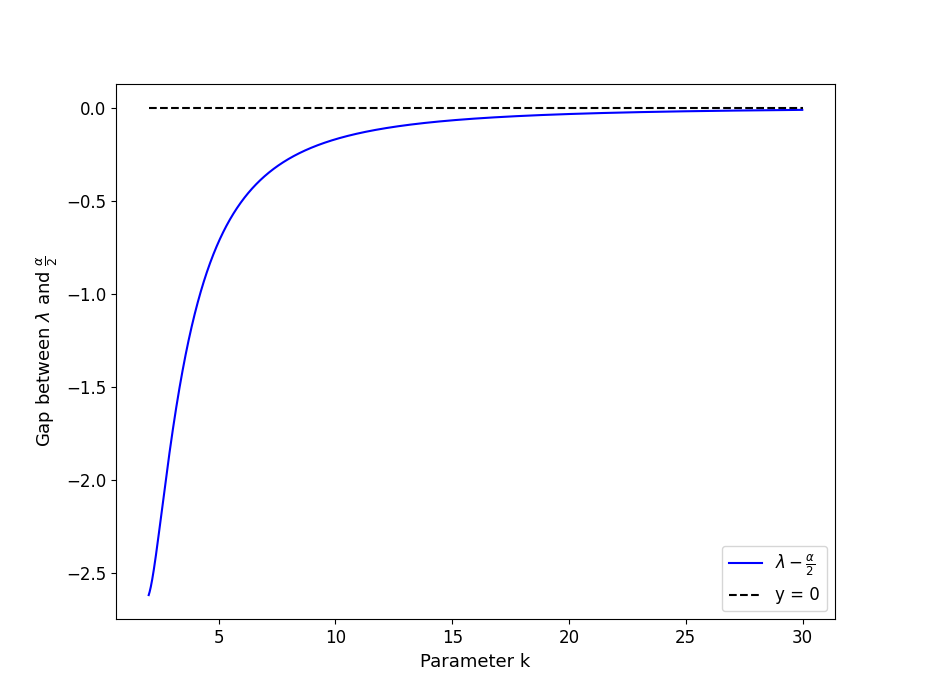}
    \caption{For $k\in[2,30]$.}
    \end{subfigure}
    \hfill
    \begin{subfigure}[t]{0.45\textwidth}
    \centering
    \includegraphics[width = \textwidth]{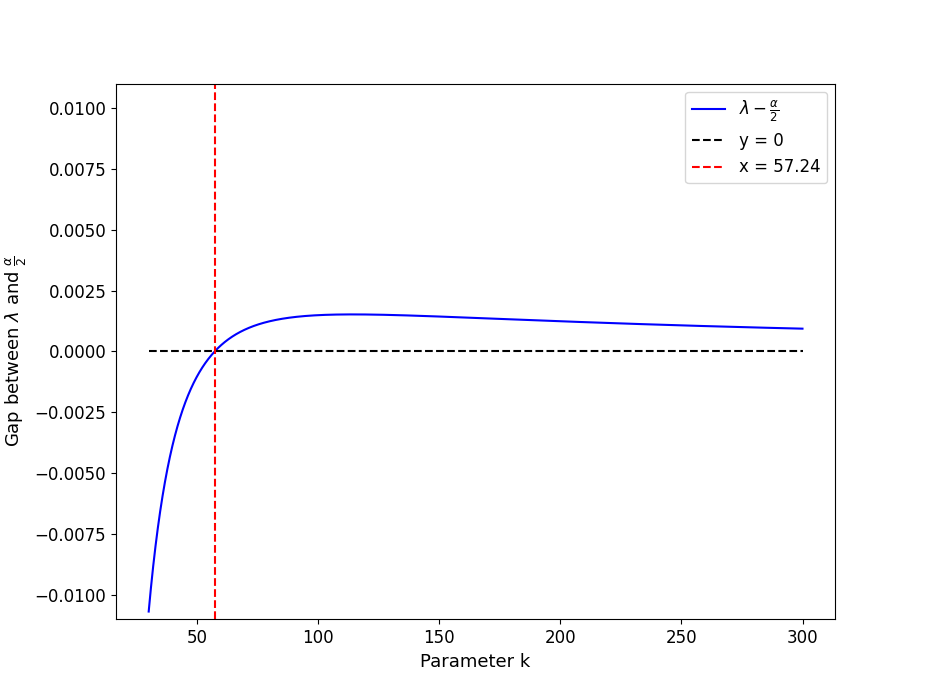}
    \caption{For $k\in[30,300]$.}
    \end{subfigure}
    \caption{Curve of $\lambda - \frac{\alpha}{2}$ $= \frac{2^{\frac{1}{k}}\left(2\cos\left(\frac{2\pi}{k}\right) -1\right) -1}{2\theta}$ versus the parameter $k$, for $\theta = 1$.}\label{fig:thresold_spectral_gap}
\end{figure}

\begin{comment}
\begin{figure}[!ht]
    \centering
    \includegraphics[scale = 0.31]{images/divers/high_parameters_spectral_gap.png}
    \caption{Curve of $\lambda - \frac{\alpha}{2}$ versus the parameter $k$ for $\theta = 1$.} 
    \label{fig:thresold_spectral_gap}
\end{figure} 
\end{comment}
\begin{figure}[!ht]
    \centering
    \begin{subfigure}[t]{0.375\textwidth}
        \centering
        \includegraphics[width=\textwidth]{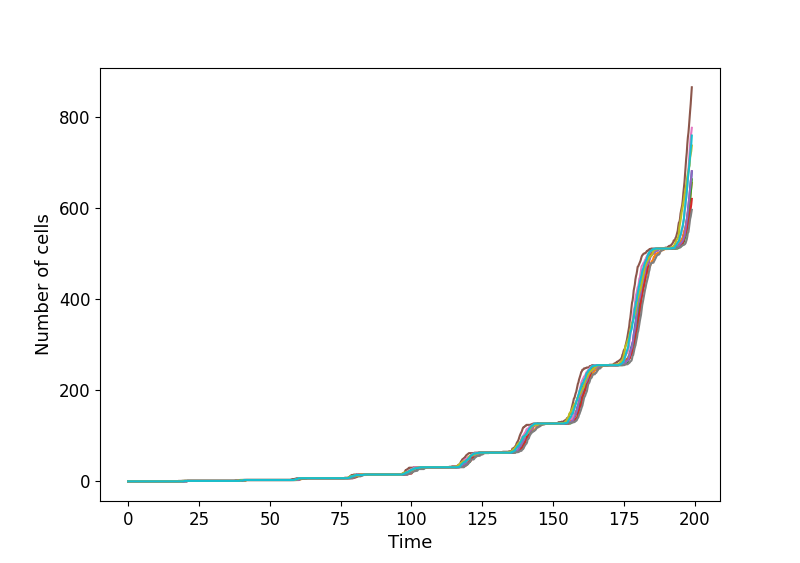}
        \caption{Ten dynamics in the oscillatory regime. This corresponds to the parameters $\left(k,\theta\right) = (400,0.05)$.}
        \label{fig:dynamics_oscillatory}
    \end{subfigure}
    \hfill
    \begin{subfigure}[t]{0.375\textwidth}
        \centering
        \includegraphics[width=\textwidth]{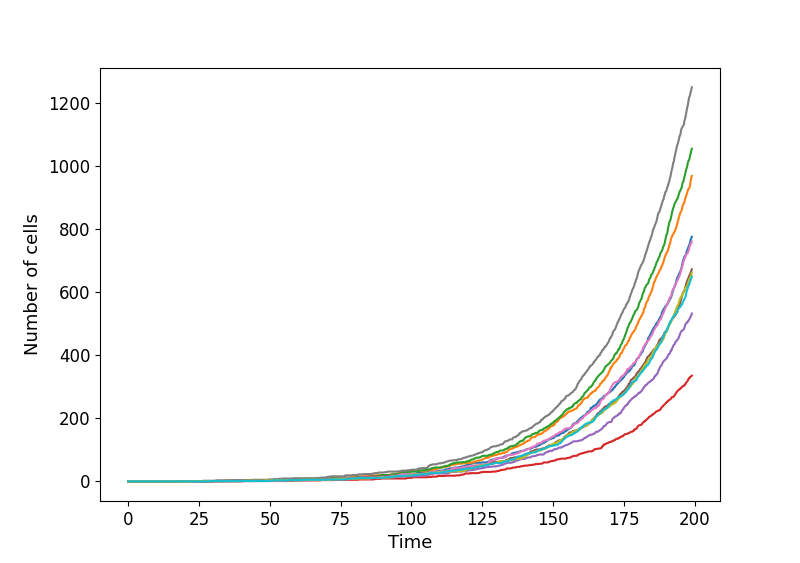}
        \caption{Ten dynamics in the Gaussian regime. This corresponds to the parameters $\left(k,\theta\right) = \left(6.25,3.38\right)$.}
        \label{fig:dynamics_gaussian}
    \end{subfigure}
    \hfill
    \begin{subfigure}[t]{0.375\textwidth}
        \centering
        \includegraphics[width=\textwidth]{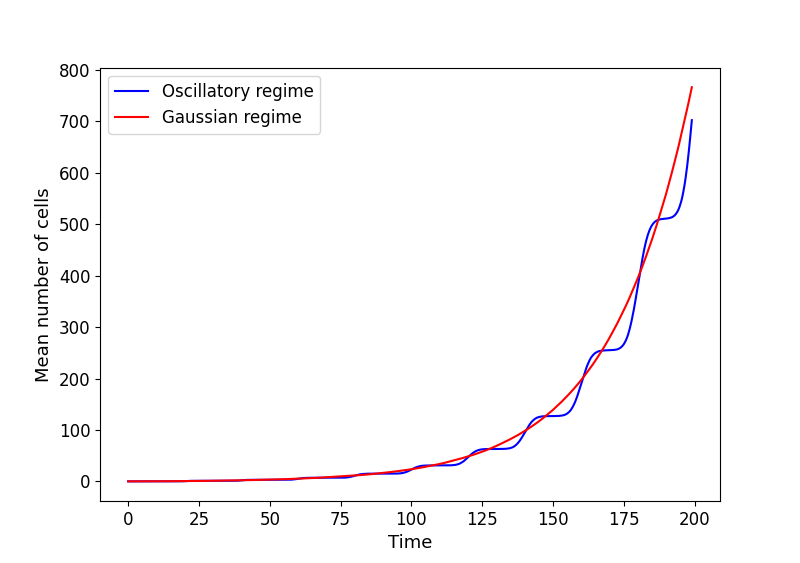}
        \caption{Evolution of the mean of $500$ trajectories. In blue, the oscillatory regime with parameters of ${\bf (a)}$. In red, the gaussian regime with parameters of~${\bf (b)}$.}
        \label{fig:mean_evolution}
    \end{subfigure}
    \caption{Illustration of the two regimes, for parameters with the same Malthusian coefficient.}
    \label{fig:oscillations_explication}
\end{figure}

We prove now  Theorem \ref{thm:main_result} by splitting  $R_t^{\delta}$ in two parts. At each time $t\geq0$, we can label the individuals by $i=1, \ldots, N_t$ and we denote by $N_{t,\delta}^i$ the number of offspring of individual $i$ at time $t+\delta$. We introduce 
 $$R_t^{\delta}=X_t^{\delta}+Y_t^{\delta},$$
 where
 \begin{equation}\label{eq:expression_decomposition_residuals}
 \begin{aligned}
 X_t^{\delta}&= N_{t+\delta}-\E(N_{t+\delta}\vert \mathcal F_t)=\sum_{i=1}^{N_t}  \left(N_{t,\delta}^i-\mathbb E(N_{t,\delta}^i \vert \mathcal F_t)\right), \\
 Y_t^{\delta}&=\E(N_{t+\delta}\vert \mathcal F_t)-e^{\delta  \alpha} \, N_t=\sum_{i=1}^{N_t} \left( \mathbb E(N_{t,\delta}^i \vert \mathcal F_t)-e^{\alpha \delta}\right).
 \end{aligned}
 \end{equation}
The result in  \cite{iksanov_asymptotic_2024}
applied to a suitable functional would allow to deal directly with $R_t^{\delta}$ but the decomposition   $R_t^{\delta}=X_t^{\delta}+Y_t^{\delta}$ is interesting for itself. It yields  two asymptotically independent contributions, which provide a natural expression of the limiting variance.
%\vinc{c'est surement le point ou on peut nous faire des reproches, cette expression se lirait elle directement de l'expression pour la limte dans \cite{iksanov_asymptotic_2024}?} \ju{A discuter mais je n'arrive vraiment pas a retrouver la meme expression de la variance sans separer les deux quantites. Le mieux est d'en discuter au tableau car cela est un peu technique.}\vinc{Ok, si c'est pas simple, ca justifie aussi de prendre le chemin qu'on a choisi}
\subsection{Asymptotic behavior of  \texorpdfstring{$X_t^{\delta}$}{X\_t\^d}} Let us study the asymptotic behavior of $(X_{t}^{\delta})_{t\geq0}$ for all $\delta >0$. We first remark by \eqref{eq:expression_decomposition_residuals} that for all $t\geq0$ and $\delta >0$, the quantity $X_{t}^{\delta}$ can be seen as the sum of $N_t$ variables that are centered and independent conditionally with respect to~$\mathcal{F}_t$. Thus, a slight adaptation of the central limit theorem  allows to obtain the asymptotic behavior of $(X_t^{\delta})_{t\geq0}$. To do so, we need the following upper bounds and regularity for  the moments of $N_{\delta}$.

\begin{lemma}\label{lemma:preliminaries_asymptotic_Xt} For any $\delta>0$, the following statements hold.
\begin{enumerate}[i)]
    \item $\sup_{a\in [0,\infty)} \left\{\mathbb{E}_a[N_{\delta}^2]+ \mathbb{E}_a\left[|N_{\delta} - \mathbb{E}_a[N_\delta]|^3\right] \right\}<\infty$.
    \item The function $a \mapsto \mathbb{E}_a\left[N_{\delta}\right]$ is continuous.
    \item The function $a \mapsto \mathbb{E}_a\left[N_{\delta}^2\right]$ is continuous.
\end{enumerate}
\end{lemma}
Note that $i)$ implies that the first moment and variance of $N_{\delta}$ are bounded with respect to the age $a$ of the root at time $0$,   for any fixed time $\delta>0$. This bound is natural : a large age accelerates the first reproduction but  in any case, we have then two individuals  with age $0$. 
\begin{proof}
Let us first prove $i)$ and let $a\geq 0$. We consider  a Bellman-Harris process $(N_t)_{t\geq 0}$ starting from  one single individual with age $a$. Then, the first time of division follows  a Gamma law with parameters $(k,\theta)$ conditioned to be larger than $a$ and its density is~\hbox{$x\in\mathbb{R}_+\mapsto g(a+x)/(1- G(a))$}. Using an integral equation (\cite[Theo. $15.1$]{Harris_1963} or~\hbox{\cite[Eq.~$7.1$]{crump_general_1968}}), we obtain by decomposition with respect to the age of death
\begin{equation}\label{eq:integral_equation}
\begin{aligned}
\mathbb{E}_a[N_{\delta}] &= 2I_{\delta}(a) + \frac{1-G(a+\delta)}{1- G(a)}, \quad \text{where } \, I_{\delta}(a)=\int_{0}^{\delta} \mathbb{E}_0[N_{\delta-x}] \frac{g(a+x)}{1- G(a)}dx . \\
\end{aligned}
\end{equation}
%density,  of the first time of division is  and $T_1^a$ the time before the first event of division for $(N_t)_{t\geq0}$. 
Moreover, using  \cite[Theorem $18.1$]{Harris_1963}, we have
\begin{equation}\label{eq:preliminaries_Xt_integraleq_square}
\mathbb{E}_a[N_{\delta}^2] = \int_0^{\delta}\mathbb{E}\left[\left(\overline{N}_{\delta - x}^{(1)} + \overline{N}_{\delta - x}^{(2)}\right)^2\right]\frac{g(a+x)}{1- G(a)} dx + \frac{1-G(a+\delta)}{1- G(a)},
\end{equation}
where $(\overline{N}_t^{(1)})_{t\geq0}$ and $(\overline{N}_t^{(2)})_{t\geq0}$ are two independent Bellman-Harris processes starting from an age $0$. Then, using the inequality $(y+z)^2 \leq 2y^2 + 2z^2$ for all $y,z\in\mathbb{R}$, and the increasing of $t \mapsto \mathbb{E}_0[N_t^2]$, allows to conclude that
\begin{equation}\label{eq:preliminaries_Xt_firstbound}
\mathbb{E}_a[N_{\delta}^2] \leq 4\mathbb{E}_0[N_{\delta}^2] + 1.
\end{equation}
Secondly, as $\mathbb{E}_a[N_\delta] \geq 0$ and $N_{\delta} \geq 0$ almost surely, we have 
$$
|N_{\delta} - \mathbb{E}_a[N_\delta]|^3 \leq \max\left(\mathbb{E}_a[N_\delta]^3, N_{\delta}^3\right) \hspace{1.5mm} \leq N_{\delta}^3 + \mathbb{E}_a[N_\delta]^3 \text{ almost surely}.
$$
In view of \eqref{eq:preliminaries_Xt_firstbound}, we easily have that $\mathbb{E}_a[N_\delta]^3 \leq \left(4\mathbb{E}_0[N_{\delta}^2] + 1\right)^{\frac{3}{2}}$. Proceeding as when we obtain \eqref{eq:preliminaries_Xt_firstbound} allows to obtain that $\mathbb{E}_a\left[N_{\delta}^3\right]$ is bounded by a constant independent from $a$. These two statements combined with \eqref{eq:preliminaries_Xt_firstbound} imply that $i)$ is proved. \\

Now, for $ii)$, as $k \geq 1$, one can easily see that $||g||_{\infty} < +\infty$. 
%By the integral equation, as $T_1^a~\sim~\Gamma(k,\theta)\,|\,T_1^a \geq a$, we have  
Using \eqref{eq:integral_equation} and the fact that $G$ is continuous, we only have to prove that $a\mapsto I_{\delta}(a)$ is continuous to obtain the continuity of $a\mapsto \mathbb{E}_a[N_{\delta}]$. It is easily shown by noticing that 
$$
\left|\mathbb{E}_0[N_{\delta-x}] g(a+x)1_{\{x\in[0,\delta]\}}\right| \leq ||g||_{\infty}\mathbb{E}_0[N_{\delta}] 1_{\{x\in[0,\delta]\}},
$$
and that the right-hand side  is integrable.  This proves $ii)$. \\

The proof of $iii)$ is very similar. We use the integral equation \eqref{eq:preliminaries_Xt_integraleq_square}, and then the following domination
$$
\left|\mathbb{E}\left[\left(\overline{N}_{\delta - x}^{(1)} + \overline{N}_{\delta - x}^{(2)}\right)^2\right] g(a+x)1_{\{x\in[0,\delta]\}}\right| \leq 4||g||_{\infty}\mathbb{E}_0[N_{\delta}^2] 1_{\{x\in[0,\delta]\}}.
$$
This completes the proof.
\end{proof}

With this lemma, we are able to obtain the asymptotic behavior of $(X_{t}^{\delta})_{t\geq0}$ for all~$\delta >0$ by a simple adaptation of the proof of the central limit theorem with Lindeberg's condition. Let us do this through the following proposition.
\begin{proposition}\label{prop:asymptotic_Xt} For any $\delta >0$, $s\in\mathbb{R}$,
$$
\mathbb{E}\left[ \exp\left( is\frac{X_{t}^{\delta}}{\sqrt{N_t}} \right) \, \bigg| \, \mathcal{F}_t\right] \overset{a.s.}{\underset{t\longrightarrow+\infty}{\longrightarrow}} \exp\left(-\frac{\sigma_{X,\delta}^2s^2}{2}\right),
$$
where $\sigma_{X,\delta}^2 = \int_{\mathbb{R}^+} \emph{Var}_a(N_{\delta}) \, p(a) \, da < +\infty$.
\end{proposition}
\begin{rem}
The variance $\sigma_{X,\delta}^2$ is finite for all $\delta >0$ because $x \mapsto \Var_{x}(N_{\delta})$ is bounded by Lemma~\ref{lemma:preliminaries_asymptotic_Xt}, and $p$ is  integrable.
\end{rem}
\begin{proof} Let $t\geq0,\,\delta >0$ and $j\in\llbracket1,N_t\rrbracket$. We denote by  $A_t^j$ the age of the individual $j$ at time $t$, i.e. the  age of the intial individual in   $(N_{s,\delta}^j)_{s\geq 0}$. Besides, we introduce the random variables
$$
Z_{t,\delta}^j = \frac{N_{t,\delta}^j - \mathbb{E}[N_{t,\delta}^j | \mathcal{F}_t]}{\sqrt{N_t}}, \qquad \sigma_{t,j,\delta}^2 = \Var_{A_t^j}(N_{\delta}),
$$
and for all $s\in\mathbb{R}$ 
$$%$e^{-s^2\sigma_{t,j,\delta}^2/2N_t} $
\begin{aligned}
K_{t,\delta}^j(s) = \mathbb{E}[ e^{isZ_{t,\delta}^j}| \mathcal{F}_t] - \exp\left(-\frac{s^2\sigma_{t,j,\delta}^2}{2N_t}\right).
\end{aligned}
$$
\noindent As $\mathbb{E}\left[Z_{t,\delta}^j | \mathcal{F}_t\right] = 0$ and $\mathbb{E}\left[(Z_{t,\delta}^j)^2 | \mathcal{F}_t\right] = \frac{\sigma_{t,j,\delta}^2}{N_t}$, we have by the triangular inequality that for all $s\in\mathbb{R}$
%\begin{equation}\label{eq:asymptotic_Xt_intermediate_zero}
%\left|K_{t,\delta}^j(s)\right| = \left|\mathbb{E}[ \Delta_{t,\delta}^j(s)| \mathcal{F}_t] - D_{t,\delta}^j(s) \right| \leq |\mathbb{E}[ \Delta_{t,\delta}^j(s)| \mathcal{F}_t]| + |D_{t,\delta}^j(s)|.
%\end{equation}
\begin{align}\label{eq:asymptotic_Xt_intermediate_zero}
&\left|K_{t,\delta}^j(s)\right| \leq \left|\mathbb{E}\left[e^{isZ_{t,\delta}^j} - \left(1 + isZ_{t,\delta}^j - \frac{s^2(Z_{t,\delta}^j)^2}{2}\right)\,\Big|\, \mathcal{F}_t \right]\right|\\
&\qquad \qquad \qquad \qquad \qquad \qquad + \left|\left(1 -\frac{s^2\sigma_{t,j,\delta}^2}{2N_t}\right)-e^{-s^2\sigma_{t,j,\delta}^2/2N_t}\right|. \nonumber
\end{align}%$z\in\left\{p\in\mathbb{C}\,|\,p\in\mathbb{R}\backslash\mathbb{R}_+\text{ or }Re(p) = 0\right\}$
%\comment{Ref a ajouter si il y a, j'ai corrige aussi le premier "for z"} 
In addition, for all $z\in\ \mathbb{R_-}\cup i\mathbb R$, $n\geq1$, we have
\begin{equation}\label{eq:asymptotic_Xt_intermediate_zero_bis}
\left|e^{z} - \sum_{j = 0}^{n} \frac{z^j}{j!}\right| \leq \frac{|z|^{n+1}}{(n+1)!},
\end{equation}
using Taylor-Lagrange formula for $z\in \R_{-}$ and  Lemma 10.1.5 in \cite{athreya_measure_2006} for $z\in  i\mathbb R$.
%Our aim is to combine the above inequalities with~\eqref{eq:asymptotic_Xt_intermediate_zero} to obtain an upper bound for~$\left|K_{t,\delta}^j\right|$. To do so,
We first apply~\eqref{eq:asymptotic_Xt_intermediate_zero_bis} for $z =is Z_{t,\delta}^j$ and $n = 2$ to bound the first term of the right-hand side of~\eqref{eq:asymptotic_Xt_intermediate_zero}. Then, we use~\eqref{eq:asymptotic_Xt_intermediate_zero_bis} for $z = -s^2\sigma_{t,j,\delta}^2/2N_t$ and $n= 1$ to bound the second term of the right-hand side of~\eqref{eq:asymptotic_Xt_intermediate_zero}. Finally, we bound both terms by their supremums in the age variable, that are finite in view of Lemma~\ref{lemma:preliminaries_asymptotic_Xt}. We obtain the following bound, for all $s\in\mathbb{R}$,
\begin{equation}\label{eq:asymptotic_Xt_intermediate_first}
\left|K_{t,\delta}^j(s)\right| \leq \frac{|s|^3}{6 N_t^{\frac{3}{2}}}\sup_{a\in [0,\infty)} \left\{\mathbb{E}_a\left[\left|N_{\delta} - \mathbb{E}_a[N_\delta]\right|^3\right] \right\}  + \frac{s^4}{8 N_t^2}\sup_{a\in [0,\infty)} \left\{\left(\text{Var}_a[N_{\delta}^2]\right)^2\right\}.
\end{equation}
Moreover, by the branching property and Eq.~\eqref{eq:expression_decomposition_residuals}, for all $s\in\mathbb{R}$, 
$$
\mathbb{E}\left[e^{isX_{t}^{\delta}/\sqrt{N_t}} \, | \, \mathcal{F}_t\right] =
\prod_{j=1}^{N_t}\mathbb{E}\left[e^{isZ_{t,\delta}^j}| \mathcal{F}_t\right].
$$
%where$$Z_{t,\delta}^j = \frac{N_{t,\delta}^j - \mathbb{E}[N_{t,\delta}^j | \mathcal{F}_t]}{\sqrt{N_t}}.$$
Therefore, using this last equality, then the inequality $|\prod_{j=1}^{p}x_j - \prod_{j=1}^{p}y_j| \leq \sum_{i=1}^p |x_j - y_j|$ for $p\in\mathbb{N}^*$, $x \in\mathbb{C}^p$, $y\in\mathbb{C}^p$ verifying $\max_{1\leq i \leq p}\left(\max\left(|x_i|,|y_i|\right)\right) \leq 1$, and finally Eq.~\eqref{eq:asymptotic_Xt_intermediate_first}, yield that there exists $C >0$ such that
\begin{align}
&\left|\mathbb{E}\left[\exp\left( is\frac{X_{t}^{\delta}}{\sqrt{N_t}} \right) \, \bigg| \, \mathcal{F}_t\right] - \exp\left[-\frac{s^2}{2N_t}\sum_{j=1}^{N_t}\ \sigma_{t,j,\delta}^2 \right]\right|\nonumber \\
& \qquad \qquad \qquad \qquad  \qquad \qquad\leq \sum_{j=1}^{N_t} \left|K_{t,\delta}^j(s)\right| \leq C\left[\frac{|s|^3}{6 N_t^{\frac{1}{2}}} + \frac{s^4}{8N_t}\right]. \label{eq:asymptotic_Xt_intermediate_second}
\end{align}
The right-hand side goes to $0$ almost surely as $t$ tends to infinity, using Lemma \ref{lemma:preliminaries_asymptotic_Xt} and the fact that $N_t$ tends to infinity.
In addition, all the statements of Lemma \ref{lemma:preliminaries_asymptotic_Xt} imply that the function $x \mapsto \Var_{x}(N_{\delta})$ is bounded continuous. Then, as $p$ defined in~\eqref{eq:stationary_age_distribution} is the same age distribution as the one introduced \cite[p.41]{athreya_convergence_1976} (denoted $A$), we have by~\cite[Corollary 3]{athreya_convergence_1976}
$$
\frac{s^2}{2N_t}\sum_{j=1}^{N_t}\ \sigma_{t,j,\delta}^2~\overset{a.s.}{\underset{t\longrightarrow+\infty}{\longrightarrow}}~\frac{s^2}{2}\sigma_{X,\delta}^2.
$$
The latter combined with~\eqref{eq:asymptotic_Xt_intermediate_second} through a triangular inequality ends the proof.
%\ju{The above,} combined with \eqref{eq:asymptotic_Xt_intermediate_second} through a triangular inequality, \ju{allows us to obtain~\eqref{eq:variance_asymptotic_Xt}.} 
\end{proof}
\noindent One can easily check that this proposition implies 
$$\frac{X_t^{\delta}}{\sqrt{N_t}} {\underset{t\longrightarrow+\infty}{\Longrightarrow}} \mathcal{N}\left(0, \sigma_{X,\delta}^2\right).$$ The conditional convergence will be useful for the study of $(X_t^{\delta},Y_t^{\delta})_{t\geq0}$.
%with the convergence of $()_{t\geq0}$.
\subsection{Asymptotic behavior of \texorpdfstring{$Y_t^{\delta}$}{Y\_t\^d}} Let us focus now on the asymptotic behavior of $(Y_{t}^{\delta})_{t\geq0}$ for all $\delta >0$. By Equation \eqref{eq:expression_decomposition_residuals}, we see that for all $t\geq0$, $\delta >0$, similar to $X_t^{\delta}$, the random variable $Y_t^{\delta}$ can be seen as the sum of $N_t$ independent random variables. 
%However, the asymptotic behavior of the latter is not necessarily given by the central limit theorem. Indeed, the second eigenvalue linked to the process $(N_t)_{t\geq0}$ $\lambda$ also plays a role, and then we can observe two different regimes for the asymptotic behavior of $(Y_{t}^{\delta})_{t\geq0}$, depending on the value of the spectral gap $\alpha - \lambda$. 
To obtain the behavior of $(Y_t^{\delta})_{t\geq0}$, we use the result in \cite{iksanov_asymptotic_2024} which deals with more general processes. A preliminary result is required to apply it, that we present in the following lemma.

\begin{lemma}\label{lemma:preliminaries_asymptotic_Yt}
Let us assume that $k\geq1$, $\theta >0$. Then for any $\delta >0$, the function~$h_1(a) = \mathbb{E}_a[N_{\delta}]$ is continuously differentiable on $[0,\infty)$. In addition, there exists $C_{h_1} >0$ such that for all~$a\geq0$
\begin{equation}\label{eq:upper_bound_useful_boundedvariation}
|h'_1(a)| \leq C_{h_1}.
\end{equation}
\end{lemma}
%Let function $f : \mathbb{R}_+ \mapsto \mathbb{R}$. Then $f$ is directly Riemann-integrable if for some $\epsilon >0$\end{rem}
\begin{proof}
Let us begin by proving that $h_1$ is continuously differentiable on $\mathbb{R}_+$. To do so, in view of~\eqref{eq:integral_equation} and the fact that $G$ is continuously differentiable, we only have to prove that $a\mapsto I_{\delta}(a)$ is continuously differentiable on $[0,+\infty)$. As for all $y\geq0$ we have
$$
g'(y) =  - \frac{y^{k-1}e^{-\frac{y}{\theta}}}{\Gamma(k)\theta^{k+1}} + \left(k-1\right)\frac{y^{k-2}e^{-\frac{y}{\theta}}}{\Gamma(k)\theta^k},
$$
and as the function $s\mapsto s^{k-2}$ decreases when $k\in (1,2)$, there exists $c > 0$ such that for all~$a\geq 0$, $x\in(0,\delta]$, 
\begin{equation}\label{eq:bound_derivative}
\left|g'(a+x)\right| \leq \begin{cases}
c & \text{ when }k\geq 2 \text{ or }k = 1, \\
c\left(1+x^{k-2}\right) & \text{ when }k\in(1,2).
\end{cases}
\end{equation}
The continuous differentiability of  $a\mapsto I_{\delta}(a)$
%with respect to the first variable
then follows from the following integrable bound, for $a\geq 0$, $x\in(0,\delta]$,  %$a\mapsto\int_{0}^{\delta} \mathbb{E}_0[N_{\delta-x}] g(a+x) dx$ by domination, as for all $\eta >0$, $x\geq0$ and $a\in[\eta +\infty[$
$$
\left|\mathbb{E}_0[N_{\delta-x}] g'(a+x)\right| \leq c\left(1 + 1_{\{k \in (1,2)\}}x^{k-2}\right)\mathbb{E}_0[N_{\delta}].
$$
Let us prove now \eqref{eq:upper_bound_useful_boundedvariation}. By \eqref{eq:integral_equation}, for all $a\geq0$,
$$
\begin{aligned}
h_1'(a) &= 2 \frac{\partial}{\partial a} I_{\delta}(a)    +  \frac{\partial}{\partial a}  \frac{1-G(a+\delta)}{1- G(a)}  \\
&= 2\int_{0}^{\delta}\mathbb{E}_0[N_{\delta-x}]\left(\frac{g'(a+x)}{1-G(a)} + \frac{g(a+x)g(a)}{(1-G(a))^2}\right)dx  + \frac{(1-G(a+\delta))g(a)}{(1-G(a))^2} - \frac{g(a+\delta)}{1-G(a)}.
\end{aligned}
$$
Then, using the fact that $G$ increases, we obtain for all $a\geq 0$,
\begin{equation}\label{eq:preliminaries_Yt_firststatement_intermediatefirst}
\begin{aligned}
\left|h_1'(a)\right| &\leq 2\mathbb{E}_0[N_{\delta}]\int_{0}^{\delta}\left(\left|\frac{g'(a+x)}{1-G(a)}\right| + \frac{g(a+x)g(a)}{(1-G(a+x))(1-G(a))}\right)dx \\
& \qquad+ \frac{g(a)}{1-G(a)} + \frac{g(a+\delta)}{1-G(a+\delta)}.
\end{aligned}
\end{equation}
To continue, we need upper bounds for the two last terms.
%$s\mapsto\frac{g(s)}{1-G(s)}$ and $x\mapsto\sup_{s\geq0}\left|\frac{g'(s+x)}{1-G(s)}\right|$. 
First, by L'Hôpital's rule, 
$$
\lim_{a \longrightarrow +\infty} \frac{g(a)}{1-G(a)} = -\lim_{a \longrightarrow +\infty} \frac{g'(a)}{g(a)} = -\lim_{a \longrightarrow +\infty} (\log(g(a)))' = \frac{1}{\theta}.
$$
The latter combined with the fact that $g$ and $G$ are continuous yields 
\begin{equation}\label{eq:preliminaries_Yt_firststatement_intermediatesecond}
\sup_{s \geq 0}\frac{g(s)}{1-G(s)}< +\infty.
\end{equation}%, we have by using~\eqref{eq:bound_derivative},the fact that $G$ increases, and the fact that $y \mapsto y^{k-2}$ decreases when $k\in(1,2)$, that for all~$a\leq 1$
%$$
%\left|\frac{g'(x)}{1-G(a)}\right| \leq \begin{cases}
%\frac{c}{1-G(1)} & \text{ when }k\geq 2, \\
%c & \text{ when }k\in(1,2) \\
%0 & \text{ when }k = 0.
%\end{cases},
%$$
Second, using that $G$ increases,  for all   $x\in (0,\delta]$  and $a> 1$,
\begin{equation}\label{eq:preliminaries_Yt_firststatement_intermediatesecond_bis}
\left|\frac{g'(a+x)}{1-G(a)}\right| = \left|\frac{k-1}{a+x}\frac{g(a+x)}{1 - G(a)} - \frac{g(a+x)}{\left(1 - G(a)\right)\theta}\right| \leq\left((k-1)+\frac{1}{\theta}\right)\frac{g(a+x)}{1 - G(a+x)},
\end{equation}
and  for all $a\leq 1$,
\begin{equation}\label{eq:preliminaries_Yt_firststatement_intermediatesecond_ter}
\left|\frac{g'(a+x)}{1-G(a)}\right| \leq \left|\frac{g'(a+x)}{1-G(1)}\right|.
\end{equation}
Then, combining~\eqref{eq:preliminaries_Yt_firststatement_intermediatesecond_bis} with~\eqref{eq:preliminaries_Yt_firststatement_intermediatesecond}, and combining~\eqref{eq:preliminaries_Yt_firststatement_intermediatesecond_ter} with~\eqref{eq:bound_derivative}, ensures that there exists $c' >0$ such that for all $x>0$,
$$
\sup_{s\geq0}\left|\frac{g'(s+x)}{1-G(s)}\right| \leq c'\left(1+x^{k-2}1_{\{k\in(1,2)\}}\right).
$$
Plugging this  bound and~\eqref{eq:preliminaries_Yt_firststatement_intermediatesecond} in~\eqref{eq:preliminaries_Yt_firststatement_intermediatefirst} 
%and using $\int_0^{\delta}x^{k-2}  dx < +\infty$ when $k\in(1,2)$ 
yields ~\eqref{eq:upper_bound_useful_boundedvariation}.
\begin{comment}
Finally,  by Lemma \ref{lemma:preliminaries_asymptotic_Xt} $i)$, we easily obtain that there exists a constant $K > 0$ such that for all $a\geq 0$,
$$
h_3(a) = e^{-\alpha a}\left((1 - G(a)) - (1 - G(a))^2\right)\left(\mathbb{E}_a\left[N_{\delta}\right] - e^{\alpha\delta}\right)^2 \leq Ke^{-\alpha a}.
$$
By \cite[Prop. $4.1.$V,$p.154$]{asmussen_applied_2003}, $a\mapsto Ke^{-\alpha a}$ is directly Riemann integrable as this is a non-increasing Lebesgue integrable function. Then, $h_3$ is directly Riemann integrable by \cite[Prop. $4.1.$IV,$p.154$]{asmussen_applied_2003}, as it is a bounded continuous function (Lemma \ref{lemma:preliminaries_asymptotic_Xt}), dominated by $a\mapsto Ke^{-\alpha a}$. 
\end{comment}
\end{proof}

Thus, we have the statement we need to apply the result in \cite{iksanov_asymptotic_2024}. We now introduce some notations linked to this article. We consider for all $t\geq0,\delta>0$ %\ju{Donner un autre nom a $h_2$ pour que cela fasse "fonction importante ?"}
$$
\Delta_t^{\delta} = Y_t^{\delta} -
\sum_{\substack{\rho\in\mathbb{C}\backslash\{-1/\theta\}, \\ \mathcal{L} g(\rho)=\frac{1}{2},  \text{Re}(\rho)>\frac{\alpha}{2} }} e^{\rho\,t} \frac{\int_{\mathbb{R}} h_2(x)e^{-\rho\,x}dx}{2\int_0^{+\infty} xe^{-\rho x}g(x) dx} W(\rho),
$$
where $h_2$ has been defined in Lemma \ref{lemma:preliminaries_asymptotic_Yt}, and $W(\rho)$ is the limit almost surely and in $L^2$ of the martingale $(W_t(\rho))_{t\geq0}$ defined in \cite[Eq. $2.17$]{iksanov_asymptotic_2024}. We do not explicit the latter as this requires introducing too many notations. For all $\delta >0$, we can obtain the asymptotic behavior of $(\Delta_t^{\delta})_{t\geq0}$ by using \cite[Theorem $2.9$]{iksanov_asymptotic_2024}. Precisely, this theorem gives the ergodic behavior for $(\Delta_t^{\delta})_{t\geq0}$, as the cost of a renormalization that is different according to the fact that the set $\{\rho\in\mathbb{C}\backslash\{-1/\theta\} :  \mathcal{L}g(\rho)=1/2,  \text{Re}(\rho)=\alpha/2\}$ is empty or not. The second case corresponds in fact to a critical case, when the spectral gap $\alpha - \lambda = \frac{\alpha}{2}$. Let us give the asymptotic behavior of $(\Delta_t^{\delta})_{t\geq0}$ for all $\delta > 0$ in the following proposition.
\begin{proposition}\label{prop:application_Iksanov} Let $\delta>0$. Then there exists $\sigma_{Y,\delta}^2 \geq 0$ such that 
$$t^{b} e^{-\alpha t/2} \Delta_t^{\delta}\, \left(1,\frac{1}{\sqrt{W}} \right) \, 
\overset{\mathcal{L}}{\underset{t\longrightarrow+\infty}{\Longrightarrow}} \, H\left(\sqrt{\frac{W}{\beta}}, \frac{1}{\sqrt{\beta}}\right),$$ where $W$ is the random variable introduced in~\eqref{equivN}, $H \sim \mathcal{N}(0,\sigma_{Y,\delta}^2)$ is independent of $W$, $\beta = 2\int_0^{\infty}ug(u)e^{-\alpha u}du$, and
$$
b = \begin{cases}
0, & \text{if }\left\{\rho\in\mathbb{C}\backslash\{-1/\theta\} : \mathcal{L}g(\rho) = 1/2,  \text{Re}(\rho)= \alpha/2\right\}= \emptyset, \\
-1/2, & \text{otherwise.}
\end{cases}
$$
%Moreover, when $\left\{\rho\in\mathbb{C}\backslash\{-1/\theta\} : \mathcal{L}g(\rho) = 1/2,  \text{Re}(\rho)\in(\alpha/2,\alpha)\right\} = \emptyset$, %the convergence is stable %(\vinc{En quel sens ? je ne sais pas si il faut garder} \ju{Dans le sens où l'on peut diviser par $\sqrt{W}$ des deux côtés et avoir $t^{b} e^{-\alpha t/2} \, \frac{1}{\sqrt{W}}\Delta_t^{\delta} \, 
%\overset{\mathcal{L}}{\underset{t\longrightarrow+\infty}{\Longrightarrow}} \,  \frac{1}{\sqrt{\beta}}\mathcal{N}(0,\sigma_{Y,\delta}^2)$ (ce que l'on a pas le droit de faire si l'on a pas de stabilité). Cela est primordial afin d'avoir que c'est $\frac{R_{t,\delta}}{\sqrt{N_t}}$ qui converge et non $\frac{R_{t,\delta}}{e^{\alpha/2 t}}$}) 
%and
%\begin{itemize}
%\item[i)]  When $b=0$, we have
%$$
%\sigma_{Y,\delta}^2= \int_{\mathbb{R}_+} \Var\left(\left(\mathbb{E}_x\left[N_{\delta}\right] - e^{\alpha\delta}\right)1_{[0,\zeta[}(x) + 2 h^{(\delta)}(x - \zeta)1_{[\zeta, +\infty[}(x)\right) e^{-\alpha\,x}dx.
%$$ 
%
%\item[ii)] When $b=1/2$, we have
%
%$$
%\sigma_{Y,\delta}^2=\frac{1}{4k^2}\frac{2^{\frac{2}{k}}}{2^{\frac{2}{k}}+1-2^{\frac{1}{k}+1}\cos\left(\frac{2\pi}{k}\right)}\left|e^{(\lambda+i\tau)\delta} - e^{\alpha\delta}\right|^2.
%$$
%\end{itemize}
\end{proposition}
%\begin{rem}
%In the case where $\left\{\rho\in\mathbb{C}\backslash\{-1/\theta\} : \mathcal{L}g(\rho) = 1/2,  \text{Re}(\rho)\in(\alpha/2,\alpha)\right\}\neq\varnothing$, the expression of $\sigma_{Y,\delta}^2$ is slightly more complicated than the expression given in Proposition \ref{prop:application_Iksanov}, and can be obtained with \cite[Eq. $2.25$]{iksanov_asymptotic_2024}. We do not give it here, as it is not required to obtain Theorem \ref{thm:main_result}.
%\end{rem}
\begin{proof} Let $\delta >0$. Bellman-Harris processes are a special class of Crump-Mode-Jagers processes where the reproduction happens only at the end of the life of each individual. We can thus apply Theorem $2.9$ in \cite{iksanov_asymptotic_2024}, 
%\vinc{La il faut etre precis sur la ref, en mettant le lien internet du papier arxiv par exemple, car en fonction des versions ce n'est pas la mm chose}\ju{J'ai mis le lien vers la V2 qui est la version que l'on utilise (la derniere)} 
which deals with Crump-Mode-Jagers processes with random characteristic, by checking the assumptions of this result, namely $(A1-A3)$, $(A5-A6)$ and $(2.19)$ (the density of the intensity measure with respect to the Lebesgue measure is immediate). In our case, the reproduction point process of the Crump-Mode-Jagers process is $\xi = \sum_{j = 1}^2 \delta_{\zeta}$, and the  random characteristic is for all~$x\in\mathbb{R}$
\begin{equation}\label{eq:random_characteristic}
\varphi(x) = \left(\mathbb{E}_x\left[N_{\delta}\right] - e^{\alpha\delta}\right)1_{[0,\zeta[}(x),
\end{equation}
where $\zeta \sim \Gamma(k,\theta).$ %We first verify that the assumptions $(A1-A6)$ of \cite[Theorem $\ju{2.9}$]{iksanov_asymptotic_2024} are satisfied.
As $\mathcal{L}g(\alpha) = \frac{1}{2}$ and $-\left(\mathcal{L}g\right)'(\alpha) = \int_0^{+\infty}xe^{-\alpha x} g(x)dx > 0$, Assumption~$(A1)$ of \cite{iksanov_asymptotic_2024} is easily satisfied.
In addition, the reproduction law is reduced to binary division, so Assumptions $(A2)$ and $(A3)$ of \cite{iksanov_asymptotic_2024} are directly satisfied (see Remark $2.1$ of~\cite{iksanov_asymptotic_2024}). By Lemma~\ref{lemma:preliminaries_asymptotic_Xt}~$i)$, we also have that for all $x\in \R$,
$$
\left|\varphi(x)\right| \leq \left(\sup_{s\geq0}\mathbb{E}_s\left[N_{\delta}\right] + e^{\alpha\delta}\right)< +\infty,
$$
%    Then, for $K > \left(\sup_{s\geq0}\mathbb{E}_s\left[N_{\delta}\right] + e^{\alpha\delta}\right)^2$ we have for all $x\in[t- \epsilon, t + \epsilon]$
%$$\mathbb{E}\left[|\varphi^{(\delta)}(x)|^21_{\{|\varphi^{(\delta)}(x)|^2 \geq K\}}\right] = 0,  \\
%$$which imply that $\sup_{x\in[t- \epsilon, t + %\epsilon]}\mathbb{E}\left[|\varphi^{(\delta)}(x)|^21_{\{|\varphi^{(\delta)}(x)|^2 \geq K\}}\right] = 0$. 
implying that the family $\varphi^2$ is locally bounded, and that $(A6)$ is verified. It thus remains to check $(A5)$ and $(2.19)$, which requires more computations, and the assumptions of \cite[Theorem $2.9$]{iksanov_asymptotic_2024} will be satisfied.\\
%\vinc{Un mot sur A4?} \ju{J'ai fait une erreur de frappe en recopiant. On a pas besoin de prouver $(A4)$ c'est pour ca que je l'ai pas mise. Elle est impliquee par $2.19$ (voir deuxieme phrase de la preuve du theoreme $\ju{2.9}$ d'Iksanov et al ou preuve theoreme $2.7$ pour plus de detail)}
%Since we use random variables distributed according to a density $g$ to obtain particle lifetimes, we have a density of the intensity measure with respect to the intensity measure.
To check~$(A5)$, we need to prove that the function $h_2 : x \mapsto e^{-\alpha x}\Var(\varphi(x)) $ is directly Riemann-integrable on $\mathbb{R}$, i.e. that 
\begin{equation}\label{eq:directly_riemann_integrable}
\sum_{n \in \mathbb{Z}} \underset{y\in(n\epsilon,(n+1)\epsilon]}{\sup} h_2(y) < +\infty, \quad \text{ and } \quad 
\epsilon\sum_{n \in \mathbb{Z}} \big\{ \underset{(n\epsilon,(n+1)\epsilon]}{\sup} h_2 - 
%\epsilon\sum_{n \geq 0} 
\underset{(n\epsilon,(n+1)\epsilon]}{\inf} h_2 \big\} \underset{\epsilon \longrightarrow 0}{\longrightarrow} 0.
\end{equation}
By~\eqref{eq:random_characteristic} and Lemma \ref{lemma:preliminaries_asymptotic_Xt} $i)$, we easily have that there exists a constant $K > 0$ such that for all $a\geq 0$,
$$
h_2(a) = e^{-\alpha a}\left((1 - G(a)) - (1 - G(a))^2\right)\left(\mathbb{E}_a\left[N_{\delta}\right] - e^{\alpha\delta}\right)^2 \leq Ke^{-\alpha a}.
$$
In addition, by \cite[Prop. $4.1.$V,$p.154$]{asmussen_applied_2003}, the function $a\mapsto Ke^{-\alpha a}$ is directly Riemann integrable as this is a non-increasing Lebesgue integrable function. Then, $h_2$ is directly Riemann integrable on $[0,+\infty)$ by \cite[Prop. $4.1.$IV,$p.154$]{asmussen_applied_2003}, because it is a bounded continuous function (Lemma \ref{lemma:preliminaries_asymptotic_Xt}) dominated by $a\mapsto Ke^{-\alpha a}$. As $h_2$ is null on $(-\infty,0)$, this implies that~\eqref{eq:directly_riemann_integrable} holds. \\
%\vinc{J'ai simplifi\'e ici, \`a v\'erifier.Il faudrait preciser aussi la notation $\varphi^{(\delta)}$, qu'est $\delta$ ? les auteurs regardent $\varphi^{2}$ pour A6 d'ailleurs, a changer ? } \ju{Je ne comprends pas les remarques.}
Let us denote $h_3(a) = \mathbb{E}\left[\varphi(a)\right] = \left(\mathbb{E}_{a}\left[N_{\delta}\right] - e^{\alpha\delta}\right)\left( 1 - G(a)\right)1_{\{a\geq 0 \}}$ for all $a\in\mathbb{R}$. To check~Equation $(2.19)$ in \cite{iksanov_asymptotic_2024}, we need to prove that for some $c < \frac{\alpha}{2}$
\begin{equation}\label{eq:integral_totvar}
\int_{\mathbb{R}}\totvar h_3(x)\left(e^{-cx} + e^{-\alpha x}\right)d x < +\infty,
\end{equation}
where 
$$
\totvar h_3(x) = \sup\left\{\sum_{j = 1}^n |h_3(x_j) - h_3(x_{j-1})|\,\big|\,-\infty < x_0 < x_1 < \hdots < x_n \leq x,\,n\in\mathbb{N}\right\}.
$$
The function $h_3$ is continuously differentiable on $(0,+\infty)$ by Lemma~\ref{lemma:preliminaries_asymptotic_Yt}, and with all derivatives equal to zero on $(-\infty,0)$. Moreover, we easily have that $|h_3(0^-) - h_3(0)| = h_3(0)$, and by continuity $\left|h_3(0^+) - h_3(0)\right| = 0$ (see Lemma \ref{lemma:preliminaries_asymptotic_Xt} $ii)$). Therefore, using the expression of the total variation for a piecewise continuously differentiable function yields for all $a\in\mathbb{R}$
\begin{equation}\label{eq:preliminaries_Yt_firststatement_intermediatefourth}
\begin{aligned}
\totvar h_3 (a) = \begin{cases}
0, &\text{ for }a < 0, \\
|h_3(0)| + \int_0^a \left|h'_3(s)\right| ds, &\text{ for }a \geq  0.
\end{cases} 
\end{aligned}
\end{equation}
For all $s> 0$, we have using a triangular inequality, and then Lemmas \ref{prop:asymptotic_Xt} and \ref{lemma:preliminaries_asymptotic_Yt}, that
$$
\begin{aligned}
|h'_3(s)| &= \left|h'_1(s)(1-G(s)) -g(s)\left(\mathbb{E}_{s}\left[N_{\delta}\right] - e^{\alpha\delta}\right)\right| \\
&\leq C_{h_1}\left(1-G(s)\right) + g(s)\left(\sup_{y\geq 0}\left(\mathbb{E}_{y}\left[N_{\delta}\right]\right) + e^{\alpha\delta}\right). 
\end{aligned}
$$
Then, plugging this in \eqref{eq:preliminaries_Yt_firststatement_intermediatefourth}, and using that $$\int_0^{+\infty} (1-G(s)) ds = \int_0^{+\infty} u g(u) du < \infty\text{ and }\int_0^{+\infty} g(s) ds = 1,$$ yields that $\totvar h_3$ is bounded on $\mathbb{R}_+$, and that~\eqref{eq:integral_totvar} holds. \\
We have now checked the assumptions of 
\cite[Theorem $2.9$]{iksanov_asymptotic_2024}. Notice that for all $p\in\mathbb{C}$ such that $Re(p)  \neq \frac{1}{\theta}$, we have  
\begin{equation}\label{eq:expression_derivative_laplace}
(\mathcal{L}g)'(p) = -\frac{k\theta}{(1+p\theta)^{k+1}} \neq 0.
\end{equation}
Then, every root of $p \mapsto \mathcal{L}g(p)-\frac{1}{2}$ is a root of multiplicity $1$, and $t^{b} e^{-\alpha t/2} \Delta_t^{\delta}$ corresponds exactly to the left-hand side term of~\cite[Theorem $2.9$ $i)$]{iksanov_asymptotic_2024}. We thus get from this theorem the convergence in law of $t^{b} e^{-\alpha t/2} \Delta_t^{\delta}$. %Applying this result ends the proof of the convergence in law. 
%Combining \vinc{a clarifier ici pour moi, comment on conclue sur la convergence en loi, (2.19) a t il ete verifie ?}\ju{Equation 2.19 a été satisfaite c'est le dernier point. On conclut sur la convergence en loi car toutes les hypothèses de Iksanovetal ont été satisfaites, donc par le theoreme on a la convergence en loi} the latter with \cite[Theorem $\ju{2.9}$]{iksanov_asymptotic_2024} yields that the convergence in distribution holds.
Adding that 
this convergence is stable by~\hbox{\cite[Remark~$2.14$]{iksanov_asymptotic_2024}}, and that $W>0$ a.s. by \cite[Theo.~$2$~,~$p.172$]{athreya_convergence_1976}, ends the proof of the convergence  in law of the bivariate random variable. 
\end{proof}
\noindent 
The previous proposition  shows that the asymptotic behavior of $(Y_t^{\delta})_{t\geq0}$ depends on  the set $\{\rho\in\mathbb{C}\backslash\{-1/\theta\} : \mathcal{L}g(\rho)=1/2\}$. In our case, 
with  Gamma distribution of lifespan, this set is explicit. Indeed,  for all $p\in\mathbb{C}$ such that $p \neq -\frac{1}{\theta}$, we have 
\begin{equation}\label{eq:laplace_transform_gamma}
\mathcal{L}g(p)=\frac{1}{(1+p \theta)^k}.
\end{equation}
We can thus solve the equation verified by the elements of this set, in view of~\eqref{eq:power_complex_number} and the fact that $\arg(\rho)\in(-\pi,\pi]$ for all $\rho\in\mathbb{C}$, and obtain %\ju{les floor et ceil sont moches, meme avec /left et /right. A voir comment avoir une belle expression...}
\begin{equation}\label{eq:expression_eigenvalues_gamma}
\begin{aligned}
\left\{\rho\in\mathbb{C}\backslash\{-1/\theta\} : \mathcal{L}g(\rho) = \frac{1}{2} \right\} &=  \left\{\rho\in\mathbb{C}\backslash\{-1/\theta\} : \left|1+\rho \theta\right|^k = 2,\,k\arg(\rho) \in 2\pi \mathbb{Z} \right\} \\ 
&= \left\{\frac{2^{\frac{1}{k}}\exp\left(\frac{2\pi\,l}{k}i\right) - 1}{\theta}\,|\,l\in\left\llbracket - \left\lceil\frac{k}{2}\right\rceil + 1, \left\lfloor \frac{k}{2} \right\rfloor\right\rrbracket\right\}.
\end{aligned}
\end{equation}
Recalling the expressions of $\alpha, \lambda, \tau$ from \eqref{expcst}, we derive 
%We notice with this expression 
that when $k\geq 2$,
$$\alpha,\,\lambda \pm i\tau \in \, \left\{\rho\in\mathbb{C}\backslash\{-1/\theta\} : \mathcal{L}g(\rho) = 1/2\right\},$$
and when $k\in [1,2)$
\begin{equation}\label{eq:set_eigenvalues_[1,2)}
\left\{\rho\in\mathbb{C}\backslash\{-1/\theta\} : \mathcal{L}g(\rho) = 1/2\right\} = \{\alpha\}. 
\end{equation}
%and $\lambda \pm i\tau \in\left\{\rho\in\mathbb{C}\backslash\{-1/\theta\} : \mathcal{L}g(\rho) = 1/2 \right\}$. 
We can now classify  the convergence of $(Y_t^{\delta})_{t\geq0}$ and explain the cases of Theorem \ref{thm:main_result}. 
%We give these cases in the following proposition. 
\begin{proposition}\label{prop:asymptotic_Yt}
Let $\delta >0$. The following result of convergence holds.
\begin{enumerate}[$i)$]
\item If $k < k_c$, then
$$
\frac{Y_{t}^{\delta}}{\sqrt{N_t}} \overset{\mathcal{L}}{\underset{t\longrightarrow+\infty}{\Longrightarrow}} \mathcal{N}\left(0, 2\alpha\sigma_{Y,\delta}^2 \right),
$$
where 
$$
\sigma_{Y,\delta}^2 = \int_{\mathbb{R}_+} \emph{Var}\left(\left(\mathbb{E}_x\left[N_{\delta}\right] - e^{\alpha\delta}\right)1_{[0,\zeta[}(x) + 2 h^{(\delta)}(x - \zeta)1_{[\zeta, +\infty[}(x)\right) e^{-\alpha\,x}dx < +\infty.
$$
\item If $k = k_c$, then
$$
\frac{Y_{t}^{\delta}}{\sqrt{tN_t}} \overset{\mathcal{L}}{\underset{t\longrightarrow+\infty}{\Longrightarrow}} \mathcal{N}\left(0, 2\alpha\sigma_{Y,\delta}^2 \right),
$$
where
$$
\sigma_{Y,\delta}^2 =  \frac{1}{2k^2}\frac{2^{\frac{2}{k}}}{2^{\frac{2}{k}}-2^{\frac{1}{k}}}\left|e^{(\lambda+i\tau)\delta} - e^{\alpha\delta}\right|^2.
$$
\item If $k > k_c$, then
$$
\left| \frac{Y_t^{\delta}}{\exp\left(\lambda t\right)} -  2\left|M_{\delta} \right| \cos\left[\tau t +  \arg\left(M_{\delta}\right)\right]\right|  \overset{\mathbb{P}}{\underset{t\longrightarrow+\infty}{\longrightarrow}} 0,
$$
where  
$$M_{\delta}= \left(e^{(\lambda + i\tau)\delta}-e^{\alpha \delta}\right)M,$$
and $M$ is complex random variable that does not depend on $\delta$.
\end{enumerate}
\end{proposition}
\begin{proof}

We prove $i)$. First, by \eqref{eq:expression_eigenvalues_gamma}, we notice that when $k\geq 2$,   
    $$\lambda = \max\left\{Re(\rho) : \rho \in \mathbb C\backslash\{-1/\theta,\alpha\}, \, \mathcal{L}g(\rho) = 1/2\right\}.$$
Then, the latter and \eqref{eq:set_eigenvalues_[1,2)} imply that in the case where $k  < k_c$, which is equivalent to $\lambda <\alpha/2$ or $k\in[1,2)$, we have
\begin{equation}\label{eq:info_eigenvalues_first_case}
\begin{aligned}
\left\{\rho\in\mathbb{C}\backslash\{-1/\theta\} : \mathcal{L}g(\rho) = \frac{1}{2}, Re(\rho) \geq \frac{\alpha}{2} \right\} &= \{\alpha\}.
%, \quad \left\{\rho\in\mathbb{C}\backslash\{-1/\theta\} : \mathcal{L}g(\rho) = \frac{1}{2}, Re(\rho) = \frac{\alpha}{2} \right\} = \varnothing. 
\end{aligned}
\end{equation}
Second, following the proof of Proposition~\ref{prop:application_Iksanov} allows to check the Assumptions of \cite[Lemma~$7.6$]{iksanov_asymptotic_2024} for the random characteristics $\varphi_1(x) = 1_{[0,\zeta[}(x)$ and $\varphi_2(x) = \mathbb{E}_x[N_{\delta}]1_{[0,\zeta[}(x)$, where $\zeta \sim \Gamma(k,\theta)$. This result  ensures that there exists $c \in(0,\frac{\alpha}{2})$ such that
$$
\mathbb{E}\left[N_{t+\delta}\right] = \sum_{\substack{\rho\in\mathbb{C}\backslash\{-1/\theta\}, \\ \mathcal{L} g(\rho)=\frac{1}{2},  \text{Re}(\rho)\geq \frac{\alpha}{2} }} \frac{\int_{0}^{+\infty} e^{-\rho\,x}(1 - G(x)) dx}{2\int_{0}^{+\infty} xe^{-\rho\,x}g(x) dx}e^{\rho(t+\delta)} + O\left(e^{c\,t}\right), 
$$
and
$$
\begin{aligned}
\mathbb{E}\left[N_{t+\delta}\right] &= \mathbb{E}\left[\sum_{j = 1}^{N_t}\mathbb{E}[N_{t,\delta}^j|\mathcal{F}_t]\right] \\
&= \sum_{\substack{\rho\in\mathbb{C}\backslash\{-1/\theta\}, \\ \mathcal{L} g(\rho)=\frac{1}{2},  \text{Re}(\rho)\geq \frac{\alpha}{2} }} \frac{\int_{0}^{+\infty} e^{-\rho\,x}\mathbb{E}_x\left[N_{\delta}\right](1 - G(x)) dx}{2\int_{0}^{+\infty} xe^{-\rho\,x}g(x) dx}e^{\rho\,t} + O\left(e^{c\,t}\right).
\end{aligned}
$$
By identifying the coefficients of these two expansions, we get
%\ju{pareil, plus de details ou pas ?} allows to conclude that 
\begin{equation}\label{eq:coefficient_expectation_equality}
\begin{aligned}
\int_{0}^{+\infty} \mathbb{E}_x\left[N_{\delta}\right]e^{-\alpha\,x}(1 - G(x)) dx &= e^{\alpha\delta}\int_{0}^{+\infty} e^{-\alpha\,x}(1 - G(x)) dx,
\end{aligned}
\end{equation}
and finally
\begin{equation}\label{eq:coefficient_expectation_equality_consequence}
\int_0^{+\infty}h_2(x)e^{-\alpha x}dx = 0.
\end{equation}
Combining this equation with \eqref{eq:info_eigenvalues_first_case} implies that for all $t\geq0$, $\delta >0$,
\begin{equation}\label{eq:equality_Delta_Y}
\Delta_{t}^{\delta} = Y_{t}^{\delta}.
\end{equation}
Combining now Equations \eqref{eq:info_eigenvalues_first_case} and \eqref{eq:equality_Delta_Y}, and Proposition \ref{prop:application_Iksanov}, we get $$e^{-\alpha t/2} \, {W}^{-1/2} \, Y_t^{\delta} \, 
\overset{\mathcal{L}}{\underset{t\longrightarrow+\infty}{\Longrightarrow}} \,  \frac{1}{\sqrt{\beta}}\mathcal{N}(0,\sigma_{Y,\delta}^2).$$ We also know that 
$$\sqrt{N_t} \, {W}^{-1/2} \, e^{-\alpha t/2} \overset{a.s.}{\underset{t\longrightarrow+\infty}{\longrightarrow}} \sqrt{n_1},$$ where $n_1 = \frac{1}{2\alpha\beta}$,  see \cite[Theo. $21.1$, $p.147$]{Harris_1963}. Combining these two results of convergence through Slutsky's lemma allows to obtain the convergence. We mention that we can also obtain this convergence by applying \cite[Corollary 2]{kang_central_1999}, but this only works when the parameter $k$ of the Gamma distribution is an integer.  \\ 
To know the value of $\sigma_{Y,\delta}^2$ when $k < k_c$, we use the expression given in \cite[Theorem~$2.15$]{iksanov_asymptotic_2024}. We are in the case where $n = -1$ (as  $p \mapsto \mathcal{L}g(p)-\frac{1}{2}$ has no roots such that \hbox{$Re(p) = \alpha/2$}). Then, if we consider $\zeta \sim \Gamma(k,\theta)$, and if we denote for all $y\in\mathbb{R}$ the function \hbox{$h^{(\delta)}(y) = (\mathbb{E}\left[N_{y+\delta}\right] - e^{\alpha\delta}\mathbb{E}\left[N_y\right] )1_{[0,\infty[}(y)$}, we have by \cite[Theorem $2.15$]{iksanov_asymptotic_2024} that $\sigma_{Y,\delta}^2$ is finite, and has the following expression
\begin{equation}\label{eq:expression_variance_fluctuations_first}
\begin{aligned}
\sigma_{Y,\delta}^2 &= \int_{\mathbb{R}} \Var(\varphi(x) + 2 h^{(\delta)}(x - \zeta)1_{[\zeta, +\infty[}(x)) e^{-\alpha\,x}dx    \\
&= \int_{\mathbb{R}_+} \Var\left(\left(\mathbb{E}_x\left[N_{\delta}\right] - e^{\alpha\delta}\right)1_{[0,\zeta[}(x) + 2 h^{(\delta)}(x - \zeta)1_{[\zeta, +\infty[}(x)\right) e^{-\alpha\,x}dx.
\end{aligned}
\end{equation}

We prove now $ii)$. As $\lambda = \max\left\{Re(\rho)\,|\,\mathcal{L}g(\rho) = \frac{1}{2}\right\}$, we have by the condition of the statement and \eqref{eq:expression_eigenvalues_gamma} that \begin{equation}\label{eq:info_eigenvalues_second_case}
    \begin{aligned}
    \left\{\rho\in\mathbb{C}\backslash\{-1/\theta\} : \mathcal{L}g(\rho) = \frac{1}{2}, \, Re(\rho) > \frac{\alpha}{2} \right\} &= \{\alpha\},\\
    \left\{\rho\in\mathbb{C}\backslash\{-1/\theta\} : \mathcal{L}g(\rho) = \frac{1}{2}, \, Re(\rho) = \frac{\alpha}{2} \right\} &= \left\{\lambda - i\tau,\lambda + i\tau\right\}. 
    \end{aligned}
    \end{equation}
We prove similarly to the previous point that \eqref{eq:coefficient_expectation_equality} is satisfied, implying that for all $t\geq0$, $\delta >0$, we have $\Delta_{t}^{\delta} = Y_{t}^{\delta}$. Then, by Proposition \ref{prop:application_Iksanov} $ii)$, 
$$t^{-1/2}e^{-\alpha t/2}\left(\sqrt{W}\right)^{-1}Y_t^{\delta} \, 
\overset{\mathcal{L}}{\underset{t\longrightarrow+\infty}{\Longrightarrow}} \,  \frac{1}{\sqrt{\beta}}\mathcal{N}(0,\sigma_{Y,\delta}^2).$$
Using  \eqref{equivN} and Slutsky's lemma, 
we can switch from $W$ to $N_t$ in this convergence and get~$ii).$ \\
%Then, we conclude easily by the fact that $\sqrt{N_t}\left(\sqrt{W}\right)^{-1}e^{-\alpha t/2} \overset{a.s.}{\underset{t\longrightarrow+\infty}{\longrightarrow}} \sqrt{n_1}$ and Slutsky's lemma.
We now compute $\sigma_{Y,\delta}^2$ when $k = k_c$, i.e. when $2\cos\left(2\pi/k\right) = 2^{-1/k} + 1$ and $k\geq 2$. It is direct that if we consider $\zeta \sim \Gamma(k,\theta)$ we have by the expression of the variance given in~\cite[Theorem $2.15$]{iksanov_asymptotic_2024} (we are in the case where $n = 0$, as the roots of $p \mapsto \mathcal{L}g(p)-\frac{1}{2}$ such that $Re(p) = \alpha/2$ have a multiplicity $1$) 
$$
\sigma_{Y,\delta}^2=  \sum_{\substack{\rho\in\mathbb{C}\backslash\{-1/\theta\},\\ \mathcal{L}g(\rho) = \frac{1}{2},\,Re(\rho) = \frac{\alpha}{2}}}\Var\left[\left(\frac{e^{\rho\delta} - e^{\alpha\delta}}{2\rho\int_{0}^{+\infty} xe^{-\rho\,x}g(x) dx}\right) e^{-\rho\,\zeta}\right].
$$ 
As for a complex random variable $Z$ and a constant $c\in\mathbb{C}$ we have $\Var(cZ) = |c|^2\Var(Z)$, and as $\left\{\rho\in\mathbb{C}\backslash\{-1/\theta\},\,2.\mathcal{L}g(\rho) = 1,\,Re(\rho) = \frac{\alpha}{2}\right\} = \{\lambda - i\tau,\lambda + i\tau\}$, we obtain 
$$
\begin{aligned}
\sigma_{Y,\delta}^2 &= \sum_{s \in \{-1,1\}}\left|\frac{e^{(\lambda + is\tau)\delta} - e^{\alpha\delta}}{2(\lambda + is\tau)\int_{0}^{+\infty} xe^{-(\lambda + is\tau)\,x}g(x) dx}\right|^2\Var\left[e^{-(\lambda + is\tau)\zeta}\right].
\end{aligned}
$$
As $\mathbb{E}\left[e^{-\alpha\zeta\,}\right] = \mathcal{L}g(\alpha) = 1/2$, as for the same reason $\mathbb{E}\left[e^{-(\lambda + i\tau)\zeta\,}\right] = \mathbb{E}\left[e^{-(\lambda - i\tau)\zeta\,}\right] = 1/2$, and as $\alpha = 2\lambda$, we have $\Var\left[e^{-(\lambda \pm i\tau)\zeta}\right] = \mathbb{E}\left[e^{-2\lambda\zeta}\right] - \left|\mathbb{E}\left[e^{-(\lambda\pm i\tau)\zeta}\right]\right|^2 = 1/2 - 1/4 = 1/4$. Plugging this in the above equation yields
$$
\begin{aligned}
\sigma_{Y,\delta}^2 &= \frac{1}{4}\left|\frac{e^{(\lambda + i\tau)\delta} - e^{\alpha\delta}}{2(\lambda + i\tau)\int_{0}^{+\infty} xe^{-(\lambda + i\tau)\,x}g(x) dx}\right|^2 + \frac{1}{4}\left|\frac{e^{(\lambda - i\tau)\delta} - e^{\alpha\delta}}{2(\lambda - i\tau)\int_{0}^{+\infty} xe^{-(\lambda - i\tau)\,x}g(x) dx}\right|^2.
\end{aligned}
$$
Using the fact that $\int_{0}^{+\infty} xe^{-(\lambda \pm i\tau)\,x}g(x) dx = -\mathcal{L}'(g)(\lambda \pm i \lambda)$, Equation \eqref{eq:expression_derivative_laplace}, the fact that $1/[1+ (\lambda \pm i\tau)\theta]^k= \mathcal{L}g(\lambda \pm i\tau) = \frac{1}{2}$, and finally the equality $2\cos\left(\frac{2\pi}{k}\right)  =2^{-\frac{1}{k}} + 1$, we get
$$
\begin{aligned}
\sigma_{Y,\delta}^2& = \frac{1}{4}\left|\frac{1}{k}\frac{2^{\frac{1}{k}}\exp\left(i\frac{2\pi}{k}\right)}{2^{\frac{1}{k}}\exp\left(i\frac{2\pi}{k}\right)-1}\right|^2
\left[\left|\left(e^{(\lambda+i\tau)\delta} - e^{\alpha\delta}\right)\right|^2 + \left|e^{(\lambda-i\tau)\delta} - e^{\alpha\delta}\right|^2\right]\\
%\end{aligned}
%$$
%and then, as the two terms we sum are conjugates, we obtain 
%$$
%\begin{aligned}
%\sigma_{\delta}^2 &= \alpha\left|\frac{2^{\frac{1}{k}+1}\exp\left(i\frac{2\pi}{k}\right)}{2^{\frac{1}{k}}\exp\left(i\frac{2\pi}{k}\right)-1}\left(e^{(\lambda+i\tau)\delta} - e^{\alpha\delta}\right)\right|^2 \\
&= \frac{1}{2k^2}\frac{2^{\frac{2}{k}}}{2^{\frac{2}{k}}+1-2^{\frac{1}{k}+1}\cos\left(\frac{2\pi}{k}\right)}\left|e^{(\lambda+i\tau)\delta} - e^{\alpha\delta}\right|^2 \\
&= \frac{1}{2k^2}\frac{2^{\frac{2}{k}}}{2^{\frac{2}{k}}-2^{\frac{1}{k}}}\left|e^{(\lambda+i\tau)\delta} - e^{\alpha\delta}\right|^2.
\end{aligned}
$$

We finally prove $iii)$. In this case,
%The condition of the statement imply that 
$$\lambda = \max\left\{Re(\rho) : \rho \in \mathbb C\backslash\{-1/\theta,\alpha\}, \, \mathcal{L}g(\rho) = 1/2\right\} \in (\alpha/2, \alpha).$$
Moreover, one can prove as in the first point that \eqref{eq:coefficient_expectation_equality} and \eqref{eq:coefficient_expectation_equality_consequence} are true.
Then, we have the following equality for all $t\geq0$, $\delta >0$,

\begin{equation}\label{eq:cor_third_case_intermediate_first}
\begin{aligned}
&\frac{Y_t^{\delta}}{\exp\left(\lambda t\right)} - \sum_{\substack{\rho\in\mathbb{C}\backslash\{-1/\theta\}, \\ \mathcal{L} g(\rho)=\frac{1}{2},  \text{Re}(\rho) = \lambda}} e^{(\rho-\lambda)t} \frac{\int_{\mathbb{R}} h_2(x)e^{-\rho\,x}dx}{2\int_0^{+\infty} xe^{-\rho x}g(x) dx} W(\rho) \\ 
&= \frac{\Delta_t^{\delta}}{\exp\left(\lambda t\right)} + \sum_{\substack{\rho\in\mathbb{C}\backslash\{-1/\theta\}, \\ \mathcal{L} g(\rho)=\frac{1}{2},  \lambda>\text{Re}(\rho)>\alpha/2 }} e^{(\rho-\lambda)t} \frac{\int_{\mathbb{R}} h_2(x)e^{-\rho\,x}dx}{2\int_0^{+\infty} xe^{-\rho x}g(x) dx} W(\rho).
\end{aligned}
\end{equation}
As $\lambda > \alpha/2$, by Proposition \ref{prop:application_Iksanov} (whatever the case) %\vinc{ca donne un peu l'impression ce debut de preuve qu'on pourrait mutualiser une bonne partie d la preuve des cas i) et iii), puis prouver ii)})
and Slutsky's lemma we have  
\begin{equation}\label{eq:cor_third_case_intermediate_second}
\frac{\Delta_t^{\delta}}{\exp\left(\lambda t\right)} \overset{\mathcal{L}}{\underset{t\longrightarrow+\infty}{\Longrightarrow}} 0.
\end{equation}
We also easily have that 
\begin{equation}\label{eq:cor_third_case_intermediate_third}
\sum_{\substack{\rho\in\mathbb{C}\backslash\{-1/\theta\}, \\ \mathcal{L} g(\rho)=\frac{1}{2},  \lambda>\text{Re}(\rho)>\alpha/2 }} e^{(\rho-\lambda)t} \frac{\int_{\mathbb{R}} h_2(x)e^{-\rho\,x}dx}{2\int_0^{+\infty} xe^{-\rho x}g(x) dx} W(\rho) \overset{a.s.}{\underset{t\longrightarrow+\infty}{\longrightarrow}} 0.
\end{equation}
Combining \eqref{eq:cor_third_case_intermediate_second} and \eqref{eq:cor_third_case_intermediate_third} through Slutsky's lemma yields
\begin{equation}\label{eq:cor_third_case_intermediate_fourth}
\frac{Y_t^{\delta}}{\exp\left(\lambda t\right)} - \sum_{\substack{\rho\in\mathbb{C}\backslash\{-1/\theta\}, \\ \mathcal{L} g(\rho)=\frac{1}{2},  \text{Re}(\rho) = \lambda}} e^{(\rho-\lambda)t} \frac{\int_{\mathbb{R}} h_2(x)e^{-\rho\,x}dx}{2\int_0^{+\infty} xe^{-\rho x}g(x) dx} W(\rho) \overset{\mathcal{L}}{\underset{t\longrightarrow+\infty}{\Longrightarrow}} 0.
\end{equation}
We know by \eqref{eq:expression_eigenvalues_gamma} that $\{\rho\in\mathbb{C}\backslash\{-1/\theta\} : \mathcal{L} g(\rho)=\frac{1}{2},  \text{Re}(\rho) = \lambda\} = \left\{\lambda - i\tau, \lambda + i\tau\right\}$. Proceeding as in the first point for \eqref{eq:coefficient_expectation_equality}, we can obtain that 
%\ju{plus d'explications ?} \vinc{je trouve que ca suffit}
$$
\begin{aligned}
\int_{0}^{+\infty} \mathbb{E}_x\left[N_{\delta}\right]e^{-(\lambda \pm i \tau)x}(1 - G(x)) dx &= e^{(\lambda\pm i\tau)\delta}\int_{0}^{+\infty} e^{-(\lambda \pm i \tau)x}(1 - G(x)) dx \\
&= \frac{e^{(\lambda\pm i\tau)\delta}}{\lambda\pm i\tau}\int_{0}^{+\infty} \left(1 - e^{-(\lambda\pm i\tau)y}\right) g(y) dy = \frac{e^{(\lambda\pm i\tau)\delta}}{2(\lambda\pm i\tau)},
\end{aligned}
$$
and more generally
\begin{equation}\label{eq:coefficient_oscillations_first}
\int_{\mathbb{R}} h_2(x)e^{-(\lambda+ i\tau)x}dx = \frac{e^{(\lambda+ i\tau)\delta}-e^{\alpha\delta}}{2(\lambda+ i\tau)} = \overline{\int_{\mathbb{R}} h_2(x)e^{-(\lambda- i\tau)x}dx}.
\end{equation}
By \eqref{eq:expression_derivative_laplace}, we also have that
\begin{equation}\label{eq:coefficient_oscillations_second}
\int_0^{+\infty} xe^{-(\lambda +i\tau)x}g(x) dx = \overline{\int_0^{+\infty} xe^{-(\lambda -i\tau)x}g(x) dx} = -(\mathcal{L}g)'(\lambda + i \tau) = \frac{k\theta}{2^{1+\frac{1}{k}}\exp\left(\frac{2\pi}{k}i\right)}.
\end{equation}
\noindent As $W(\lambda+i\tau)$ and $W(\lambda-i\tau)$ are limits of conjugated martingales (see \cite[Eq. $2.17$]{iksanov_asymptotic_2024} for the expression of the martingale), we have that 
\begin{equation}\label{eq:coefficient_oscillations_third}
W(\lambda+i\tau) = \overline{W(\lambda-i\tau)}.
\end{equation}
Plugging \eqref{eq:coefficient_oscillations_first}, \eqref{eq:coefficient_oscillations_second} and \eqref{eq:coefficient_oscillations_third} in \eqref{eq:cor_third_case_intermediate_fourth} and using the equality $z+\overline{z} = 2|z|\cos(\arg(z))$ yields that the statement of the proposition is proved for $M = \frac{1}{2k}\frac{2^{\frac{1}{k}}\exp\left(\frac{2\pi}{k}i\right)}{2^{\frac{1}{k}}\exp\left(\frac{2\pi}{k}i\right) - 1} W(\lambda +i \tau).$
\end{proof}
\noindent It remains now to combine Propositions \ref{prop:asymptotic_Xt} and \ref{prop:asymptotic_Yt} to obtain the asymptotic behavior of $(R_t^{\delta})_{t\geq0}$ for all $\delta >0$.
\subsection{Proof of  Theorem \ref{thm:main_result}} 
\begin{proof}[Proof of i)]
Let $\delta >0$, $t\geq0$. As $R_{t}^{\delta}= X_{t}^{\delta} + Y_{t}^{\delta}$ and $Y_{t}^{\delta}$ is $\mathcal{F}_t-$measurable, conditioning with respect to $\mathcal{F}_t$ yields
$$
\begin{aligned}
\mathbb{E}\left[e^{isR_{t}^{\delta}/\sqrt{N_t}}\right] &= \mathbb{E}\left[\mathbb{E}\left[e^{isX_{t}^{\delta}/\sqrt{N_t}}\,\Big\vert\, \mathcal{F}_t\right]e^{isY_{t}^{\delta}/\sqrt{N_t}}\right].
\end{aligned}
$$
Then, using the triangular inequality and the fact that $\vert \exp(ix)\vert=1$ for $x\in \mathbb R$ and  $\sigma_{\delta}^2 = \sigma_{X,\delta}^2 + \sigma_{Y,\delta}^2$, we get
%$$
%\begin{aligned}
%\left|\mathbb{E}\left[e^{isR_{t}^{\delta}/\sqrt{N_t}}\right] - e^{-\sigma_{\delta}^2s^2/2}\right| &\leq \left|\mathbb{E}\left[\left(\mathbb{E}\left[e^{isX_{t}^{\delta}/\sqrt{N_t}}\bigg\vert \mathcal{F}_t\right]-e^{-\sigma_{X,\delta}^2s^2/2}\right)e^{isY_{t}^{\delta}/\sqrt{N_t}}\right]\right| \\ 
%&\qquad + e^{-\sigma_{X,\delta}^2s^2/2}\left|\mathbb{E}\left[e^{isY_{t}^{\delta}/\sqrt{N_t}}\right]-e^{-\sigma_{Y,\delta}^2s^2/2}\right|. \\ 
%\end{aligned}
%$$
%This becomes as $|e^{isY_{t}^{\delta}/\sqrt{N_t}}| = 1$
\begin{equation}\label{eq:main_result_first_case_intermediate}
\begin{aligned}
\left|\mathbb{E}\left[e^{isR_{t}^{\delta}/\sqrt{N_t}}\right] - e^{-\sigma_{\delta}^2s^2/2}\right| &\leq  \mathbb{E}\left[\left|\mathbb{E}\left[e^{isX_{t}^{\delta}/\sqrt{N_t}}\,\Big\vert\, \mathcal{F}_t\right]-e^{-\sigma_{X,\delta}^2s^2/2}\right|\right] \\ 
&\qquad + e^{-\sigma_{X,\delta}^2s^2/2} \left|\mathbb{E}\left[e^{isY_{t}^{\delta}/\sqrt{N_t}}\right]-e^{-\sigma_{Y,\delta}^2s^2/2}\right|.
\end{aligned}
\end{equation}
By Proposition \ref{prop:asymptotic_Xt} and the dominated convergence theorem, we have
$$
\mathbb{E}\left[\left|\mathbb{E}\left[e^{isX_{t}^{\delta}/\sqrt{N_t}}\,\Big\vert\, \mathcal{F}_t\right]-e^{-\sigma_{X,\delta}^2s^2/2}\right|\right] \underset{t\longrightarrow+\infty}{\longrightarrow} 0.
$$
By Proposition \ref{prop:asymptotic_Yt} and the Levy's theorem, we also have
$$
\left|\mathbb{E}\left[e^{isY_{t}^{\delta}/\sqrt{N_t}}\right]-e^{-\sigma_{Y,\delta}^2s^2/2}\right| \underset{t\longrightarrow+\infty}{\longrightarrow} 0.
$$
Plugging these two results of convergence in \eqref{eq:main_result_first_case_intermediate} implies that $i)$ is proved.
\end{proof}
\begin{proof}[Proof of ii)] By Proposition \ref{prop:asymptotic_Xt} and Slutsky's lemma, 
$X_{t}^{\delta}/\sqrt{tN_t}$
converges to $0$ in law and thus in probability, as
$t$ tends to infinity.
%$ \overset{\mathcal{L}}{\underset{t\longrightarrow+\infty}{\longrightarrow}} 0 \Longleftrightarrow \frac{X_{t}^{\delta}}{\sqrt{tN_t}} \overset{\mathbb{P}}{\underset{t\longrightarrow+\infty}{\longrightarrow}} 0.$$
By Proposition \ref{prop:asymptotic_Yt},
$Y_{t}^{\delta}/\sqrt{tN_t}$
%$ \overset{\mathcal{L}}{\underset{t\longrightarrow+\infty}{\Longrightarrow}} 
converges in law to a centered Gaussian variable
with variance $2\alpha\sigma_{Y,\delta}^2$.
Adding these two  convergences with Slutsky's lemma ends the proof of $ii)$.
\end{proof}
\begin{proof}[Proof of iii)] 
%To simplify notations, we denote for all $t\geq0$, $\delta >0$ $$C_t^{\delta} = 2\left|W_{\delta} \right| \cos\left[\tau t +  \arg\left(W_{\delta}\right)\right].$$ By \eqref{eq:expression_eigenvalues_gamma}, 
In the case $iii)$,
$\lambda > \alpha/2$ and Proposition \ref{prop:asymptotic_Xt}
%and Slutsky's lemma, we have
ensures that
%\begin{equation}\label{eq:main_result_third_case}
$X_{t}^{\delta}\exp\left(-\lambda t\right)$
converges to $0$ 
in law and thus in probability, as $t$ tends to infinity.
%r\overset{\mathcal{L}}{\underset{t\longrightarrow+\infty}{\longrightarrow}} 0 \Longleftrightarrow \frac{X_{t}^{\delta}}{\exp\left(\lambda t\right)} \overset{\mathbb{P}}{\underset{t\longrightarrow+\infty}{\longrightarrow}} 0.
%\end{equation}
Besides, Proposition \ref{prop:asymptotic_Yt} $iii)$ ensures that
$Y_{t}^{\delta}\exp\left(-\lambda t\right)$ converges in probability to
$2\left|M_{\delta} \right| \cos\left[\tau t +  \arg\left(M_{\delta}\right)\right]$.
Adding these two limits yields the result.
\end{proof}
\section{Estimation of the parameters from simulations}\label{sect:simus}
\subsection{Framework and estimation protocol}\label{subsect:framework_protocol}
Let us use the results
above on   the asymptotic behavior of $(R_{t}^{\delta})_{t\geq0}$  and propose an efficient way to infer the parameters $(k,\theta)$ of the lifespan.  We assume 
that our data set consists in the observation of $n_{\text{data}}$ independent realizations of a Bellman-Harris with time distribution $\Gamma(k,\theta)$ at fixed instants of measures. We may forget the first times, when the population is small, since we rely on asymptotic analysis and are motivated by this framework. We may also discuss on the role of the times and choose the relevant set of times to exploit.  The corresponding data set we use is then  denoted by 
\begin{equation}\label{eq:dataset_simus}
(N_{i\Delta}^{(j)} : i\leq I,j\in\llbracket 1, n_{\text{data}}\rrbracket),
\end{equation}
where $\Delta >0$, $I \in \mathbb{N}^*$. %Time step $\delta$ for fluctuations $(R_{t,\delta})_{t\geq0}$ must be chosen as a multiple of $\Delta$.

We also need a numerical approximation of $\sigma_{\delta}^2$ introduced in Theorem \ref{thm:main_result} $i)$. The latter depends on the parameters $(k,\theta)$ of the lifetime distribution. Using the relation $\alpha~=~(2^{1/k}-1)/\theta$, we can also say that $\sigma_{\delta}^2$ depends on the values of $(k,\alpha)$. We prefer this viewpoint, as this is more relevant for our inference method. Thus, from now on, $\sigma_{\delta}^2$ refers to the function

\begin{equation}\label{eq:variance_residuals}
\sigma_{\delta}^2 : (k,\alpha)  \in [1,k_c)\times\mathbb{R}_+^*\mapsto \sigma_{\delta}(k,\alpha).
\end{equation}
The domain of definition of $\sigma_{\delta}^2$ in the variable $k$ is $[1,k_c)$, as it is the range of parameters in which Theorem \ref{thm:main_result} $i)$ is valid.
%In Figure \ref{fig:psi_independence_alpha}, we observe that for $\delta = \ln(2)/\alpha$, the value of $\sigma_{\delta}^2$ does not seems to change when $\alpha$. Thus, for $\delta = \ln(2)/\alpha$, we approximate $\sigma_{\delta}^2$ with an “approximator" independent of $\alpha$, that we denote $k \mapsto \overline{\sigma}_{\delta}^2(k)$. This allows us to save a lot of time for the computation of our “approximator", without compromising the quality of the results, see Section \ref{subsubsect:quality_gaussian}. For the moment, the independence of $\sigma_{\delta}^2$ to $\alpha$ has been only observed on simulations. The obtention of a theoretical result to justify it may be done in future works, and would allow to complete this study. 

We denote $(k,\alpha) \mapsto \overline{\sigma}_{\delta}^2(k,\alpha)$ the approximation of $\sigma_{\delta}^2$ we use here to do our inference. Let us explain how we compute it when $\alpha$ is fixed. We consider a grid of parameters
$$
\mathbb{G}_p = \left\{1 + pl\,|\, l \in \llbracket0,\lfloor (k_c - 1)/p\rfloor -1\rrbracket\right\},
$$
where $p >0$, that covers the domain of $\sigma_{\delta}^2$ in the variable $k$, see~\eqref{eq:variance_residuals}. Then, for any~$k\in \mathbb{G}_p$, we proceed by Monte Carlo simulations to approximate all the expectations and variances that compose $\sigma_{\delta}^2(k,\alpha)$, except those in function $h^{(\delta)}$. We refer to Section \ref{subsect:approximation_mean_Bellman-Harris} to see how we approximate the latter. This gives us at the end an approximation $\overline{\sigma}_{\delta}^2(k,\alpha)$ of $\sigma_{\delta}^2(k,\alpha)$ for all $k\in\mathbb{G}_p$, that we extend to $k\geq 1$ using interpolation techniques. %We are now able to approximate $\sigma_{\delta}^2(k,\alpha)$ for all $k\geq 1$, when $\alpha$ is fixed.

To verify that this method works, we plot in Figure \ref{fig:approximation_residuals} the evolution of $R_{t,\delta}/{\sqrt{N_t}}$, for different sets of parameters, and see if the latter converges to a value close to $\overline{\sigma}_{\delta}^2$. We observe that this is the case. Thus, this approximation method seems good.

%We illustrate the results of the approximation in Figure \ref{fig:approximation_residuals}, and we see that the approximation seems very good. 
%We take $p~\in~\{1/20,1/10,1/2\}$ for the step of the grid. %The mesh size is not large linked to the fact that the computation of $\overline{\sigma}_{\delta}^2$ needs time.
%Increasing the mesh size of our parameter grid would allow us to be even more precise. However, as the approximation we have is already good, see Figure \ref{fig:approximation_residuals} and Section \ref{subsubsect:quality_gaussian}, given the time it would take to do the approximation with a larger grid, we considered that the cost was not worth it. %Again, we denote $\overline{\sigma}_{\delta}^2(k,1)$ the approximation of $\sigma_{\delta}^2(k,1)$ for parameters not in the grid, and more generally, $\overline{\sigma}_{\delta}^2$ the function that approximate $\sigma_{\delta}^2$. 

\begin{figure}[!ht]
    \centering
    \begin{subfigure}[t]{0.4\textwidth}
        \centering
        \includegraphics[width=\textwidth]{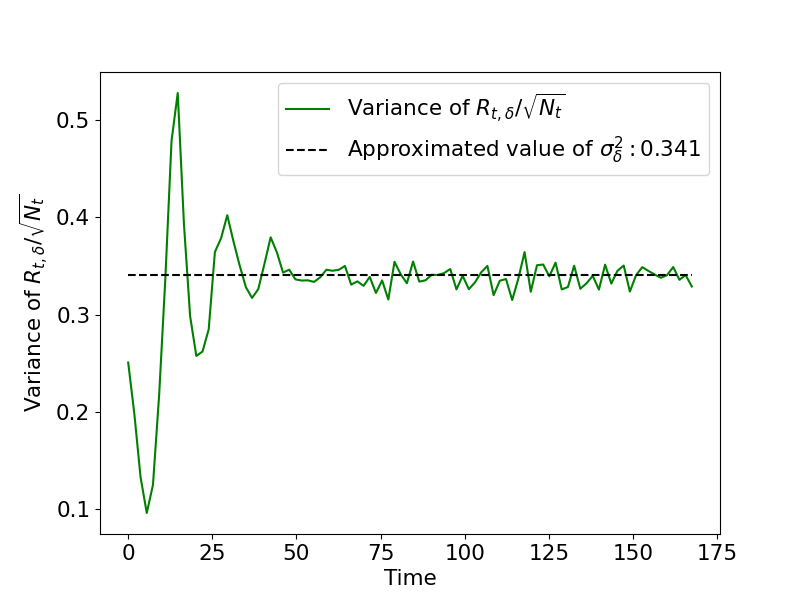}
        \caption{$(k,\theta) = (18.75,0.8)$.}
    \end{subfigure}
    \hfill
    \begin{subfigure}[t]{0.4\textwidth}
        \centering
        \includegraphics[width=\textwidth]{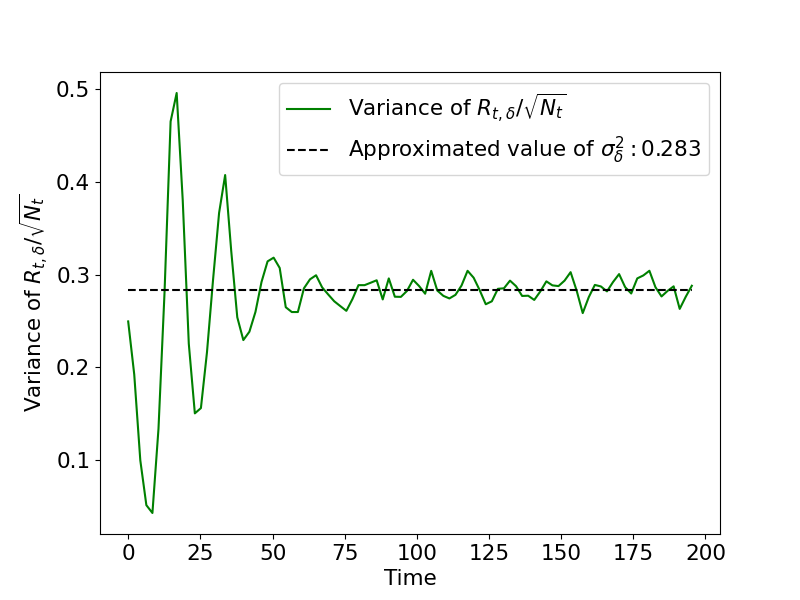}
        \caption{$(k,\theta) =(28.9, 0.59)$.}
    \end{subfigure}
    
    \begin{subfigure}[t]{0.4\textwidth}
        \centering
        \includegraphics[width=\textwidth]{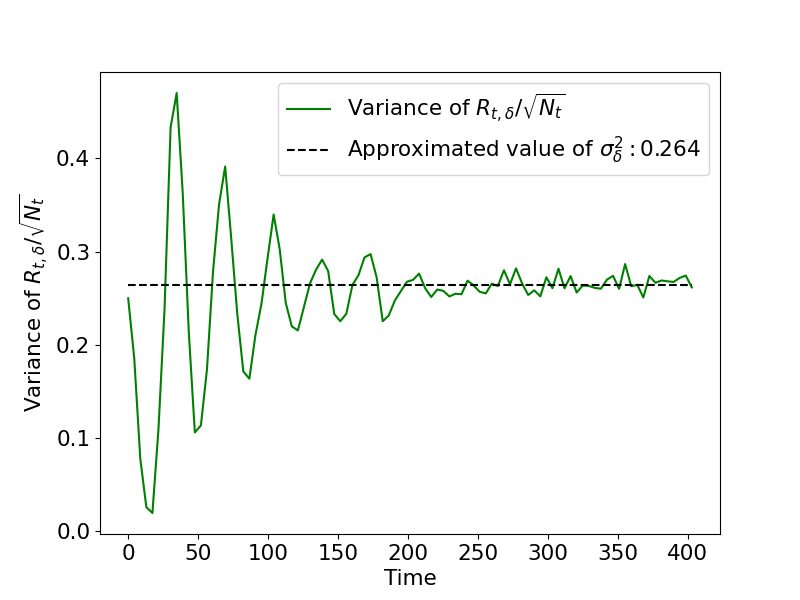}
        \caption{$(k,\theta) =(35,1)$.}
    \end{subfigure}
    \hfill 
    \begin{subfigure}[t]{0.4\textwidth}
        \centering
        \includegraphics[width=\textwidth]{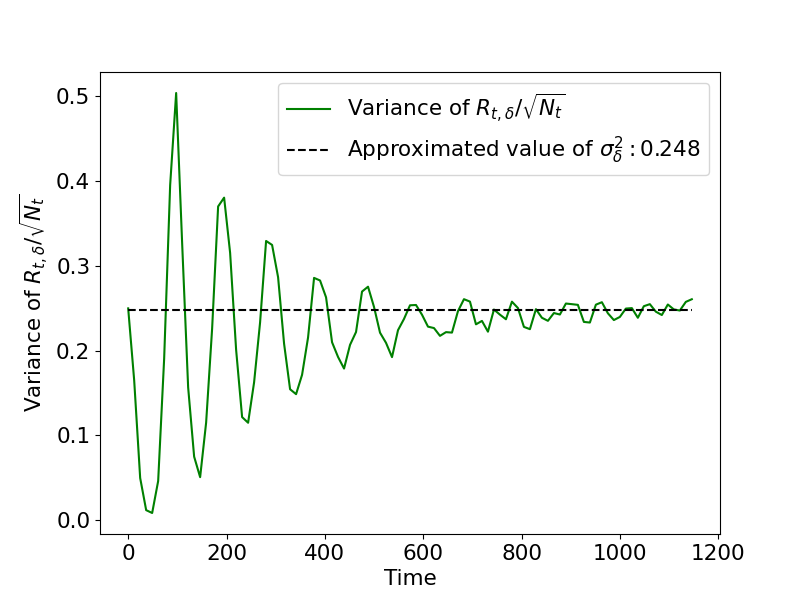}
        \caption{$(k,\theta) =(46.8,2.1)$.}
    \end{subfigure}
    \caption{Comparison of $\overline{\sigma}_{\delta}^2$ with the variance of $\frac{R_{t,\delta}}{\sqrt{N_t}}$ in different cases. \it{We use the grid of parameters $\mathbb{G}_{1/20}$, and we take $\delta$ as explained in Section \ref{subsubsect:step_identifiability}. We also estimate the variance of $\frac{R_{t,\delta}}{\sqrt{N_t}}$ using the empirical estimator of the variance with $2000$ simulations of Bellman-Harris dynamics. To simplify the computation of $\overline{\sigma}_{\delta}^2$, we do the approximation that $\sigma_{\delta}^2$ does not depends on $\alpha$ (or equivalently $\theta$), see Figure \ref{fig:psi_independence_alpha} and the end of Section \ref{subsubsect:step_identifiability}.}}
    \label{fig:approximation_residuals}
\end{figure}
%\begin{figure}[!ht]
%    \centering
%    \includegraphics[scale = 0.275]{images/simus/estimation_important_quantities/approximation_residuals/approximation variance residuals.jpg}
%    \caption{Comparison of $\overline{\sigma}_{\delta}^2$ with the asymptotic value of $\frac{R_{t,\delta}}{\sqrt{N_t}}$ in different cases.}
%    \label{fig:approximation_residuals}
%\end{figure}
%refer to Section \ref{subsect:approx_var_residuals}. 
%suppose that 
%considering a grid of parameters $\text{Grid}_{\text{param}} \subset [1,+\infty[\times\mathbb{R}_+^*$ such that $|\text{Grid}_{\text{param}}|<+\infty$, 
%we have approximated the value of $\sigma_{\delta}^2(k,\theta)$ for all $(k,\theta) \in [1,+\infty[\times\mathbb{R}_+^*$ by an estimator $\widehat{\sigma_{\delta}^2}(k,\theta)$ (see Appendix \ref{subsect:approx_var_residuals}).
%When $\delta \neq \ln(2)/\alpha$, we do not have necessarily the independance of $\sigma_{\delta}^2$ to $\alpha$. However, this is not a problem because we only approximate $\sigma_{\delta}^2$ for $\alpha = 1$, see Figure \ref{fig:residuals_without_identifiability}. The procedure of approximation is the same as the one done to obtain $\overline{\sigma}_{\delta}^2$, so we do not detail it.

We now have everything we need to do the inference. We follow the following pipeline to recover $(k,\theta)$, and  test it with simulations. In the rest of Section \ref{sect:simus}, the “number of simulations" is the number of dynamics we have simulated to create our dataset, so corresponds to $n_{\text{data}}$.  \\
% (we don't take into consideration the critical case as we are certainly not in this case)

\emph{Step 1 : estimation $\widehat{\alpha}$ of   the Malthusian coefficient $\alpha$.}
For all $j\in \llbracket1,n_{\text{data}}\rrbracket$, we do a linear regression of $t\mapsto \log (N_t^{(j)})$ using our data set $(N_{i\Delta}^{(j)} : i\leq I)$, which gives us an estimated value $\widehat{\alpha}^{(j)}$. Then, we take
$$
\widehat{\alpha} = \frac{1}{n_{\text{data}}}\sum_{i = 1}^{n_{\text{data}}} \widehat{\alpha}^{(j)}
$$
as an estimation of $\alpha$.

The estimation of $\alpha$ we obtain is very precise, even with a small number of simulations, as we see in Figure \ref{fig:illustration_estimation_alpha} where all the relative errors are below $1\%$. %Additionally, Figure \ref{fig:illustration_estimation_alpha} illustrates a slight increase of the error when $k$ grows, due to oscillations (see Figure \ref{fig:comparison_linear_regressions}). \vinc{le montrer sur un graphe ?} \ju{fait} 
In addition, even exponential quantities such as $e^{\widehat{\alpha}\delta}$ closely approximate $e^{\alpha\delta}$. Indeed, we plot in Figure \ref{fig:illustration_estimation_exponential_alpha} the relative errors for the estimation of $e^{\alpha\delta}$, with $\delta$ of the order we choose in practice (we refer to Sections \ref{subsubsect:step_detection_regime} and~\ref{subsect:inference_Gaussian_regime}), and we see that all the errors are below $1.5\%$. % One can also notice that Figure \ref{fig:illustration_estimation_linked_to_alpha} shows a slight increase of the error when the parameter $k$ grows. This is due to oscillations, and has been illustrated in Figures \ref{fig:mean_cells_small_amplitude} and \ref{fig:mean_cells_large_amplitude}.
%One can observe a slight increase of the error when $k$ grows, due to oscillations (see Figure \ref{fig:comparison_linear_regressions})
%\vinc{et/ou montrer ici que les quantites varient peu ? parce que a priori c'est sensible en alpha du fait de l'exponentielle} \ju{préfère ne pas en parler. à discuter.} \\

\smallskip
%\begin{figure}[!ht]
%    \centering\includegraphics[scale = 0.3]{images/simus/estimation_important_quantities/relatives_errors_linked_to_alpha.jpg}
%    \caption{Relative error of the estimation of $\alpha$ for a number of simulations in $\left\{ 50l\,|\,l\in\llbracket1,60\rrbracket\right\}$, and for different parameters (left). Relative error of the estimation of $e^{\alpha\delta}$ for $\delta \in \{\frac{l}{10}\frac{\log(2)}{\widehat{\alpha}}\,|\,l\in\llbracket1,20\rrbracket\}$, done with $100$ simulations, and for different parameters (right).}
%    \label{fig:illustration_estimation_alpha}
%\end{figure}
\begin{figure}[!ht]
    \centering
    \begin{subfigure}[t]{0.45\textwidth}
        \centering
        \includegraphics[width = \textwidth]{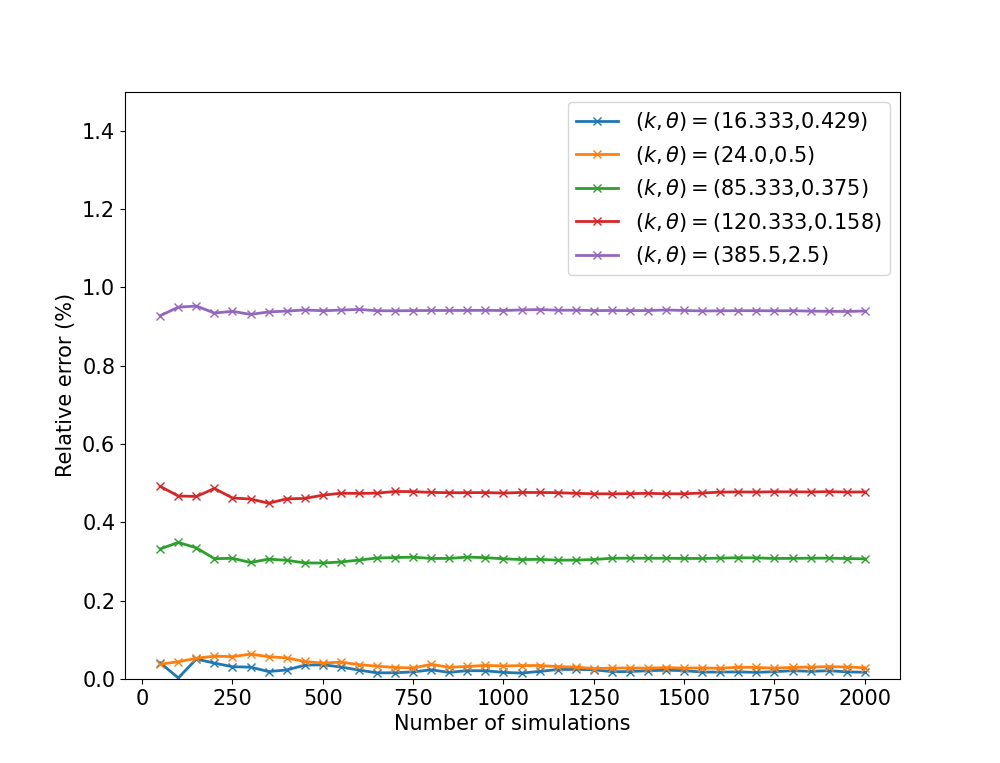}
        \caption{Relative error of the estimation of $\alpha$ versus the number or simulations, for a number of simulations in $\left\{ 50l\,|\,l\in\llbracket1,40\rrbracket\right\}$.}
        \label{fig:illustration_estimation_alpha}
    \end{subfigure}
    \hfill
    \begin{subfigure}[t]{0.45\textwidth}
        \centering
        \includegraphics[width = \textwidth]{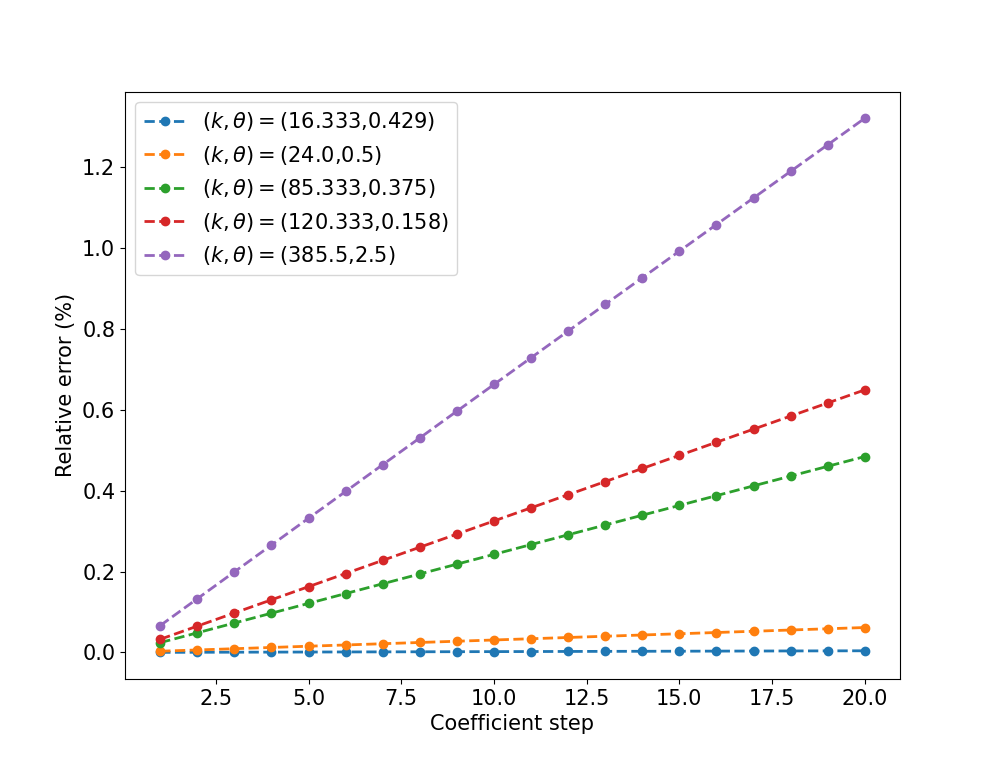}
        \caption{Relative error of the estimation of $e^{\alpha l \delta}$, with $\delta = \frac{\log(2)}{10\widehat{\alpha}}$, versus $l\in \llbracket 1, 20\rrbracket$. \it{The estimation was done using $100$ simulations.}}
        \label{fig:illustration_estimation_exponential_alpha}
    \end{subfigure}
    \caption{Illustration of the estimation of $\alpha$ and $e^{\alpha\delta}$ for different parameters. \it{Due to oscillations, see Figures \ref{fig:mean_cells_small_amplitude} and \ref{fig:mean_cells_large_amplitude}, the errors slightly increase when $k$ increases.}} \label{fig:illustration_estimation_linked_to_alpha}
\end{figure}
\emph{Step 2 : determination of the (Gaussian or oscillating) regime.}
We assume here that we do not fall in the critical and need to determine from the data set if $k >k_c$ or $k < k_c$. For that purpose, we compare $\widehat{\alpha}$ with the growth of the empirical variance of $N_{t+\delta} - e^{\widehat{\alpha}\delta}N_t$
\begin{equation}\label{eq:empirical_variance_residuals}
\widehat{\Var}\left(N_{t+\delta} - e^{\widehat{\alpha}\delta_1}N_t\right) = \frac{1}{n_{\text{data}}}\sum_{k = 1}^{n_{\text{data}}} \left(\left(N_{t+\delta}^{(k)} - e^{\widehat{\alpha}\delta}N_t^{(k)}\right) - \frac{1}{n_{\text{data}}}\sum_{j = 1}^{n_{\text{data}}}\left(N_{t+\delta}^{(j)} - e^{\widehat{\alpha}\delta}N_t^{(j)}\right) \right)^2,
\end{equation}
where $t=i\Delta$, $i< I$.
%,t+\delta \in \{t_i,i\in I\}$, that is. Following  Theorem \ref{thm:main_result}, we compare this empirical variance to  $\exp(\widehat{\alpha} t)$. Indeed this latter gives the good order of magnitude only in the Gaussian regime. 
At this step, we need to be careful in the choice
of  $\delta=\delta_1=n_1\Delta$, where $n_1 \in \mathbb{N}^*$, see Section \ref{subsect:determination_regime}.\\
%\vinc{j'ai pas mal modifie a discute, en particulier le lien avcec botre thm n'est pas completemnt clair, a la fois parce qu'on met $\widehat{\alpha}$ et pas $\alpha$, et parce que notre convergence en loi n'implique pas la convergence des variances a priori ...}\ju{ok, nous verrons}\\

\emph{Step 3 : estimation   $(\widehat{k},\widehat{\theta})$ of $(k,\theta)$.}
Recall that $\alpha$ has been estimated by $\widehat{\alpha}$. Depending on the regime, the estimator $(\widehat{k},\widehat{\theta})$ is different.\\
a) In the Gaussian regime, we first fix $n_{\text{times}}\in\mathbb{N}^*$, $\left(T_j\right)_{1\leq j \leq n_{\text{times}}}$ a family of times such that for all $j\in\llbracket1,n_{\text{times}}\rrbracket$ we have $\left(T_j/\Delta,(T_j+\delta)/\Delta\right)\in \llbracket0,I\rrbracket^2$. Then, we use the following estimators to estimate $k$ and $\theta$
\begin{equation}\label{eq:Gaussian_regime_estimator}
\begin{aligned}
        \widehat{k} = \underset{k\in[1,k_c[}{\text{argmin}} \left|\overline{\sigma}_{\delta}^2(k,\widehat{\alpha}) - \frac{1}{n_{\text{times}}}\sum_{j = 1}^{n_{\text{times}}}\widehat{\Var}\left( N_{T_j+\delta} - e^{\widehat{\alpha}\delta}N_{T_j}\right)\right|, \hspace{1.5mm}\text{ and }\hspace{1.5mm}\widehat{\theta}= (2^{1/\widehat{k}} - 1)/\widehat{\alpha},
        \end{aligned}
        \end{equation}
        where $k_c$ defined in~\eqref{eq:definition_thresold_kc} is the threshold between the two regimes,  and $\overline{\sigma}_{\delta}^2(k,\widehat{\alpha})$ is the function used to approximate $\sigma_{\delta}^2(k,\widehat{\alpha})$, see the third paragraph of this section. %and $\widehat{\sigma_{\delta}^2}(k,\theta)$ is an estimation of $\sigma_{\delta}^2(k,\theta)$. %\vinc{il  ne faudrait pas combiner en $T$ ?} \ju{Je ne comprends pas ce que veut dire "combiner en T"} \vinc{je veux dire utiliser non pas un seul couple $(T,T+\delta)$ pour l'estimation, mais tous ceux qu'on peut}.
        %To compute the term $\sigma_{Y,\delta}^2$ in $\sigma_{\delta}^2$, we use the formula given in \eqref{eq:expression_variance_fluctuations_first}. The values of all the expectations and variance needed to obtain $\sigma_{\delta}^2$ have been approximated with their empirical estimator. %\vinc{A clarifier mon avis comme $\sigma_{\delta}$ est estime il lui faut un chapeau aussi, et peut etre commnecer par lui puis introduire une notation pour $\widehat{\Var}\left( N_{T+\delta} - e^{\widehat{\alpha}\delta}N_T\right)$
       %pour tout mettre sur le meme plan ?}
        Here the time step
        $\delta=\delta_2 = n_2\Delta$, where $n_2 \in \mathbb{N}^*$, needs also to be chosen carefully to ensure uniqueness of the argmin. It will be different from $\delta_1$, 
    see Sections \ref{subsect:determination_regime} and \ref{subsect:inference_Gaussian_regime}. In addition, the times $(T_j)_{1\leq j \leq n_{\text{times}}}$ must be chosen “large", as we are interested in the asymptotic value of this variance. We choose $n_{\text{times}}$ times instead of only one, to have a more stable estimation of this variance. For more information, we refer to the end of Section~\ref{subsubsect:summary_choices_hyperparameters}.\\
    % and $R_T^{\delta_2} = N_{T+\delta} - e^{\widehat{\alpha}\delta}N_T$  .
b) In the  oscillating regime, we first estimate $\lambda$ by $\widehat{\lambda}$ thanks to a linear regression :  $\widehat{\lambda}$ is the slope of 
$$t\mapsto  \frac{1}{2}
\log\left(\widehat{\Var}\left(N_t - e^{\widehat{\alpha}\delta_1}N_t\right)\right).$$ Then, we estimate $(k,\theta)$ with $$(\widehat{k},\widehat{\theta}) = \underset{\substack{(k,\theta)\in(k_c,+\infty[\times\mathbb{R}_+^* \\ (2^{1/k} - 1)/\theta = \widehat{\alpha}}}{\text{argmin}} \left|(2^{1/k}\cos\left(2\pi/k\right)-1)/\theta - \widehat{\lambda}\right|.$$
%\vinc{mettre des $\widehat{}$ ou ca manque}

Let us now detail the main points of this procedure and give the results obtained by simulations.
We first explain how to choose  $\delta_1 >0$ for the determination of the regime in Step $2$.
Then, we consider  the Gaussian regime and explain how to choose $\delta_2$ and compare in that case the estimated value to the theoretical one for simulations. We   finally deal with the oscillating regime. The score we use to measure the quality of the estimation is the relative error, defined as 
$$
\text{relative error}=\frac{|\text{estimated value} - \text{theoric value}|}{\text{theoric value}}.
$$

\subsection{Determination of the regime}\label{subsect:determination_regime}

%As said in the introduction of this section, the second step of the pipeline of estimation is devoted to the detection of the regime. We want to know if we are in the Gaussian regime or in the oscillating regime. 
%We use $\widehat{Var}\left(N_t - e^{\widehat{\alpha}\delta_1}N_t\right)$, where $\delta_1 >0$. 
%\subsubsection{Illustration of the two regimes} 

\subsubsection{Why do we need to take care of the time step for the detection of the regime ?} 
%\vinc{faire attention $\delta$, $\delta_1$}
Our approach relies on the study of
$$R_{T}^{\delta} =N_{T+\delta} -e^{\delta  \alpha} \, N_T\,=  X_{T}^{\delta} + Y_{T}^{\delta}.$$
 Indeed, $R_{T}^{\delta}$ 
can be estimated from data, and is related
to the parameters of interest $(k,\theta)$ in a sensible way. This link has been studied in Section \ref{sect:asympt}. We have shown that when $\lambda > \alpha/2$,   $X_{T}^{\delta}$ and $Y_{T}^{\delta}$ behave for large $T$  as follows
$$ X_{T,\delta} \approx \mathcal{N}(0,\sigma_{X,\delta}^2)\sqrt{n_1W}e^{T\alpha/2 }, 
\qquad  Y_{T}^{\delta}  \approx 2\left|M_{\delta}\right| \cos\left[\tau T +  \arg\left(M_{\delta}\right)\right]e^{T \lambda }.
$$
The Gaussian part $X_{T,\delta}$ then prevails compared to $Y_{T}^{\delta}$ when $T\rightarrow\infty$, see Theorem 
\ref{thm:main_result}. We also recall that
$$M_{\delta}=f(\delta)M, \quad  \text{with }  f(\delta) = \left|e^{(\lambda + i\tau)\delta} - e^{\alpha\delta}\right|.$$
For applications, $T$ may not be very large. If $T\times(\lambda - \alpha/2)$ is positive but small, and  
$M_{\delta}$ is close to $0$, the two contributions $X_{T}^{\delta}$ and $Y_{T}^{\delta}$ may be comparable. In that case, it becomes  difficult to detect that
$\lambda - \alpha/2$ is positive and thus that we are in the oscillating regime. That's why the  choice of  the time step $\delta$  matters, and we explain now  how to avoid that $f(\delta)$ (and thus $M_{\delta}$) is too small.
%when $\left|W_{\delta}\right|$ is too small, even if $e^{\lambda t}$ is greater than $e^{\alpha/2 t}$, for $t\in[0,T]$, the product 
%$$|2\left|W_{\delta} \right| \cos\left[\tau t\arg\left(W_{\delta}\right)\right]e^{\lambda t}$$  may be smaller than the product $\mathcal{N}(0,\sigma_{X,\delta}^2)\sqrt{W}e^{\alpha/2 t}$. Then, when we make the linear regression of $\log(\widehat{\Var}(R_{t,\delta}))$ versus the time to detect the regime, the term linked to $X_{t,\delta}$ can distort the results because it is not yet negligible with respect to $Y_{t,\delta}$. 

\subsubsection{Time step $\delta_1$ and  detection of the regime.}\label{subsubsect:step_detection_regime}
Let us write $\delta=\delta(c) = c\log(2)/\alpha$, where $c>0$, and recall that 
$$\alpha = \frac{2^{\frac{1}{k}} - 1}{\theta}, \qquad \lambda=\frac{2^{\frac{1}{k}}\cos\left(\frac{2\pi}{k}\right) - 1}{\theta}>\alpha/2 , \qquad  \tau=\frac{2^{\frac{1}{k}}\sin\left(\frac{2\pi}{k}\right)}{\theta}.$$
We search a value of $c$  so that

$$
f(\delta(c))= \left|e^{(\lambda + i\tau)c.\log(2)/\alpha} - e^{\alpha c.\log(2)/\alpha}\right|= \left|e^{(\lambda + i\tau)c.\log(2)/\alpha} - 2^c\right|
$$ 
is not too close to $0$. As we have
$$
\begin{aligned}
\lim_{k\rightarrow\infty} (\lambda + i\tau)\delta(c) &=\lim_{k\rightarrow\infty} c\log(2)\left(\frac{2^{\frac{1}{k}}\cos\left(\frac{2\pi}{k}\right) - 1}{\frac{1}{k}} + i\frac{2^{\frac{1}{k}}\sin\left(\frac{2\pi}{k}\right)}{\frac{1}{k}}\right) \frac{\frac{1}{k}}{2^{\frac{1}{k}}-1}
%,\\  &\underset{k\longrightarrow+\infty}{\sim} \exp\left(c\alpha\delta\right)\exp\left(c\frac{2\pi}{k}\frac{log(2)}{2^{1/k}-1}i\right) \\ &
\\
&=c\left(\log(2) + i2\pi\right),
\end{aligned}
$$
we obtain
\begin{equation}\label{eq:modulus_versus_c_k_infinity}
\begin{aligned}
\lim_{k\rightarrow\infty} f(\delta(c))
& %\underset{k\longrightarrow+\infty}{\sim}
%\exp\left(c\alpha\delta\right)|\exp\left(c.2\pi i\right) - 1| \\ &= 
\, = \, 2^c\sqrt{2\left(1-\cos\left(2\pi c\right)\right)}.
\end{aligned}
\end{equation}
By this expression and the fact that the oscillating regime corresponds to large values of $k$, we see that the values of $c$ that seem interesting are those such that the cosines term is minimal, i.e. equal to $-1$. This corresponds to cases where $\delta = c \log(2)/\alpha$, with $c - 1/2\in \mathbb{N}$. At the opposite, the values of $c$ that should be avoided are those such that the cosines term is maximal, i.e. equal to $1$, when $c \in\mathbb{N}^*$. To illustrate this, let us plot in Figure \ref{fig:evol_module} the values of $f(\delta) = \left|e^{(\lambda + i\tau)\delta} - e^{\alpha\delta}\right|$ versus $\delta$, in the case where $(k,\theta) = (80.2,4)$. 
\begin{figure}[!ht]
\centering
\includegraphics[scale = 0.4]{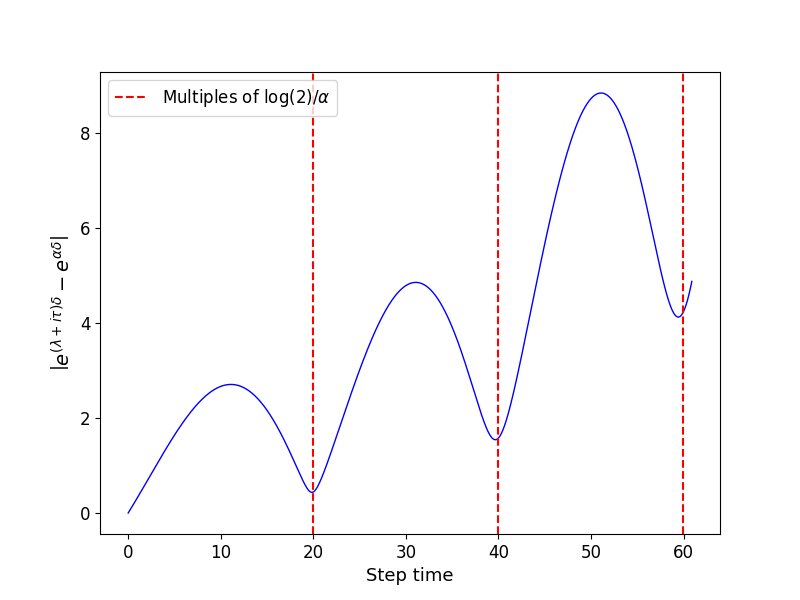}
\caption{Curve of $\left|e^{(\lambda+i\tau)\delta} - e^{\alpha\delta}\right|$ versus $\delta$, when $(k,\theta) = (80.2,4)$. \it{Red lines represent multiples of $\log(2)/\alpha$.}}
\label{fig:evol_module}
\end{figure}
We see in Figure \ref{fig:evol_module} that the value of the modulus is very low when $\delta$ is a multiple of~$\log(2)/\alpha$. Then, choosing a step time that is a multiple of this value would imply bad results on the estimation of the regime. %This enters in contradiction with the advice given in \cite{barizien_studying_2019}. 

At the opposite, in Figure \ref{fig:evol_module}, the amplitude of the oscillations seems to be maximal when $\delta = l\log(2)/\alpha + \log(2)/(2\alpha)$, where $l\in\mathbb{N}$. Then, it is better to take a time step as close as possible of these values when we detect the regime. By default, we take the minimal step that satisfy this condition, which is 
$$
\delta_1 = \underset{i\Delta,\,i\in\llbracket0,I\rrbracket}{\text{argmin}} \left|\Delta i -\log(2)/(2\widehat{\alpha})\right|.
$$
We take the minimal step because the error of the estimation of $e^{\alpha\delta}$ increases when $\delta$ increase, as observed in Figure \ref{fig:illustration_estimation_exponential_alpha}. 

\subsubsection{The threshold for the detection}\label{subsubsect:threshold_detection} To detect the regime,  we compare 
 the slope of the linear regression of $t\mapsto \widehat{\Var}\left(N_t - e^{\widehat{\alpha}\delta_1}N_t\right)$, which is $2\widehat{\lambda}$, to $\widehat{\alpha}$. Indeed, in the Gaussian regime, $\widehat{\lambda}$ does not estimate $\lambda$, but $\alpha/2$. In practice, we decide that we are in the Gaussian regime if  the difference of these two terms is less than $10\%$ of $\widehat{\alpha}$. It is relevant in various cases, see the results in Sections \ref{subsubsect:quality_gaussian} and \ref{subsect:inference_oscillating_regime}. A more systemic exploration of the set of parameters or theoretical may be interesting and would allow to be more precise.
% we  even in the Gaussian regime, the slope detected will never be exactly $\alpha$ as expected, but a number very close to it. That's why we need to take a threshold, such that "if the relative error between the slope and $\alpha$ is greater lower than the thresold, we detect that we are in the Gaussian regime, otherwise, we are in the oscillating regime". The thresold has been arbitrarly chosen for the inference we make in this article as $10\%$. A more precise study can be done to take the optimal thresold, we don't do it as this value seemed to be sufficient for good results. %\ju{en parler de ça ou supprimer la partie ?}

\subsection{Inference  in the Gaussian regime}\label{subsect:inference_Gaussian_regime}
\subsubsection{Identifiability : the choice of the time step $\delta_2$ for the estimation}\label{subsubsect:step_identifiability} We need to choose the time step $\delta_2$  to achieve Step $3$. Our estimator \eqref{eq:Gaussian_regime_estimator}
involves a minimization, and the problem is identifiable if there is a unique minimizer  (argument of the minimum). This depends on the choice of $\delta$. %Indeed, once we have estimated $\alpha$ by $\widehat{\alpha}$, we are bound to find the minimizer with the form $(k,\theta)=(k,(2^{1/k}-1)/\widehat{\alpha})$. 
Denoting $\widehat{\alpha}$ the estimator of $\alpha$, the uniqueness of the minimizer requires an injectivity property, for the function
$$
k \mapsto \sigma_{\delta}^2(k,\widehat{\alpha})= \sigma_{X,\delta}^2(k,\widehat{\alpha}) + 2\widehat{\alpha}\sigma_{Y,\delta}^2(k,\widehat{\alpha}).
$$

The function $\sigma_{\delta}^2(.,\widehat{\alpha})$ is not injective in general, as we see in Figure \ref{fig:residuals_without_identifiability}, where two different values of $k$ can have the same image by $\sigma_{\delta}^2(.,\widehat{\alpha})$. This comes from the fact that $\sigma_{X,\delta}^2(.,\widehat{\alpha})$ and 
$\sigma_{Y,\delta}^2(.,\widehat{\alpha})$ have opposite monotonicity.

\begin{figure}[!ht]
    \centering
%    \begin{subfigure}[t]{0.45\textwidth}
%        \centering
%    \includegraphics[width = \textwidth]{images/simus/gaussian-inference/study_time_step/step_time_without_identifiability_1over4.jpg}
%    \caption{For $\delta = \log(2)/(4\widehat{\alpha})$.}
%    \end{subfigure}
%    \hfill 
    \begin{subfigure}[t]{0.45\textwidth}
        \centering
    \includegraphics[width = \textwidth]{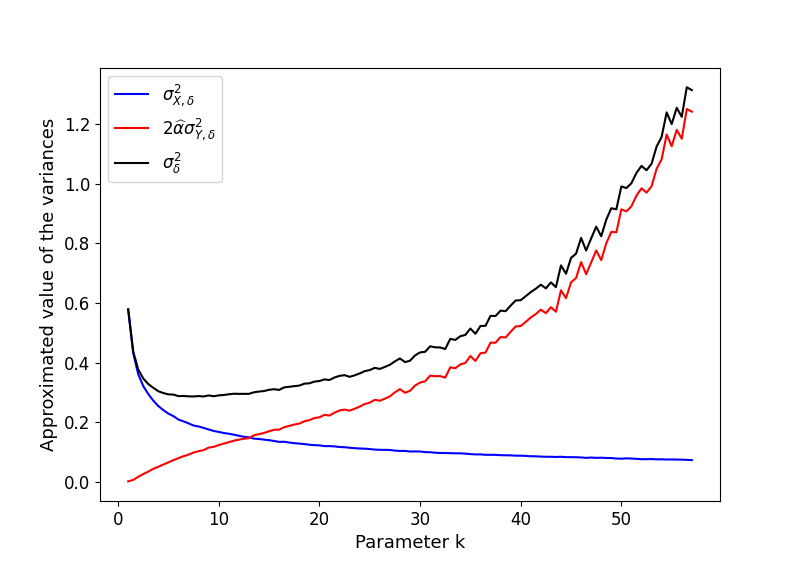}
    \caption{For $\delta = \log(2)/(2\widehat{\alpha})$.}
    \end{subfigure}
    \hfill
    \begin{subfigure}[t]{0.45\textwidth}
    \centering
    \includegraphics[width = \textwidth]{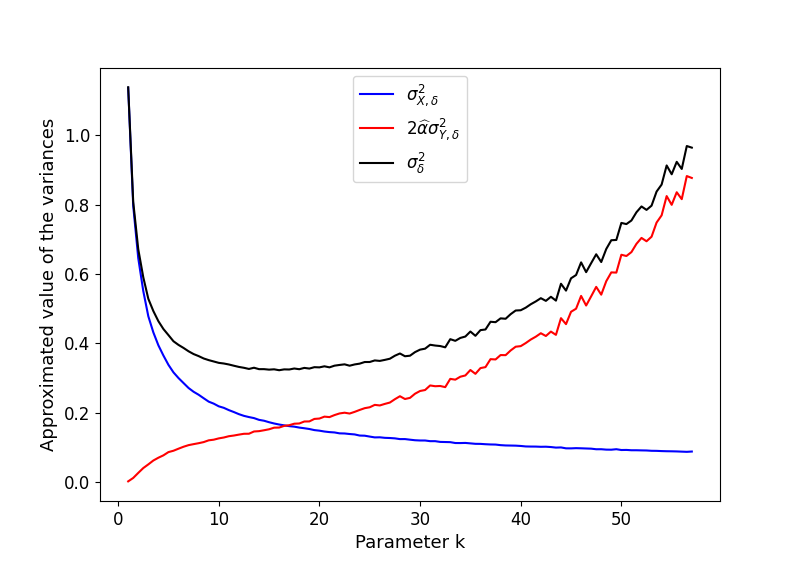}
    \caption{For $\delta = 3\log(2)/(4\widehat{\alpha})$.}
    \end{subfigure}
    \caption{$\sigma_{\delta}^2(k,\widehat{\alpha})$ (black), $\sigma_{X,\delta}^2(k,\widehat{\alpha})$ (blue) and $2\widehat{\alpha}\sigma_{Y,\delta}^2(k,\widehat{\alpha})$ versus $k$, for $\widehat{\alpha} = 1$, and for different time steps $\delta$ such that $\sigma_{\delta}^2(.,\widehat{\alpha})$ is not injective.  \it{$\sigma_{\delta}^2$ has been approximated using $\overline{\sigma}_{\delta}^2$, computed with the grid of parameters $\mathbb{G}_{1/2}$. The mesh size is small, because the non-smoothness of the curves when k is large decreases the readability of the figure with a larger mesh size. To approximate $\sigma_{X,\delta}^2$ and~$\sigma_{Y,\delta}^2$, we only keep the part of our approximation of $\sigma_{\delta}^2$ that corresponds to each of these quantities.}}
    \  \label{fig:residuals_without_identifiability}
\end{figure}
%Figure \ref{fig:residuals_without_identifiability} also illustrates the fact when $k$ is high, the approximation of $\psi_Y$, and as a consequence the approximation of $\psi$, has difficulties to stabilize for certain time steps. This is due to the fact that even if we are in the Gaussian regime, we approach to the threshold $k_c$. Then, the oscillations contained in $\psi_Y$, linked to the second eigenvalue, make the approximation more difficult. This can be solved by either decreasing the mesh size, but this yields to be less precise, or by increasing the number of simulations we use to compute $\overline{\sigma}_{\delta}^2$ with Monte-Carlo methods. The latter yields that the time to compute $\overline{\sigma}_{\delta}^2$ increases, and we want to avoid it because the computation is already long. Thus, in addition to find a time step for which the function $\psi(.,\delta)$ is injective, we want to find a time step for which we do not have this problem of stability when $k$ is high.
%We deal with this issue by choosing $\delta$ so that the contribution $\psi_Y$ is small enough to recover the injectivity of $\psi$ from the monotonicity of $\psi_X$, as we see in Figure \ref{fig:plot_injectivity}.

% Due to the fact that our approximation is based Monte-Carlo methods, curves not totally smoothed when $k$ is high. However, they still illustrate the fact that we do not have injectivity for these time steps.

In the oscillating regime, we saw in Section \ref{subsect:determination_regime} that we can significantly reduce the influence of $(Y_{t}^{\delta})_{t\geq0}$ by choosing $\delta = \log(2)/\widehat{\alpha}$, see Figure \ref{fig:evol_module}. Thus, our conjecture is that even in the Gaussian regime, choosing $\delta = \log(2)/\widehat{\alpha}$ would allow us to have injectivity of $\sigma_{\delta}^2(.,\widehat{\alpha})$, because the impact of $\sigma_{Y,\delta}^2(.,\widehat{\alpha})$ will be significantly reduced. To check this hypothesis, we plot in Figure \ref{fig:plot_injectivity} the evolution of $\sigma_{\delta}^2(k,\widehat{\alpha})$ versus the parameter $k$ for several values of $\widehat{\alpha}$. If the curves that we obtained are injective, then our conjecture is valid. We see in Figure \ref{fig:plot_injectivity} that this is the case. %In addition, one can see that there is no difficulties of stabilization for the approximation of $\psi$ when $k$ is high for this step. 

\begin{figure}[!ht]
    \centering
    \begin{subfigure}[t]{0.45\textwidth}
        \centering
        \includegraphics[width = \textwidth]{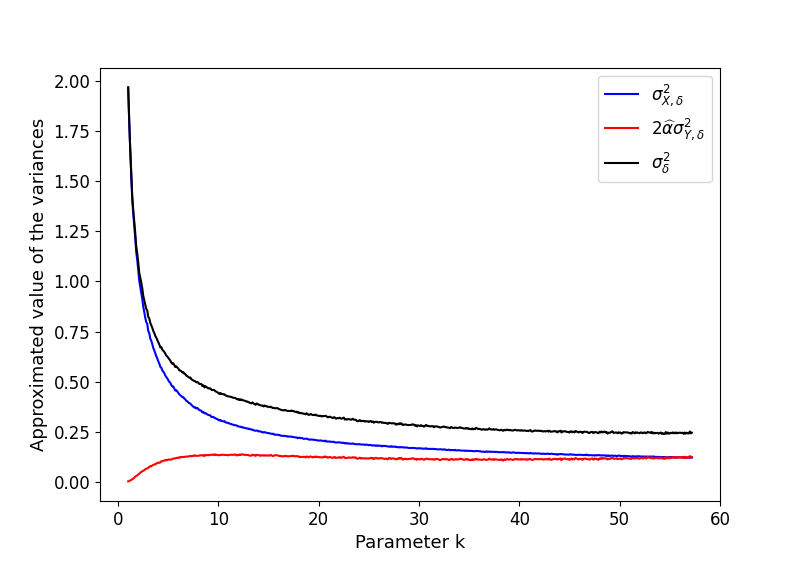}
        \caption{$\sigma_{\delta}^2(k,\widehat{\alpha})$ (black), $\sigma_{X,\delta}^2(k,\widehat{\alpha})$ (blue) and $2\widehat{\alpha}\sigma_{Y,\delta}^2(k,\widehat{\alpha})$ (red) versus $k$, for $\widehat{\alpha} = 1$.}
    \end{subfigure}
    \hfill 
    \begin{subfigure}[t]{0.45\textwidth}
        \centering
        \includegraphics[width = \textwidth]{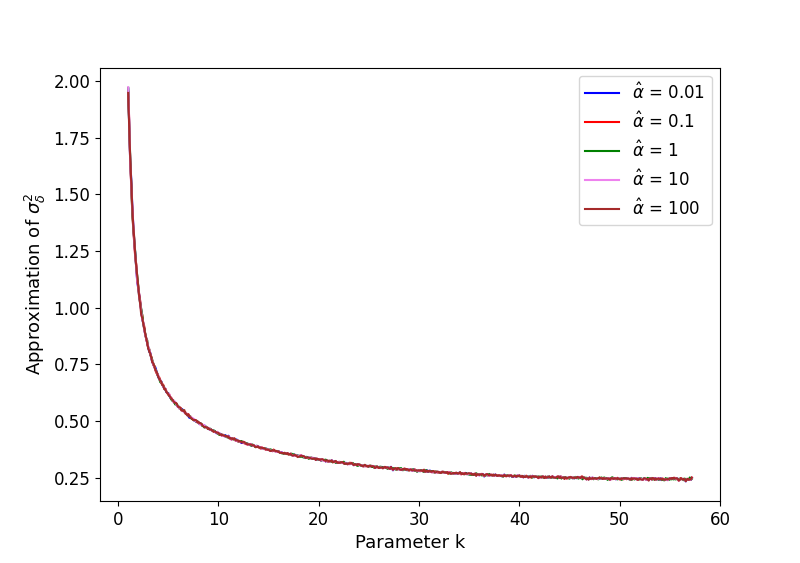}
        \caption{$\sigma_{\delta}^2(k,\widehat{\alpha})$ versus $k$ for $\widehat{\alpha}\in[0.01,0.1,1,10,100]$. \it{The figure on the right shows that $\sigma_{\delta}^2(k,\widehat{\alpha})$ seems to be independent of the value of $\widehat{\alpha}$.}}\label{fig:psi_independence_alpha}
    \end{subfigure}
    \caption{$\sigma_{\delta}^2(k,\widehat{\alpha})$ versus $k$ for $\delta = \log(2)/\widehat{\alpha}$, for different values of $\widehat{\alpha}$. \it{To approximate $\sigma_{\delta}^2$ (and $\sigma_{X,\delta}^2$ and $\sigma_{Y,\delta}^2$ that compose it), we use $\overline{\sigma}_{\delta}^2$. The grid of parameters used to do the approximation is $\mathbb{G}_{1/10}$.}}
    \label{fig:plot_injectivity}
\end{figure}

\noindent Then, for our inference in the Gaussian regime, we are going to use 

$$
\delta_2 = \underset{i\Delta,\,i\in\llbracket0,I\rrbracket}{\text{argmin}} \left|\Delta i -\log(2)/\widehat{\alpha}\right|.
$$
When $\delta_2$ is sufficiently close to $\log(2)/\widehat{\alpha}$, as we have observed in Figure \ref{fig:psi_independence_alpha} that $\sigma(k,\widehat{\alpha})$ seems independent of $\widehat{\alpha}$, we use $\overline{\sigma}_{\delta}^2(k,1)$ to approximate $\sigma(k,\widehat{\alpha})$. This allows us to save a lot of time for the computation of our “approximator" (that can take several days if we use a large mesh size), without compromising the quality of the results.

%As for $\delta_1$, replacing $\log(2)/\widehat{\alpha}$ by a multiple of it may also work, but it is better to choose the minimal step possible, because the error of the estimation of $e^{\alpha\delta}$ increases when $\delta$ increases, as illustrated in Figure \ref{fig:illustration_estimation_exponential_alpha}.

\subsubsection{Overview of the choices of the parameters to do the inference}\label{subsubsect:summary_choices_hyperparameters}
We are now going to check the quality of our inference method on simulated data. Let us first give  an overview of our choices for the parameters of the protocol presented in Section~\ref{subsect:framework_protocol}. They rely on the results of the previous sections.
%These parameters have been chosen thanks to the studies conducted in the previous sections, and thanks to new arguments presented in this section for the times~$(T_j)_{1\leq j \leq n_{\text{times}}}$ needed to do Step~$3$, see~\eqref{eq:Gaussian_regime_estimator}.} 
\smallskip

\emph{Choice 1: Time step for the determination of the regime.} The first parameter we must choose is the time step when we compute the growth of the empirical variance presented in~\eqref{eq:empirical_variance_residuals}. Our study in~Section~\ref{subsubsect:step_detection_regime} has highlighted that the step 
	\begin{equation}\label{eq:choice_step_time_regime_detection}
		\delta_1 = \underset{i\Delta,\,i\in\llbracket0,I\rrbracket}{\text{argmin}} \left|\Delta i -\log(2)/(2\widehat{\alpha})\right|
	\end{equation}
	should be a good choice. The reason is that we need to have a time step for which the modulus of $M_{\delta}$, defined in~\eqref{eq:random_modulus_oscillating_regime}, is large, and that we have exhibited in~Section~\ref{subsubsect:step_detection_regime} that the time step $\log(2)/(2\alpha)$ maximizes this modulus when $k\rightarrow +\infty$, see~\eqref{eq:modulus_versus_c_k_infinity}. %The potential improvement would be to determine the time step that maximizes this module when $k$ is large, but within a range of values relevant for our biological motivation.}
	%This would allow us to define a new choice of step, similar to the one in~\eqref{eq:choice_step_time_regime_detection}, but with a quantity other than~$\log(2)/(2\alpha)$ in its definition. 
	
	\smallskip
	
	\emph{Choice 2: Threshold for the determination of the regime.} The second parameter we must choose is the threshold for the relative error between $2\widehat{\lambda}$ and $\widehat{\alpha}$, from which we determine if $2\lambda$ is different from $\alpha$ or not, see Section~\ref{subsubsect:threshold_detection}. Taking it at $5\%$ leads to detection errors, while taking it at $15\%$ seems too high. Then, we set this threshold  at $10\%$. The fact that all the regime detected in Sections~\ref{subsubsect:quality_gaussian} and~\ref{subsect:inference_oscillating_regime} are correct validates that this choice is appropriate.
	
	\smallskip
	
	\emph{Choice 3: Time step for the inference in the Gaussian regime.} The third parameter we must choose is the time step when we compute the estimators of $(k,\theta)$ in the Gaussian regime, see~\eqref{eq:Gaussian_regime_estimator}. Our numerical study in Section~\ref{subsubsect:step_identifiability} has highlighted that the time step $\log(2)/\alpha$ is a time step for which the problem is identifiable. Then, a relevant choice for this time step is
$$
\delta_2 = \underset{i\Delta,\,i\in\llbracket0,I\rrbracket}{\text{argmin}} \left|\Delta i -\log(2)/\widehat{\alpha}\right|.
$$
Indeed, it is the closest to $\log(2)/\alpha$ that we can take for a dataset with the form presented in~\eqref{eq:dataset_simus}. %However, we do not yet have theoretical guarantees that this is true. We also do not yet know to what extent having not exactly $\log(2)/\alpha$, but an approximate value of it, can cause problems. Further studies are needed to understand this better. 

\emph{Choice 4: Time measurements for the inference in the Gaussian regime.}
We finally need to choose the family of time measurements $\left(T_j\right)_{1\leq j\leq n_{\text{times}}}$ for estimating the variance of the asymptotic fluctuations, see~\eqref{eq:Gaussian_regime_estimator}. As illustrated in Figure~\ref{fig:approximation_residuals}, the fluctuations converge towards $\sigma_{\delta}^2$ as $t\rightarrow +\infty$ by oscillating around it, especially when $k$ is close to $k_c$. This is due to the second eigenvalue.  Estimating the asymptotic variance of fluctuations with the empirical variance at only one time is thus not optimal, as this estimator may be too far above or below the real value of $\sigma_{\delta}^2$ due to the oscillations. We thus need to choose a family of times $\left(T_j\right)_{1\leq j\leq n_{\text{times}}}$ such that when $k$ is close to $k_c$, the empirical variance of the fluctuations is above $\sigma_{\delta}^2$ for a part of them, and below $\sigma_{\delta}^2$ for the other part. This would allow to have a good estimation of~$\sigma_{\delta}^2$ by taking the average of them, as presented in~\eqref{eq:Gaussian_regime_estimator}.

To get this property, we include  the time $\tilde{T}_1 = I\Delta - \delta_2$, which is the largest time for which we can compute the fluctuations. Moreover, we complete the family  of times with $\tilde{T}_1 - (j-1)\log(2)/(8\alpha)$, where~$j\in\llbracket1,8\rrbracket$. This allows to explore inside a time interval of length $\log(2)/\alpha$, that is approximately the period of the oscillations when $k$ is large, see Figure~\ref{fig:evol_module}.% We therefore have both times when the variance is greater than $\sigma_{\delta}^2$, and times when it is smaller.} 

In practice, we do not necessarily have measurements for the times presented above, see~\eqref{eq:dataset_simus}. Hence, we actually consider $n_{\text{times}} = 8$ and for all~$j\in\llbracket1,n_{\text{times}}\rrbracket$,
$$
T_j = \underset{i\Delta,\,i\in\llbracket0,I\rrbracket}{\text{argmin}} \left|I\Delta - \delta_2 -(j-1)\log(2)/(8\widehat{\alpha})\right|.
$$
The fact that all the estimation results are good in Section~\ref{subsubsect:quality_gaussian} confirms that this choice is appropriate.
%\ju{We would like to study in greater depth and more rigorously in a future work how these times should be chosen.} 
%}% By taking the average value of these estimations, we will thus have a better estimation of $\sigma_{\delta}^2$ than with only one time, for which the empirical variance can be above or below $\sigma_{\delta}^2$

%\smallskip 
%\ju{For all of these three choices, the time step and the threshold given above is not necessarily optimal. In addition, it lacks of theoretical guarantees. In future works, we would like to improve these choices, and the guarantees we have on them. This would allows us to improve the quality and the robustness of our inference method.}

\subsubsection{Quality of the estimation}\label{subsubsect:quality_gaussian} Now that we know which time steps, threshold and time measurements we take, we study the quality of the estimation on simulations. We denote in this paragraph \hbox{$(k_1,\theta_1) = (35,1)$}, $(k_2,\theta_2) = (25.4,2)$, $(k_3,\theta_3) = (14.5,3.4)$, and $(k_4,\theta_4) = (44,1.5)$. As for all $l\in \llbracket1,4\rrbracket$ we have $k_l < k_c$, these parameters are such that we are in the Gaussian regime, see Figure \ref{fig:thresold_spectral_gap}. For all $l\in\llbracket1,5\rrbracket$, we have simulated $n_{\text{data}} = 2000$ Bellman-Harris dynamics with lifetimes distributed according to $\Gamma(k_l,\theta_l)$, up to $8000$ cells. Then, for all~$l\in \llbracket1,5\rrbracket$, we create a dataset $(N_{l\Delta}^{(j)} : l\leq I,j\in\llbracket 1, n_{\text{data}}\rrbracket)$, where $I\in\mathbb{N}^*$, and denoting $\alpha_l = (2^{1/k_l} - 1)/\theta_l$,  $\Delta = \frac{\log(2)}{8\alpha_l}$. 

We use the pipeline explained at the beginning of Section \ref{sect:simus} to retrieve~$(k_i,\theta_i)$ for all~\hbox{$i\in\llbracket1,5\rrbracket$}, see Figure~\ref{fig:Gaussian_illustration_pipeline}. Then, we obtain an estimation of the mean and coefficient of variation of the distribution using the relation $(\mu,\frac{\sigma}{\mu}) = (k\theta, 1/\sqrt{k})$. For this inference, we use $\overline{\sigma}_{\delta}^2(k,1)$ to approximate $\sigma_{\delta}^2(k,\widehat{\alpha})$, see the end of Section \ref{subsubsect:step_identifiability}. The latter has been computed using a grid of parameters $\mathbb{G}_{1/20}$. The results of the estimation are given in~Table~\ref{tab:results_gaussian}. The scores that are the more relevant to consider are those for the estimation of~$(\mu,\frac{\sigma}{\mu})$. We see that they are very good. 

\begin{table}[!ht]
	\begin{center}
		\scalebox{0.65}{\begin{tabular}{|| c | c | c | c | c ||} 
				\hline
				Theoretical values of  $(k,\theta)$ & $(35, 1)$ & $(25.4, 2)$ & $(14.5, 3.4)$ & $(44, 1.5)$ \\ [0.5ex]  
				\hline
				$(k,\theta)$ Param. inferred  & $(34.67, 1.010)$ & $(25.68, 1.979)$ & $(14.65, 3.363)$ & $(47.71, 1.381)$  \\  [1ex] 
				
				$(k,\theta)$ Relative error & $(0.9302\%, 1.016\%)$ & $(1.093 \%, 1.058\%)$ & $(1.067 \%, 1.067\%)$ & $(8.426 \%,  7.919\%)$  \\  [1ex] 
				\hline\hline
				Theoretical values of $(\mu,\frac{\sigma}{\mu})$ & $(35, 0.1690)$ & $(50.8, 0.1984)$ & $(49.3,0.2626)$ & $(66, 0.1508)$ \\ [0.5ex]
				\hline
				$(\mu,\frac{\sigma}{\mu})$ Param. inferred  & $(35.03, 0.1698)$ & $(50.81, 0.1973)$ & $(49.29, 0.2612)$ & $(65.89, 0.1448)$ \\  [1ex] 
				
				$(\mu,\frac{\sigma}{\mu})$ Relative error & $(0.07611\%, 0.4684\%)$ & $(0.02374)\%, 0.5421\%)$ & $(0.01178 \%, 0.5293\%)$ & $(0.1598 \%, 3.964\%)$ \\  [1ex] 
				\hline
		\end{tabular}}.
	\end{center}
	\caption{Results of the estimation in the Gaussian regime. \it{The results are given with $4$ significant digits.}%\it{As $\psi(k,\delta)$ seems independent of the value of $\widehat{\alpha}$ by Figure \ref{fig:psi_independence_alpha}, we use $ \overline{\sigma}_{\delta}^2(k,(2^{1/k} - 1))$ instead of $\overline{\sigma}_{\delta}^2(k,(2^{1/k} - 1)/\widehat{\alpha})$, in the estimator introduced in \eqref{eq:Gaussian_regime_estimator}. This allows us to save a lot of time for our inference, without compromising the quality of the results.}
	}\label{tab:results_gaussian}
\end{table} 
\begin{figure}[!ht]
	\centering
	\begin{subfigure}[t]{0.475\textwidth}
		\centering
		\includegraphics[scale = 0.25]{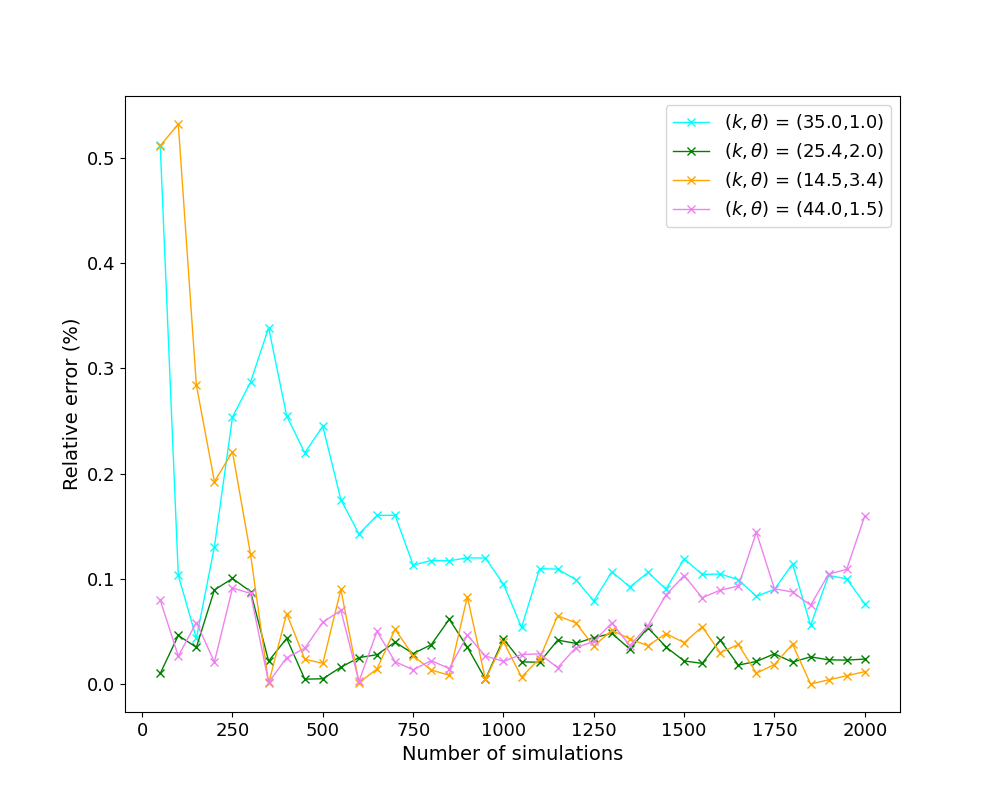}
		\caption{For $\mu$.}
		\label{fig:Gaussian_number_simulations_mean}
	\end{subfigure}
	\hfill
	\begin{subfigure}[t]{0.475\textwidth}
		\centering
		\includegraphics[scale = 0.25]{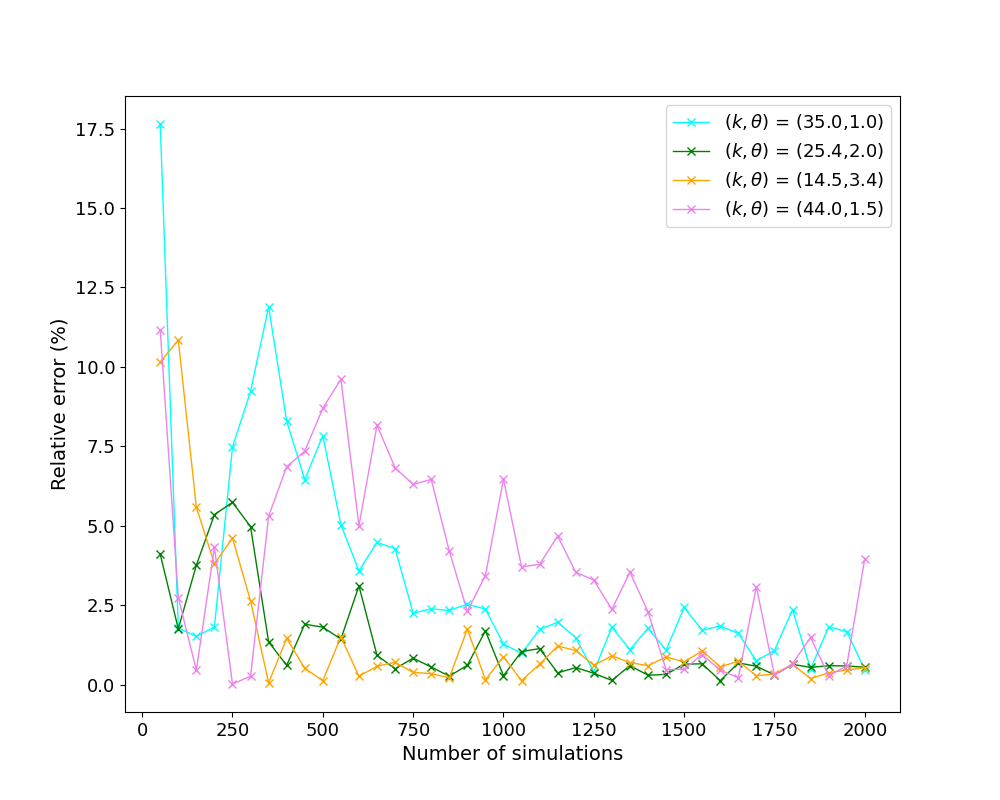}
		\caption{For $\sigma/\mu$.}
		\label{fig:Gaussian_number_simulations_coeff_var}
	\end{subfigure}
	\caption{Relative error of the estimation versus the number of simulations for the inference in the Gaussian regime.}
	\label{fig:Gaussian_number_simulations}
\end{figure}
\noindent To do this estimation, we used $2000$ simulations. For applications on real data, we rarely have so many simulations. Let us check that the estimation still works even if we use fewer simulations. We use the pipeline to estimate $(\mu,\frac{\sigma}{\mu})$ for $n_{\text{data}}\in\left\{ 50l\,|\,l\in\llbracket1,40\rrbracket\right\}$, and then we plot the relative error of the estimation versus $n_{\text{data}}$ in Figure \ref{fig:Gaussian_number_simulations}. We see that with a small number of simulations, even if the score the relative error has increased, scores are 

\begin{figure}[H]
	\centering
	\begin{subfigure}[t]{0.325\textwidth}
		\centering
		\includegraphics[scale = 0.25]{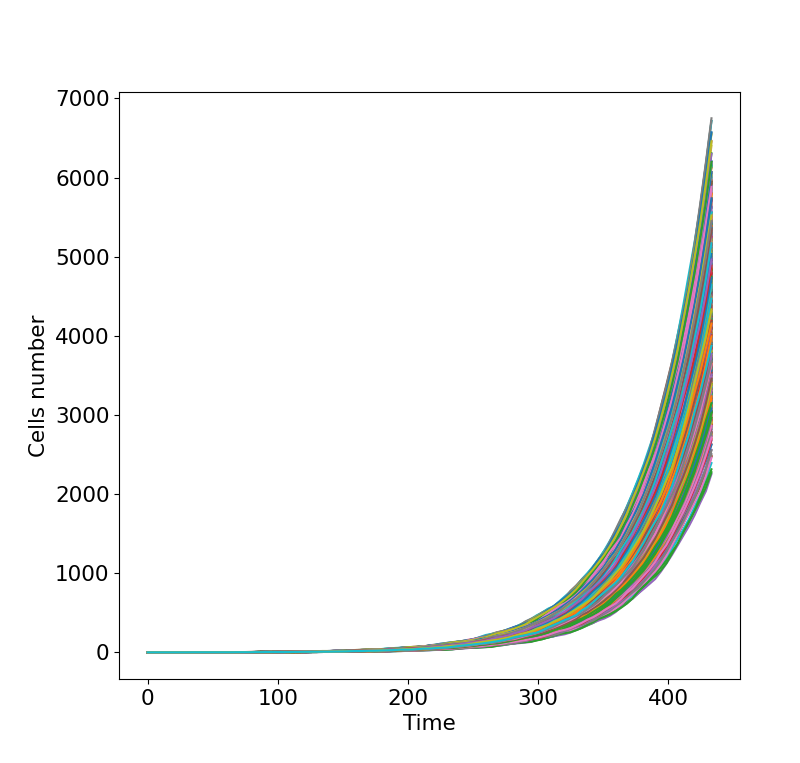}
		\caption{Number of individuals along time.}
	\end{subfigure}
	\hfill
	\begin{subfigure}[t]{0.325\textwidth}
		\centering
		\includegraphics[scale = 0.25]{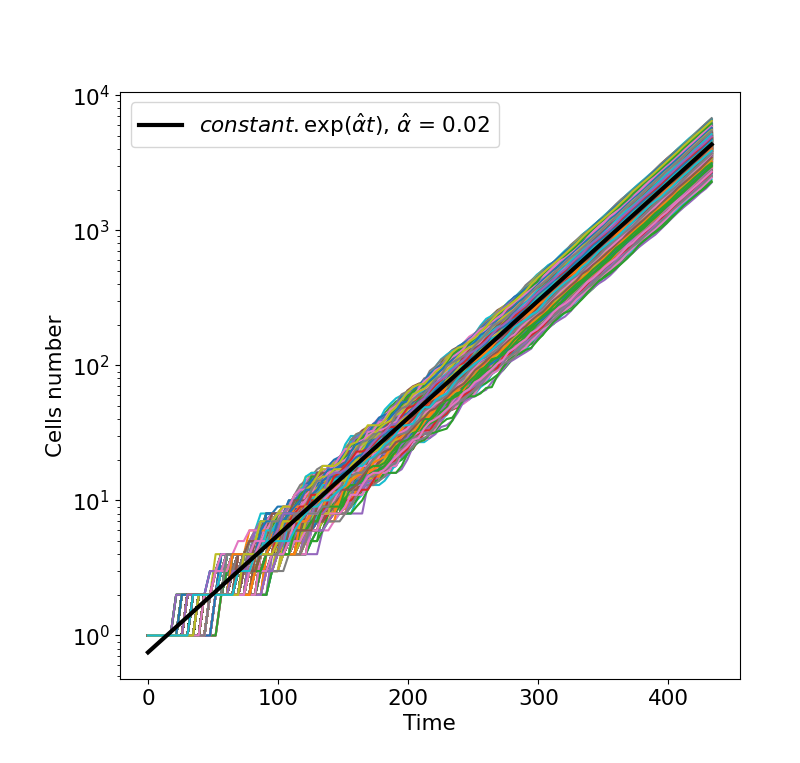}
		\caption{Number of individuals along time, at the log-scale. \it{The mean of the coefficients of determination of the linear regressions done to obtain $\widehat{\alpha}$ is $R_{\text{mean}}^2 = 0.9998$.}}
	\end{subfigure}
	
	\begin{subfigure}[t]{0.325\textwidth}
		\centering
		\includegraphics[scale = 0.25]{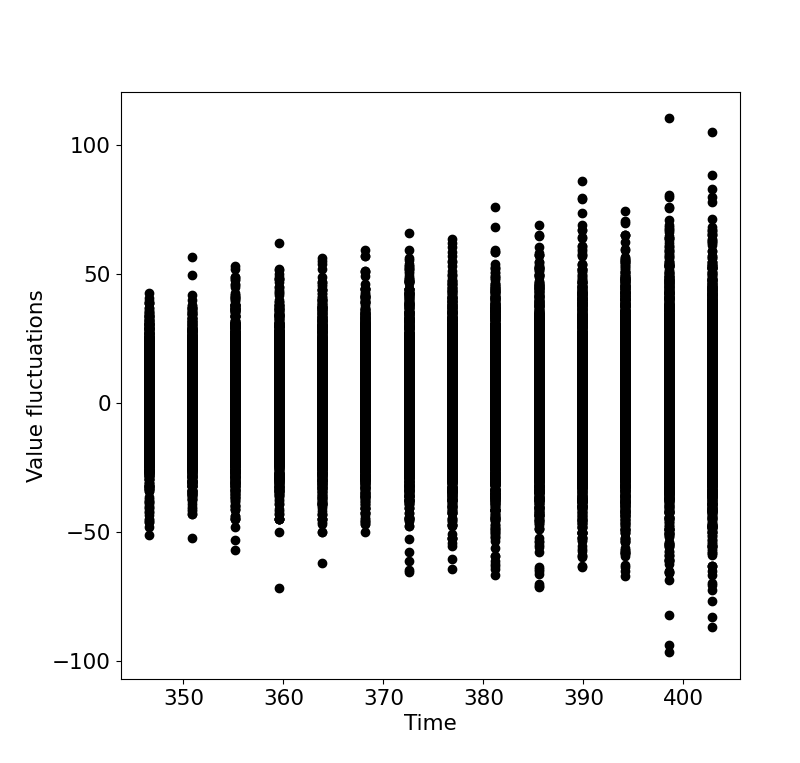}
		\caption{Fluctuations observed on simulations.}
	\end{subfigure}
	\hfill
	\begin{subfigure}[t]{0.325\textwidth}
		\centering
		\includegraphics[scale = 0.25]{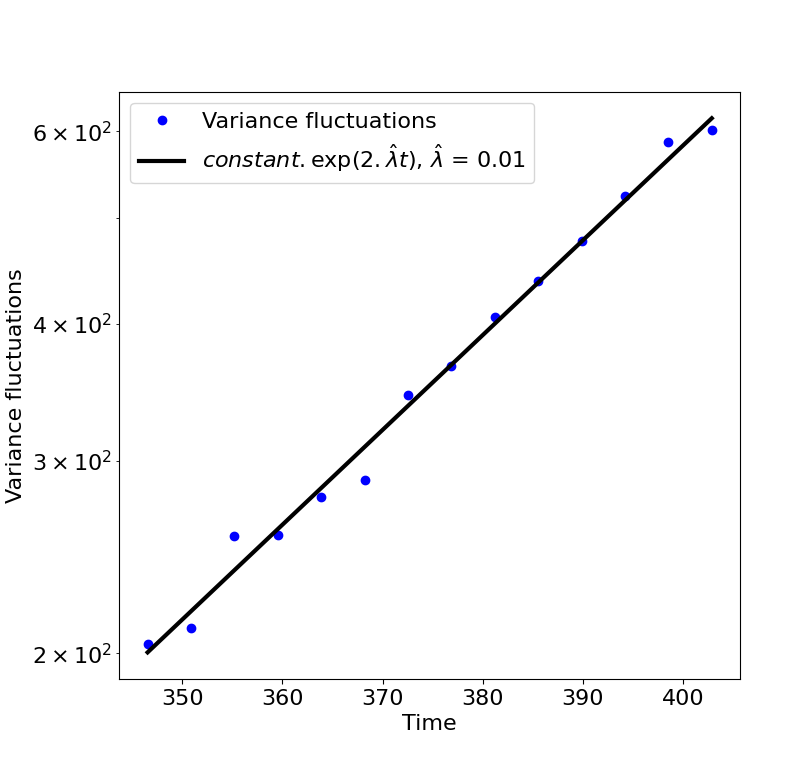}
		\caption{Variance of fluctuations, at the log-scale. \it{The coefficient of determination of the linear regression done to obtain $2\widehat{\lambda}$ is $R^2 = 0.9996$.}} 
	\end{subfigure} 
	
	\begin{subfigure}[t]{0.325\textwidth}
		\centering
		\includegraphics[scale = 0.25]{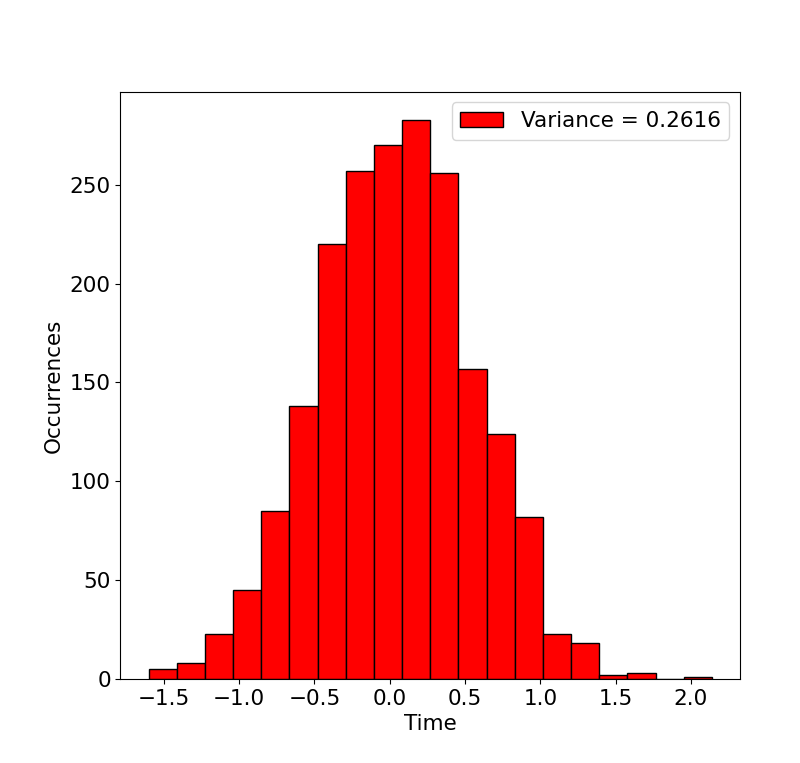}
		\caption{Histogram of the fluctuations at the last observation time.} 
	\end{subfigure}
	\caption{Illustration of the different steps for the inference with 2000 simulations, for the parameters $(k_1, \theta_1) = (35, 1)$.}
	\label{fig:Gaussian_illustration_pipeline}
\end{figure}

\noindent still correct considering the fact that we do not have a lot of simulations. Indeed, for the estimation of $\mu$, even with a small number of simulations, all the relative errors are below $1\%$ as we see in Figure \ref{fig:Gaussian_number_simulations_mean}. For the estimation of $\sigma/\mu$, that corresponds to Figure~\ref{fig:Gaussian_number_simulations_coeff_var}, quickly we are under $10\%$ of relative error. Thus, the inference is satisfying for this regime.

\subsection{Inference in the oscillating regime}\label{subsect:inference_oscillating_regime} 
Let us now consider  the oscillating regime. By Figure \ref{fig:thresold_spectral_gap}, we know that we are in this regime when $k > k_c$. Thus, we study examples where this condition is satisfied. We run $n_{\text{data}} =  2000$ simulations of Bellman-Harris processes with lifetimes distributed according to $\Gamma(k_l,\theta_l)$, with $l\in\llbracket1,4\rrbracket$ and~\hbox{$(k_1,\theta_1) = (70,1)$}, $(k_2,\theta_2) = (125,2)$, $(k_3,\theta_3) = (200.5,1)$, $(k_4,\theta_4) = (385.5,4)$. Then, we create a dataset from these  simulations. The simulations and the creation of the dataset are done in the same way as in the Gaussian regime, see Section \ref{subsubsect:quality_gaussian}.

We use the pipeline described at the beginning of this section to recover the lifetime parameters and then the mean and coefficient of variation of the distribution, see Figure~\ref{fig:oscillating_illustration_pipeline}. The time steps, threshold and time measurements we choose for that are the same as those presented in Section~\ref{subsubsect:summary_choices_hyperparameters}. The scores are displayed in Table \ref{tab:results_oscillating}. We see that the estimation of $(\mu,\frac{\sigma}{\mu})$ is very good, except in the case where $(k_4,\theta_4) = (385.5, 2.5)$ for which the score is not as satisfying for~$\frac{\sigma}{\mu}$. The reason is that when $k$ is very high, the amplitude of the oscillations is very large. Then, the linear regression done to obtain $\widehat{\lambda}$ is less reliable, as illustrated in Figure~\ref{fig:variance_fluctuations_large_amplitude}. For our motivation, cases with such low variability in cells lifetime are unrealistic, so this is not a significant issue. 

\begin{table}[!ht]
	\begin{center}\scalebox{0.7}{
			\begin{tabular}{||c | c |  c |  c | c ||} 
				\hline
				Theoretical values of   $(k,\theta)$ & $(70, 1)$ & $(125, 2)$ & $(200.5, 3)$ & $(385.5, 2.5)$  \\  \hline
				$(k,\theta)$ Param. inferred & $(77.78, 0.9017)$ & $(131.2, 1.911)$ & $(206.4, 2.931)$ & $(519.8, 1.871)$ \\ 
				$(k,\theta)$ Relative error & $(11.11\%, 9.835\%)$ & $(4.947 \%, 4.444\%)$ & $(2.940 \%, 2.315\%)$ & $(34.83 \%, 25.15\%)$  \\ 
				\hline\hline 
				Theoretical values of  $(\mu,\frac{\sigma}{\mu})$ & $(70, 0.1195)$ & $(250, 0.08944)$ & $(601.5, 0.7062)$ & $(963.8, 0.05093)$  \\  \hline
				$(\mu,\frac{\sigma}{\mu})$ Param. inferred & $(70.13, 0.1134)$ & $(250.7, 0.08731)$ & $(604.8, 0.06961)$ & $(972.7, 0.04386)$ \\  
				$(\mu,\frac{\sigma}{\mu})$ Relative error & $(0.1803\%, 5.130\%)$ & $(0.2833 \%, 2.385\%)$ & $(0.5562 \%, 1.438\%)$ & $(0.9251 \%, 13.88\%)$  \\   \hline
		\end{tabular}}
		\caption{Results of the estimation in the oscillating regime. \it{The results are given with $4$ significant digits.}}\label{tab:results_oscillating}
	\end{center}
\end{table}
\begin{figure}[!ht]
	\centering
	\begin{subfigure}[t]{0.475\textwidth}
		\centering
		\includegraphics[scale = 0.25]{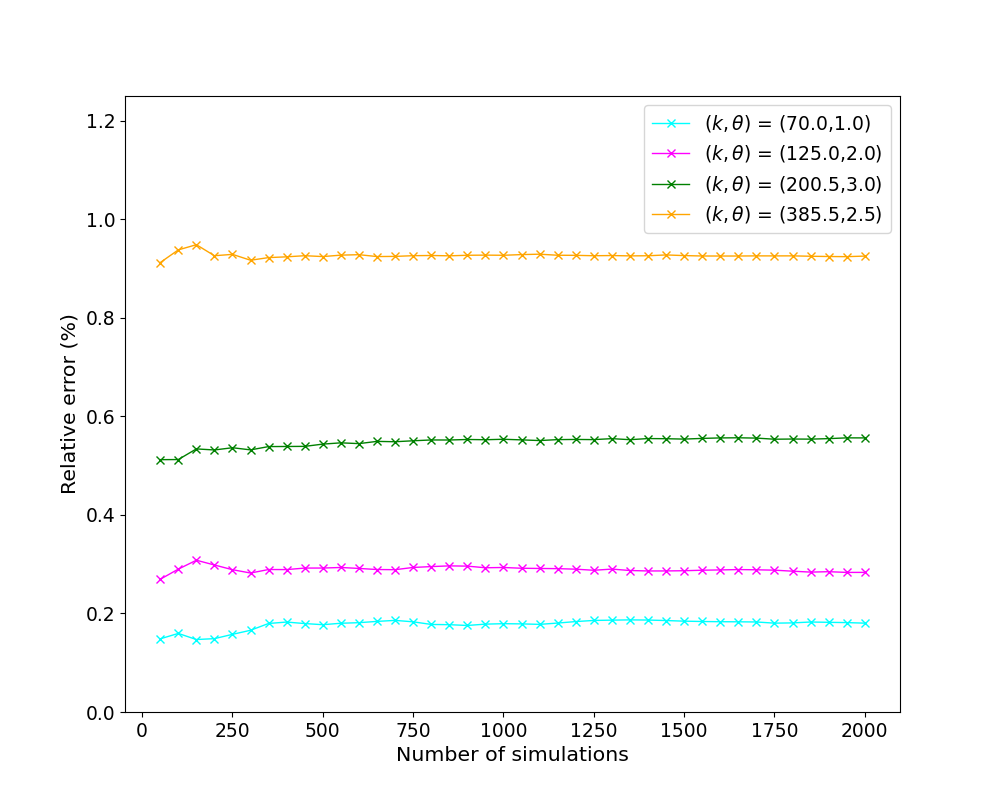}
		\caption{For $\mu$.}
	\end{subfigure}
	\hfill
	\begin{subfigure}[t]{0.475\textwidth}
		\centering
		\includegraphics[scale = 0.25]{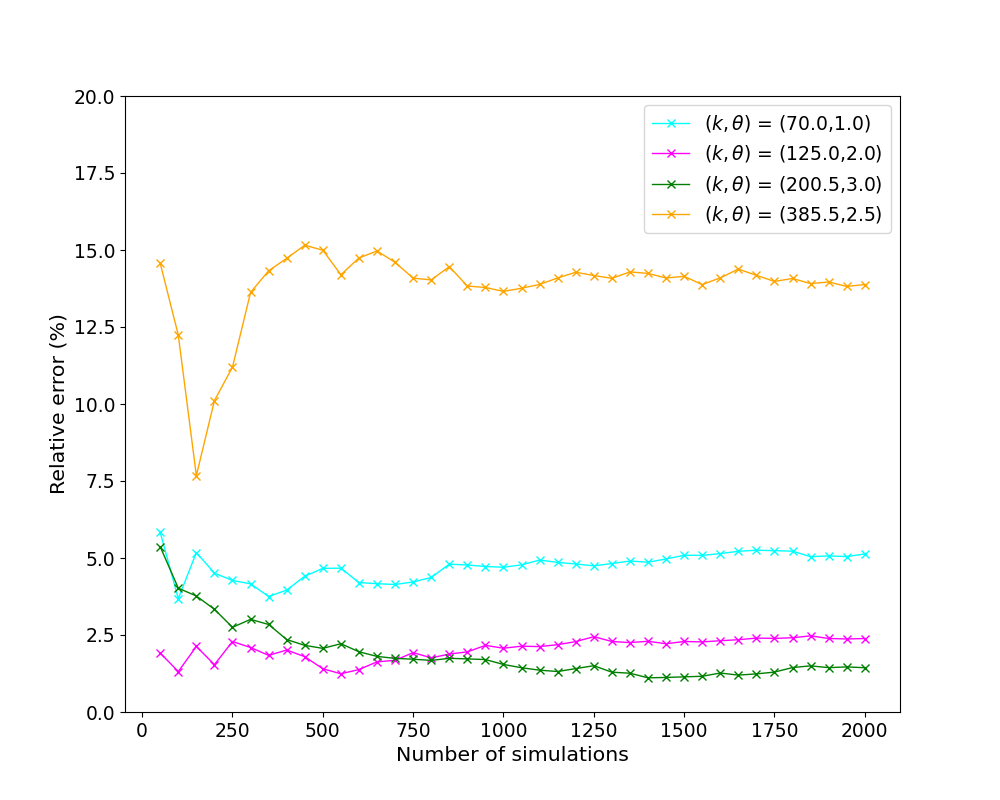}
		\caption{For $\sigma/\mu$.}
	\end{subfigure}
	\caption{Relative error of the estimation versus the number of simulations for the inference in the oscillating regime.}
	\label{fig:evol_n_simuls_oscillating}
\end{figure}

To see the effect of the number of simulations, we plot the evolution of the relative error for $n_{\text{data}}\in\left\{ 50l\,|\,l\in\llbracket1,40\rrbracket\right\}$ in Figure \ref{fig:evol_n_simuls_oscillating}. Again, the results are good except for $(k_4,\theta) = (385.5,2.5)$. The number of simulations seems to have a little influence on the relative error. Thus, the inference procedure seems very reliable for cases that motivate our study, and less precise for unrealistic cases with oscillations of large amplitude.

An estimation of the second eigenvalue based on the amplitude or on the frequency of the oscillations rather than a linear regression could improve the results in cases with large oscillations. As this is less relevant than a linear regression for the cases that motivate our study, we have not developed this type of estimation for the moment.

\begin{figure}[!ht]
    \centering
    \begin{subfigure}[t]{0.475\textwidth}
        \centering
        \includegraphics[scale = 0.3]{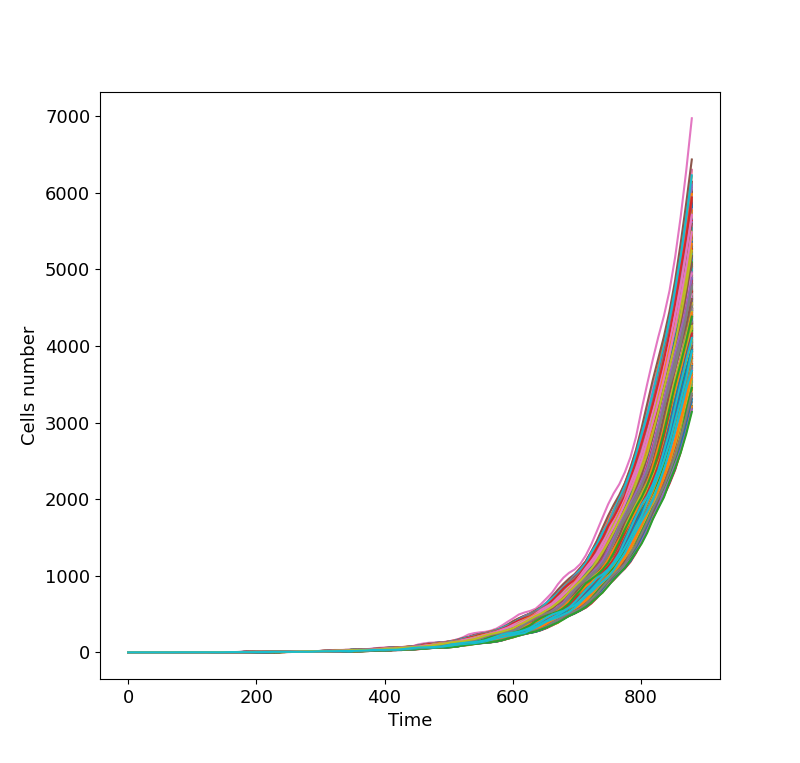}
        \caption{Number of individuals along time.}
    \end{subfigure}
    \hfill
    \begin{subfigure}[t]{0.475\textwidth}
        \centering
        \includegraphics[scale = 0.3]{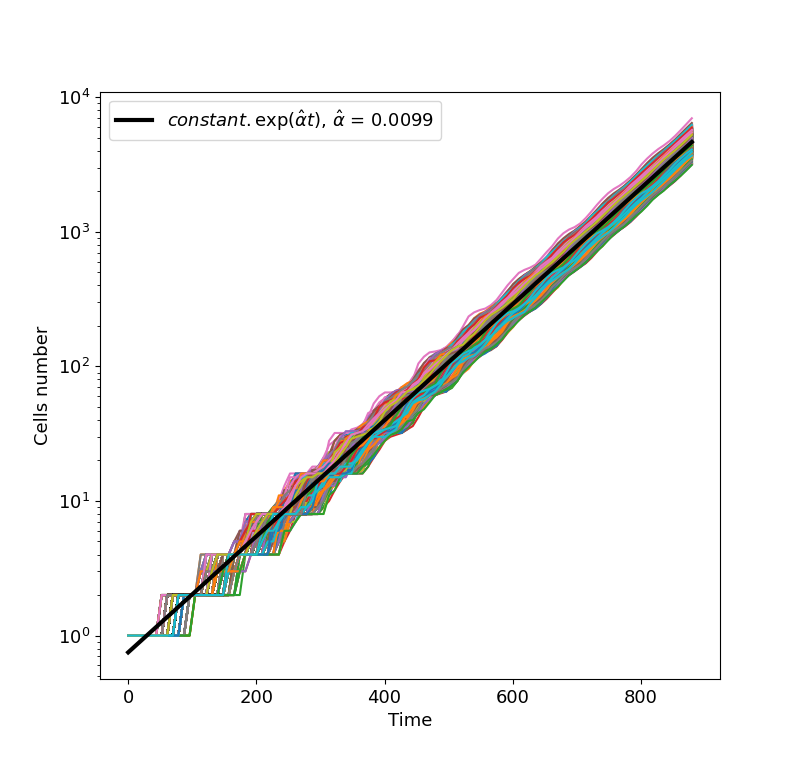}
        \caption{Number of individuals along time, at the log-scale. \it{The mean of the coefficients of determination of the linear regressions done to obtain $\widehat{\alpha}$ is $R_{\text{mean}}^2 = 0.9995$.}}
    \end{subfigure}
    
    \begin{subfigure}[t]{0.475\textwidth}
        \centering
        \includegraphics[scale = 0.3]{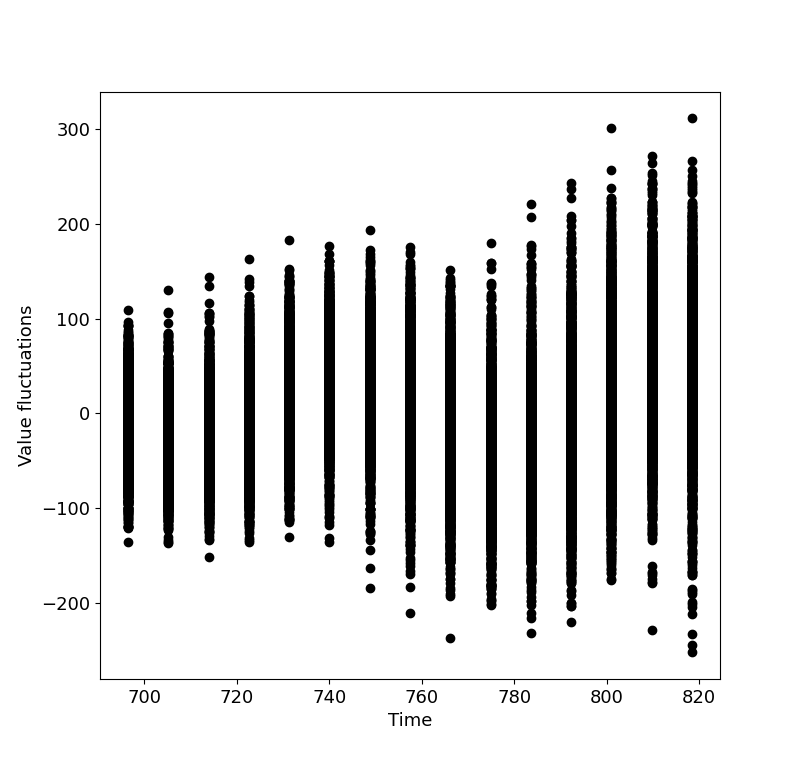}
        \caption{Fluctuations observed on simulations.}
    \end{subfigure}
    \hfill 
    \begin{subfigure}[t]{0.475\textwidth}
        \centering
        \includegraphics[scale = 0.3]{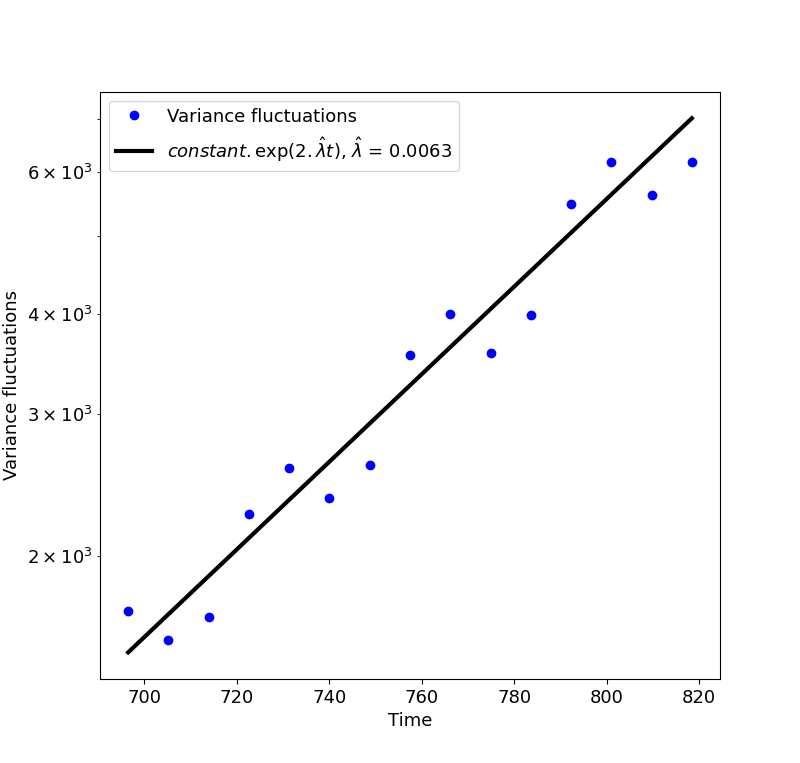}
        \caption{Variance of fluctuations, at the log-scale. \it{The coefficient of determination of the linear regression done to obtain $\widehat{\lambda}$ is $R^2 = 0.9926$.}}
    \end{subfigure}
    \caption{Illustration of the different steps for the inference with 2000 simulations, for the parameters $(k_1, \theta_1) = (70, 1)$.}
    \label{fig:oscillating_illustration_pipeline}
\end{figure}

\begin{figure}[!ht]
    \centering
    \begin{subfigure}[t]{0.485\textwidth}
        \centering
        \includegraphics[scale = 0.3]{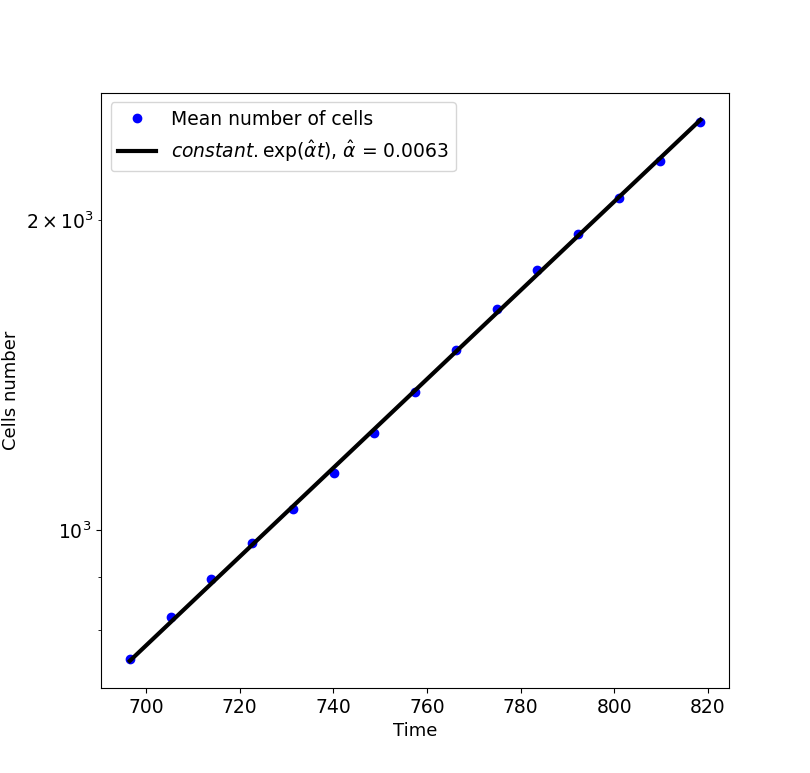}
        \caption{Mean number of cells versus time, $(k_1,\theta_1) = (70,1)$. \it{The mean of the coefficients of determination of the linear regressions done to obtain $\widehat{\alpha}$ is $R_{\text{mean}}^2 = 0.9995$.}}\label{fig:mean_cells_small_amplitude}
    \end{subfigure}
    \hfill
    \begin{subfigure}[t]{0.485\textwidth}
        \centering
        \includegraphics[scale = 0.3]{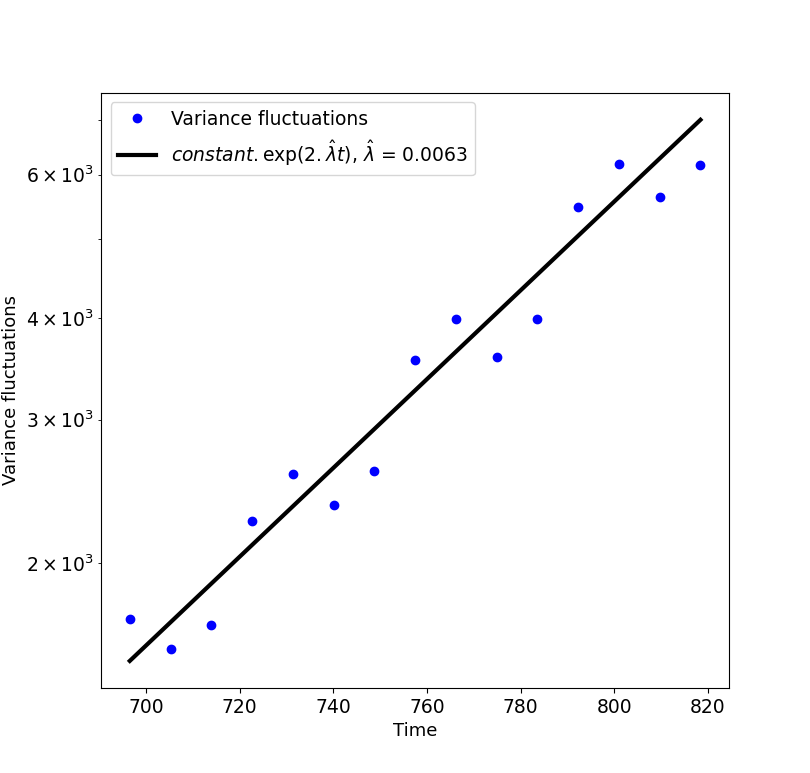}
        \caption{Mean number of cells versus time, $(k_4,\theta_4) = (385.5,2.5)$. \it{The mean of the coefficients of determination of the linear regressions done to obtain $\widehat{\alpha}$ is $R_{\text{mean}}^2 = 0.9920$.}}\label{fig:mean_cells_large_amplitude}
    \end{subfigure}
    
    \begin{subfigure}[t]{0.485\textwidth}
        \centering
        \includegraphics[scale = 0.3]{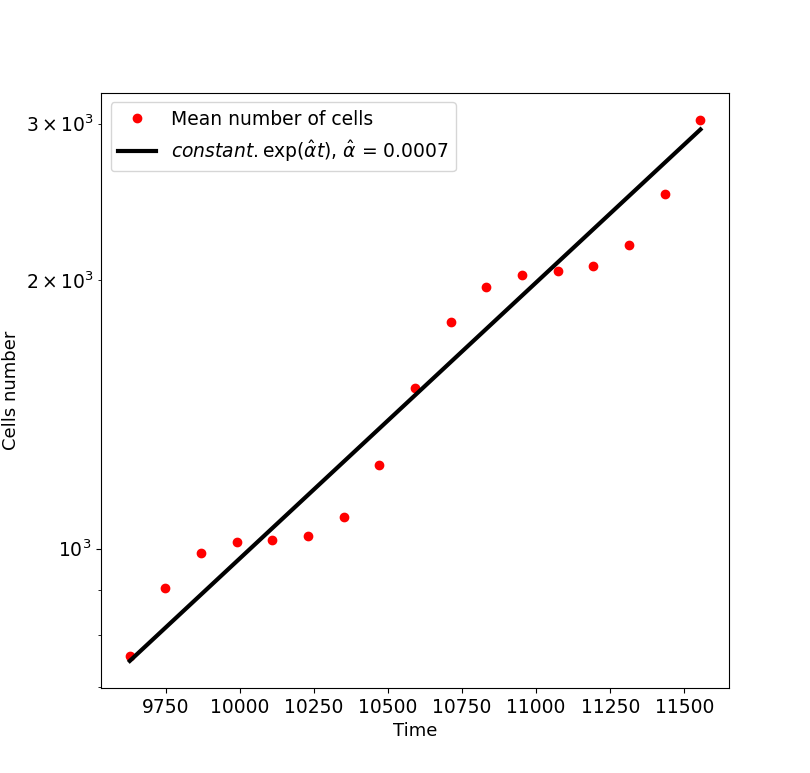}
        \caption{Variance fluctuations versus time, $(k_1,\theta_1) = (70,1)$. \it{The coefficient of determination of the linear regression done to obtain $\widehat{\lambda}$ is $R^2 = 0.9926$.}}
        \label{fig:variance_fluctuations_small_amplitude}
    \end{subfigure}
    \hfill 
    \begin{subfigure}[t]{0.485\textwidth}
        \centering
        \includegraphics[scale = 0.3]{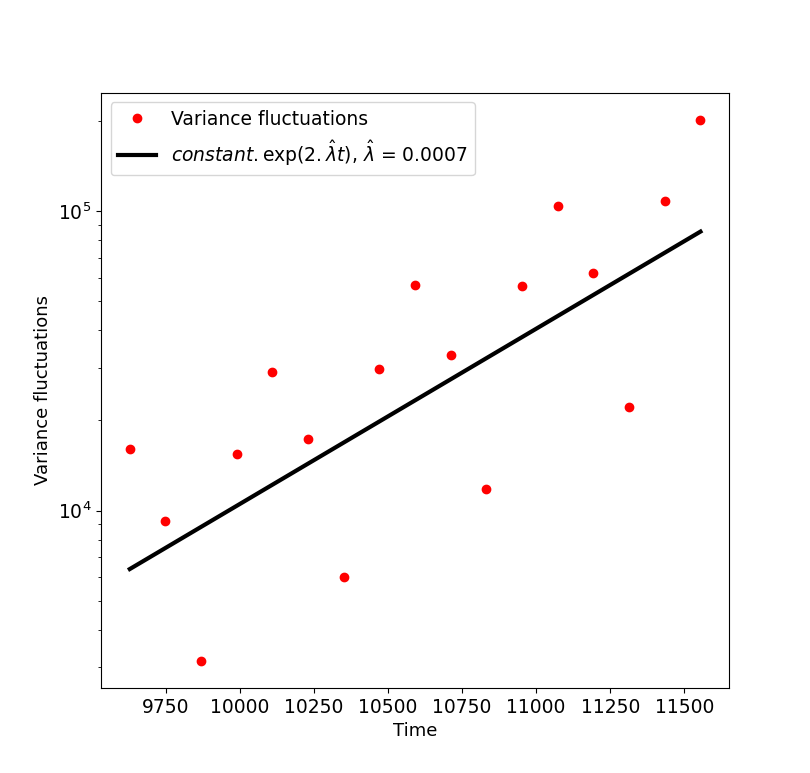}
        \caption{Variance fluctuations versus time, $(k_4,\theta_4) = (385.5,2.5)$. \it{The coefficient of determination of the linear regression done to obtain $\widehat{\lambda}$ is $R^2 = 0.8613$.}}\label{fig:variance_fluctuations_large_amplitude}
    \end{subfigure}
    \caption{Comparison at the log-scale of the mean number of cells and variance of fluctuations, with exponential curves whose coefficients are respectively $\widehat{\alpha}$ and~$2\widehat{\lambda}$, for $(k_1,\theta_1) = (70,1)$ and $(k_4,\theta_4) = (385.5,2.5)$. \it{Mean and variance are computed using empirical estimators with $2000$ simulations.}}
    \label{fig:comparison_linear_regressions}
\end{figure}
%\newpage 

%\ju{We conclude this section by summing up the choices of the time steps and threshold used in our protocol of estimation. There are three statistics, and the question of how to choose them was studied in Sections~\ref{subsubsect:step_detection_regime},~\ref{subsubsect:threshold_detection} and~\ref{subsubsect:step_identifiability}.}

 %\vinc{Ca ferait sens d'évoquer rapidement ce qui arrive pour d'autres parametres : a t on explorer plus ? est ce que ca marche pareil alors ?} \ju{fait avec les nouveaux params, les anciens étaient pas bien}
%\begin{figure}[!ht]
%    \centering
%    \includegraphics[scale = 0.3]{images/simus/oscillating-inference/relatives_error_oscillating_regime.jpg}
%    \caption{Relative error versus number of simulations for the inference of $\mu$ (left) and $\sigma/\mu$ (right) in the oscillating regime.}
%    \label{fig:evol_n_simuls_oscillating}
%\end{figure}

%\begin{figure}[!ht]
%    \centering
%    \includegraphics[scale = 0.4]{images/simus/oscillating-inference/surcritical_curves.png}
%    \caption{ Figure A plots the number of individuals along time.   Figure B represents the same as the latter, at the log-scale. Figure C shows for each realization, the fluctuations observed on simulations, gathering the successive times. Figure D finally shows the corresponding variance.}
%\end{figure}

\section{Estimation of parameters from data}
\label{desdata}
Our original motivation for this work is the characterization of the variability of time division  using monitoring of population sizes. Let us apply our work to two dataset counting dividing cells. All the quantities given below are given with $3$ significant numbers.

\subsection{First data set} The first dataset we use comes from %Ref.
~\cite{coates_antibiotic-induced_2018}. It contains $13$ realizations of the growth of E. Coli bacteria, obtained with videomicroscopy, where the exact number of cells is counted in each frame. Measurements start after $48$ minutes, and are done every $15$ minutes, see Figure~\ref{fig:data_Sylvain_cells_number}. Since the cells are grown in a very large recipient, we assume that they do not compete for resources, which is why they remain in exponential phase until the end of the measurements.

%We denote $T_{\text{measure}} = \{48,63,78,93,108,123,138,153\}$ the times where the measurements are done, and for each realisation $i\in \llbracket 1,13 \rrbracket$ we denote $(N^{(i)}_{t})_{t\in T_{\text{measure}}}$ the measures associated to this realisation. For all $t\in T_{\text{measure}}$ such that $t\neq 153$, $i \in \{48,63,78,93,108,123,138,153\}$, we denote the fluctuation $R_t^{(i)} = $.

%\begin{figure}[!ht]
%    \centering
%    \includegraphics[scale = 0.5]{images/experimental/data-Sylvain/cells_vs_times_Sylvain_control.png}
%    \caption{Number of cells versus time for the first dataset}
%\end{figure}
%\begin{figure}[!ht]
%    \centering
%    \includegraphics[scale = 0.32]{images/experimental/data-Sylvain/first_estimation_curves.jpg}
%    \caption{Logarithm of number of cells versus time with a straight line with slope $\widehat{\alpha}_1 $(left) and variance of fluctuations versus times (right) for the first dataset }
%\end{figure}
We assume that the lifetime of these bacteria is distributed according to a Gamma distribution. Our aim is to estimate the mean and the coefficient of variation of bacteria lifetime with our method. First, we estimate the Malthusian coefficient and determine in which regime we are. The estimation of $\alpha$ gives us $\widehat{\alpha}_1 = 0.0262\text{ minutes}^{-1}$, see Figure~\ref{fig:data_Sylvain_log_cells_number}. As we have $\log(2)/(2.\widehat{\alpha}_1) = 13.3\text{ minutes}$ and that $\Delta = 15 \text{ minutes}$, we use the time step $\delta_1 = 15\text{ minutes}$ to identify the regime. The result of the estimation of $\lambda$ is \hbox{$\widehat{\lambda}_1=0.0188\text{ minutes}^{-1}$}, see Figure~\ref{fig:data_Sylvain_variance_fluctuations}. Then,  
$$\frac{\left|2\widehat{\lambda}_1 - \widehat{\alpha}_1\right|}{\widehat{\alpha}_1} = 44.1\% >10\%$$
and we have the oscillating regime.
We now estimate lifetime parameters. We use the estimator for the oscillating regime. Here are the parameters we obtain
$$
\begin{aligned}
(\widehat{k}_1,\widehat{\theta}_1) &= \underset{\substack{(k,\theta)\in[1,+\infty[\times\mathbb{R}_+^*  \\ (2^{1/k} - 1)/\theta = \widehat{\alpha}_1}}{\text{argmin}} \left|(2^{1/k}\cos\left(2\pi/k\right)-1)/\theta - \widehat{\lambda}_1\right| = \left(102, 0.260\text{ minutes}\right).
\end{aligned}
$$
The latter corresponds to the following mean and coefficient of variation
$$
\begin{aligned}
\left(\widehat{\mu}_1,\frac{\widehat{\sigma}_1}{\widehat{\mu}_1}\right) =  \left(\widehat{k}_1\widehat{\theta}_1,\frac{1}{\sqrt{\widehat{k}_1}}\right) = \left(26.6\text{ minutes},9.893\%\right).
\end{aligned}
$$
Although the coefficient of variation is a little low, the values estimated seem fairly credible. %Thus, the inference seems to have worked.
\begin{figure}[!ht]
    \centering
    \begin{subfigure}[t]{0.475\textwidth}
        \centering
        \includegraphics[scale = 0.3]{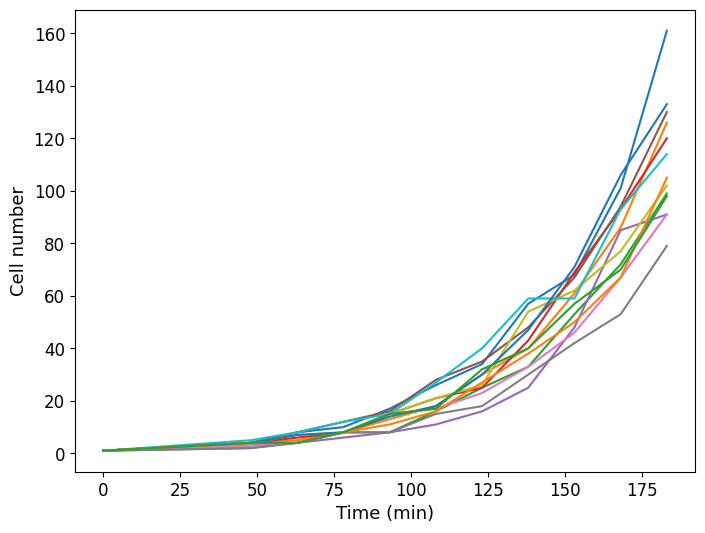}
        \caption{Cell number at each time of measurement.}
        \label{fig:data_Sylvain_cells_number}
    \end{subfigure}
    \hfill
    \begin{subfigure}[t]{0.475\textwidth}
        \centering
        \includegraphics[scale = 0.3]{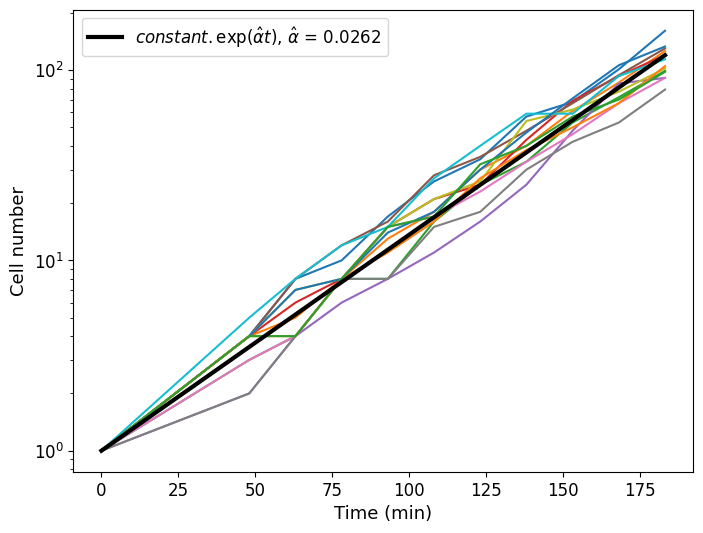}
        \caption{Cell number at each time of measurement at the log-scale. \it{The mean of the coefficients of determination of the linear regressions done to obtain $\widehat{\alpha}_1$ is $R_{\text{mean}}^2 = 0.986$.}}
        \label{fig:data_Sylvain_log_cells_number}
    \end{subfigure}
    
    \begin{subfigure}[t]{0.475\textwidth}
        \centering
        \includegraphics[scale = 0.3]{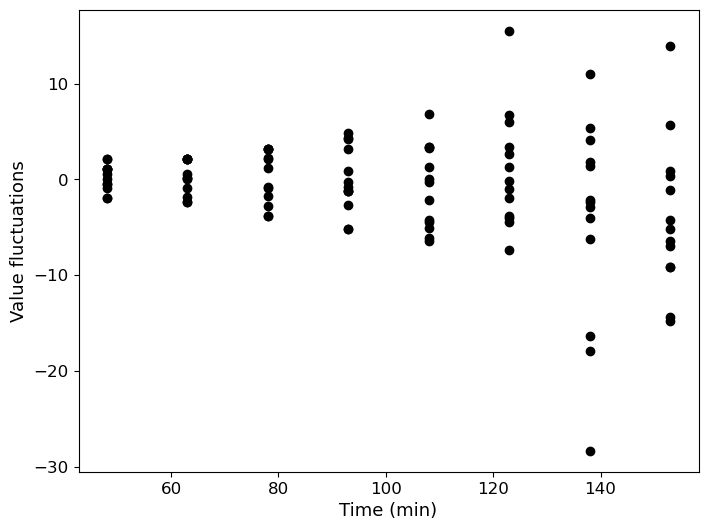}
        \caption{Fluctuations $R_{t,\delta_1}$ at each time of measurement.}
        \label{fig:data_Sylvain_values_fluctuations}
    \end{subfigure}
    \hfill 
    \begin{subfigure}[t]{0.475\textwidth}
        \centering
        \includegraphics[scale = 0.3]{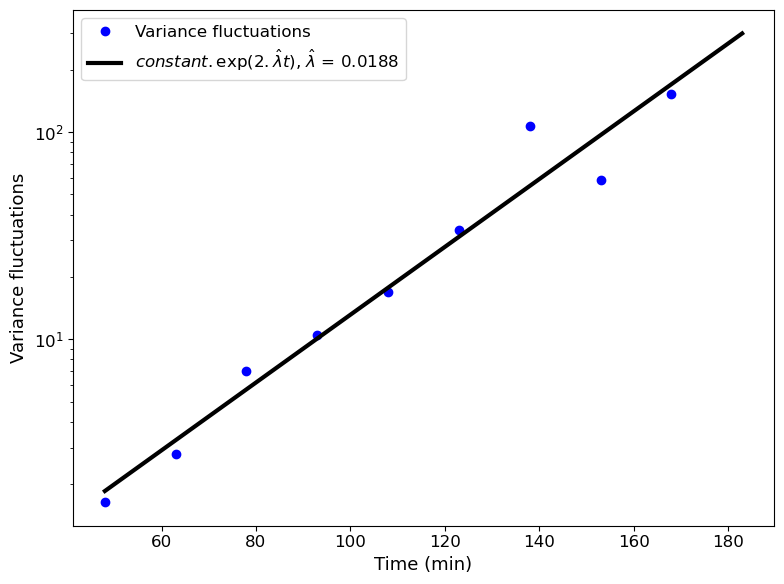}
        \caption{Variance of the fluctuations at each time of measurement, at the log-scale. \it{The coefficient of determination of the linear regression done to obtain $\widehat{\lambda}_1$ is $R^2 = 0.960$.}}    \label{fig:data_Sylvain_variance_fluctuations}
    \end{subfigure}
    \caption{Illustration of the different steps done to estimate the parameters of cells lifetime for the dataset coming from \cite{coates_antibiotic-induced_2018}.}
    \label{fig:data_Sylvain}
\end{figure}
%\begin{figure}[!ht]
%    \centering
%    \includegraphics[scale = 0.4]{images/experimental/data-Sylvain/données_Sylvain_courbes.png}
%    \caption{Illustration of the different steps done to estimate the parameters of cells lifetime for the dataset coming from \cite{coates_antibiotic-induced_2018}. Figure A represents the measurements of the cells number versus the time.  Figure B represents the same as the latter, at the log-scale. Figure C shows for each time of measurement and each realization, the fluctuations we have. Figure D finally shows the variance of the fluctuations, at each time of measurement.}\label{fig:data_Sylvain}
%\end{figure}
\subsection{Second data set}
The second dataset we use comes from experiments conducted by the authors of the paper. Such data set is the original motivation for our study, see in particular \cite{amselem2016universal,barizien_studying_2019}. It contains $347$ realizations of the growth of bacteria E.Coli, obtained by encapsulating the original sets within anchored microfluidic droplets, see~Figure~\ref{fig:data_Charles_number_cells}. The drops are then placed on the stage of a motorized microscope and scanned at regular time intervals. Image analysis is then used to obtain the number of cells per droplet, as explained in %Ref.
~\cite{quellec2023measuring}. Here, the cells are grown in so-called MOPS minimal media, which explains their slow growth compared with typical cells growing in rich LB culture media~\cite{quellec2023measuring}. The measurements are done every $30$ minutes, then $\Delta = 0.5$ hours, and the number of bacteria is represented by an ``intensity of fluorescence" in this dataset. It has been shown that there is a coefficient of proportionality between the intensity of fluorescence and the number of cells \cite{barizien_studying_2019}, although up to our knowledge the value of this coefficient has not yet been well estimated. 

As for the first set, we want to obtain the mean and the coefficient of variation of the times between division of the bacteria thanks to our method. As the relationship between the number of cells and the fluorescence intensity is linear, the eigenvalues remain the same in both settings. Thus, if we detect the oscillating regime, the inference is possible. The first thing we do is to estimate the Malthusian coefficient, and identify in which regime we are. We obtain thanks to linear regression that $\widehat{\alpha}_2 = 0.450 \text{ hours}^{-1}$, see~Figure~\ref{fig:data_Charles_log_number_cells}. As $\log(2)/(2\widehat{\alpha}_2) = 0.770 \text{ hours}$ and that $\Delta = 0.5$ hours, we use $\delta_1 = 1$ hour to determine the regime. %Taking $\delta_1 = 0.5$ hour might also work, 
It is not a problem that $\delta_1$ is not close to $\log(2)/(2\widehat{\alpha}_2)$. The more important is to avoid to be too close to $\log(2)/\widehat{\alpha}_2$, see Section~\ref{subsubsect:step_detection_regime}, which is the case here.
Our estimation of the second eigenvalue is $\widehat{\lambda}_2 =  0.343 \text{ hours}^{-1}$, see Figure~\ref{fig:data_Charles_variance_fluctuations}. We detect again the oscillating regime, because the relative error between $2\widehat{\lambda}_2$ and $\alpha$ is~$52.4\%$, which is strictly greater than $10\%$. Then, the formula used for the first data set now gives
%for this data set we obtain
%using the formula we have used for the previous dataset yields
$$
(\widehat{k}_2,\widehat{\theta}_2) = (120.1,\,0.0129 \text{ hours}),
$$
which corresponds to the following mean and coefficient of variation 
$$
\left(\widehat{\mu}_2,\frac{\widehat{\sigma}_2}{\widehat{\mu}_2}\right) =  \left(\widehat{k}_2\widehat{\theta}_2,\frac{1}{\sqrt{\widehat{k}_2}}\right) = \left(1.55 \text{ hours},9.12\%\right).
$$
The value of $\mu_2$ obtained in this case is much higher than the one for the first dataset, in agreement with the slower growth rate observed in Figure~\ref{fig:data_Charles}. This is due to the different nature of the cell culture media. The estimated coefficient of variation appears to be fitting and relevant in this context. One can also see that we obtain a coefficient of variation similar to the one obtained in the previous dataset. This seems rather relevant.
%Consequently, our estimation of variability also seems to work for this dataset.

\begin{figure}[!ht]
    \centering
    \begin{subfigure}[t]{0.475\textwidth}
        \centering
        \includegraphics[scale = 0.3]{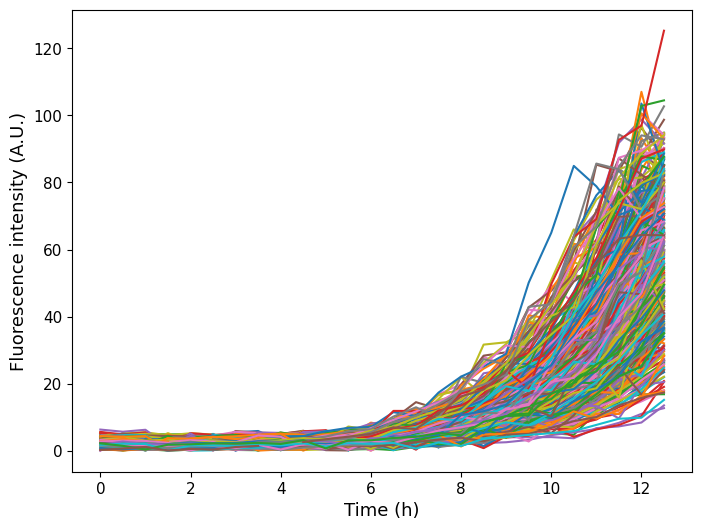}
        \caption{Fluorescence intensity at each time of measurement.}
        \label{fig:data_Charles_number_cells}
    \end{subfigure}
    \hfill
    \begin{subfigure}[t]{0.475\textwidth}
        \centering
        \includegraphics[scale = 0.3]{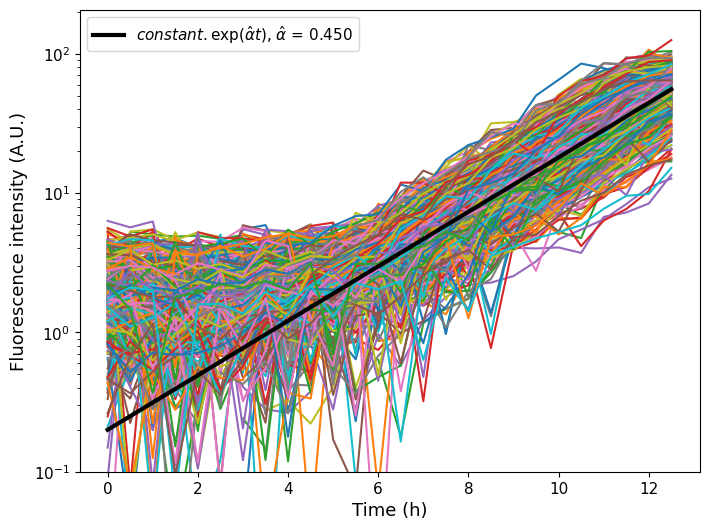}
        \caption{Fluorescence intensity at each time of measurement, at the log-scale. \it{The mean of the coefficients of determination of the linear regressions done to obtain $\widehat{\alpha}_2$ is $R_{\text{mean}}^2 = 0.926$.}}
        \label{fig:data_Charles_log_number_cells}
    \end{subfigure}
    
    \begin{subfigure}[t]{0.475\textwidth}
        \centering
        \includegraphics[scale = 0.3]{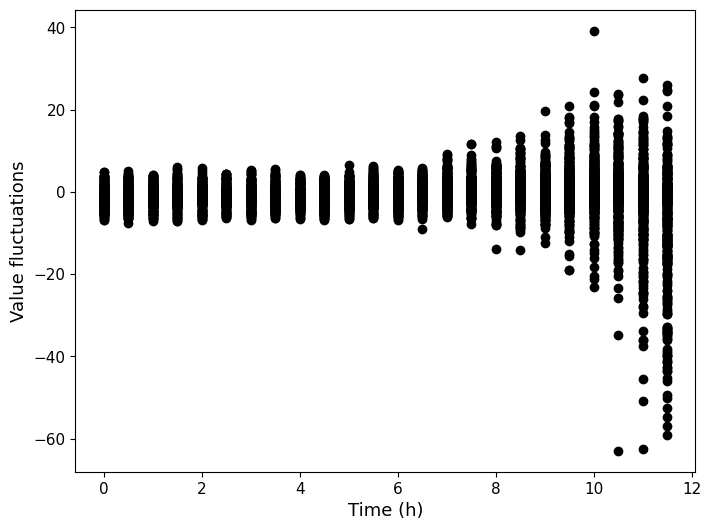}
        \caption{Fluctuations $R_{t,\delta_1}$ at each time of measurement.}
        \label{fig:data_Charles_fluctuations}
    \end{subfigure}
    \hfill 
    \begin{subfigure}[t]{0.475\textwidth}
        \centering
        \includegraphics[scale = 0.3]{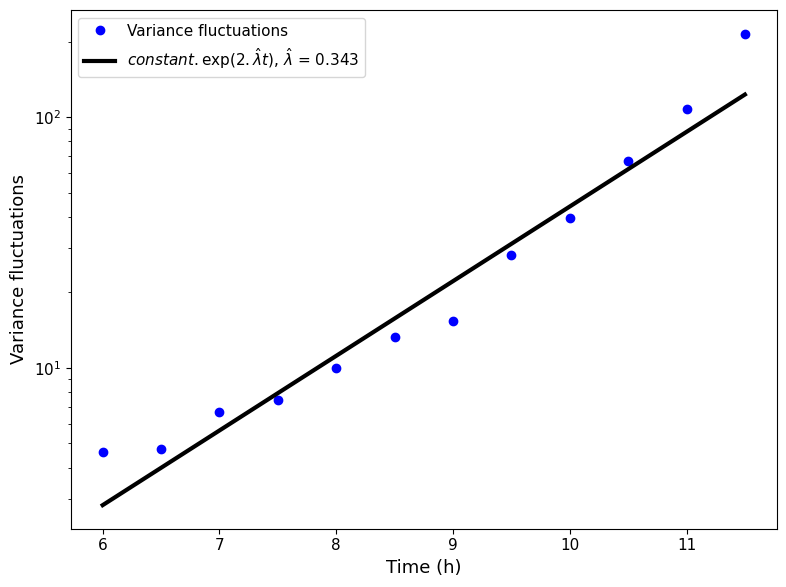}
        \caption{Variance of the fluctuations at each time of measurement, at the log-scale. \it{The coefficient of determination of the linear regression done to obtain $\widehat{\lambda}_2$ is $R^2 = 0.954$.}}
        \label{fig:data_Charles_variance_fluctuations}
    \end{subfigure}
    \caption{Illustration of the different steps done to estimate the parameters of cells lifetime for the dataset coming from our experiments.}
    \label{fig:data_Charles}
\end{figure}

%\begin{figure}[ht]
%    \centering
%    \includegraphics[scale = 0.4]{images/experimental/data-Charles 1/données_Charles_courbes.png}
%    \caption{Illustration of the different steps done to estimate the parameters of cells lifetime for the dataset coming from one of the authors. Figure A represents the measurements of the cells number versus the time.  Figure B represents the same as the latter, at the log-scale. Figure C shows for each time of measurement and each realization, the fluctuations we have. Figure D finally shows the variance of the fluctuations, at each time of measurement.}\label{fig:data_Charles}
%\end{figure}
%\newpage
\section{Discussion}
In this work, we show how to obtain the parameters of the law of the lifespan of individuals using   the total population sizes  at given times. We exploit the distribution of fluctuations   from one observation time to the next one by comparing the population size to the value expected from deterministic exponential growth. We model the population by a Bellman-Harris process, which   does not include 
heredity. This  provides a natural and simple statistical framework with independent and identically distributed lifespans.
Despite this simplicity, the estimation of parameters of this distribution is subtle. Indeed, two regimes exist for the asymptotic fluctuations of the process. Determining in which regime  we fall yields a first difficulty. It is worth noting that the critical value between the two regimes correspond to a coefficient variation of around $10\%$. This is realistic for biological problems and for instance for microfluidic data which motivates this work. This motivates a precise mathematical analysis. We propose a procedure  to determine the regime of fluctuations and tune the length of the time between observations.
This procedure  has been confronted to simulations and real data.  For simulations, our results show that the estimation we propose gives an efficient approach. The comparison with biological data also argues for its relevance. Indeed, applying our inference method on two different datasets returns mean and coefficient of variation that are consistent with our biological knowledge.
\\

In many biological contexts (cell growth and cancer emergence, epidemics, invasive species...),   obtaining the mean individual behavior is rather direct from the growth rate of the population size. Our study shows that determining the individual variability  or stochasticity  from measures at the population level is difficult. It puts in light that the choice of the times of observation is crucial. It also pledges for high sampling rate, and measures  at the individual levels rather than  at the population level. This would require us to generate lots of data using high magnification experiments. In contrast, by using the current techniques, we are able to measure the parameters of the division rates with a relatively low sampling time.\\

This work can be seen as a first step for inverse problem and can pave the way for the investigation of  more sophisticated and realistic models. In particular, one may focus on other distribution for the lifespan, even if it is less explicit than the Gamma law. For instance, the distribution of time for cell division is sometimes best described by the~\hbox{$a_0$-shifted} gamma distribution, see \cite{stukalin_age-dependent_2013} and references therein. A more challenging issue would be to take into account heredity. In particular, lots of attention has been paid to the effect of the size and of the variation of size from the previous division (sizer and adder model).  This yields interesting perspectives. In particular, in this direction, robustness study and sensitivity analysis would be relevant. In our pursuit of capturing the variability in lifespan, we may also consider the initial individual differences through a mixed model.

At a more technical level, we expect to complement the results of Section \ref{sect:simus} by more quantitative bounds on the estimation of parameters. In particular, the fact that we first estimate $\alpha$ and then use the value obtained for the estimation of the remaining parameter of the lifespan has been tested and validated by simulations but could be more formally investigated. Besides, we have used a threshold at $10\%$ for the detection of the Gaussian and oscillating regime from the data, based on empirical explorations. More generally, a sensitivity analysis would be interesting  to complement this study and help for potential applications. It will also be valuable to obtain theoretical assurances on the property we have inferred on $\sigma_{\delta}^2$ when $\delta = \log(2)/\alpha$, as the injectivity of this function in the variable~$k$, or the fact that it does not depend on $\alpha$.

%It will also be valuable to obtain more theoretical assurances on the approximation of $h^{(\delta)}$, and in a broader context, on the approximation of $\sigma_{\delta}^2$. As we have an explicit Laplace transform for the expectation of cells number, see Section \ref{subsect:approximation_mean_Bellman-Harris}, one approach would be to look at numerical methods for inverting Laplace transforms \cite{cohen_numerical_2007,kuhlman_review_2013}.

%\vinc{+ mixed effect ? \\Il faut surement aborder le fait qu'en definissant $R_t^{\delta}$ en utilisant $\lambda$, en pratique c'est une estimation de $\lambda$ \`a laquelle on a acces.}
%\charles{The critical value of $k$ is realistic for real biological systems. This justifies the need for a mathematical treatment of the problem. The emerging microfluidic formats, including droplet-based, will provide this type of data.}
%\charles{In order to obtain $\mu$ and (especially) $\sigma$ from experimental data, we need to have a very high sampling rate, in order to detect the individual division times of cells. This would require us to generate lots of data using high magnification experiments. In contrast, by using the current techniques, we are able to measure the parameters of the division rates with a relatively low sampling time}. 

\section{Appendix and complements}
\subsection{Approximation of \texorpdfstring{$h^{(\delta)}$}{h}}\label{subsect:approximation_mean_Bellman-Harris}

Function $h^{(\delta)}$ can not be approximated with Monte-Carlo simulations because the error of the empirical estimator of $h^{(\delta)}$ grows like $e^{\alpha t/2}$ when $t$ increases, in view of Theorem \ref{thm:main_result}. In particular, if we denote $\widehat{h}^{(\delta)}$ the estimator obtained with Monte-Carlo simulations, the function
$$
x \mapsto \Var\left((\mathbb{E}_x\left[N_{\delta}\right] - e^{\alpha\delta})1_{[0,\zeta[}(x) + 2 \widehat{h}^{(\delta)}(x - \zeta)1_{[\zeta, +\infty[}(x)\right) e^{-\alpha\,x}
$$
is not integrable. Thus, to approximate $h^{(\delta)}$, we proceed differently. We use the fact that for all $t\geq0$, the quantity $\mathbb{E}[N_t]$ is equal to the following when $k\in\mathbb{N}^*$
\begin{equation}\label{eq:approximation_mean_Bellman-Harris}
\overline{\mathbb{E}[N_t]} = \frac{1}{2k}\sum_{l = -\left\lceil \frac{k}{2} \right\rceil + 1}^{\left\lfloor \frac{k}{2} \right\rfloor}\frac{2^{\frac{1}{k}}\exp\left(\frac{2\pi\,l}{k}i\right)}{2^{\frac{1}{k}}\exp\left(\frac{2\pi\,l}{k}i\right) - 1}\exp\left(\frac{2^{\frac{1}{k}}\exp\left(\frac{2\pi\,l}{k}i\right) - 1}{\theta} t\right).\\
\end{equation}
We also use the fact that $\mathbb{E}[N_t]$ can be approximated by the following when $k\notin \mathbb{N}^*$ and~$t > 0$,
\begin{equation}\label{eq:gaver_stehfest_mean_Bellman-Harris}
\widehat{\mathbb{E}[N_t]} = \frac{\ln(2)}{t}\sum_{l = 1}^{200}\frac{C_lt}{l\ln(2)} \frac{\left(1+l\theta\ln(2)/t\right)^k - 1}{\left(1+l\theta\ln(2)/t\right)^k - 2},
\end{equation}
where for all $l\in\llbracket1,200\rrbracket$
$$
C_l = \frac{(-1)^{l+100}}{100!} \sum_{j = \left\lfloor\frac{l+1}{2}\right\rfloor}^{\min(l,100)} j^{100+1} \binom{100}{j}\binom{2j}{j}\binom{j}{l-j}.
$$
Then, we approximate $h^{(\delta)}$ with $t \mapsto \overline{\mathbb{E}[N_{t+\delta}]} - e^{\alpha\delta}\overline{\mathbb{E}[N_{t}]}$ or $t \mapsto \widehat{\mathbb{E}[N_{t+\delta}]} - e^{\alpha\delta}\widehat{\mathbb{E}[N_{t}]}$, depending on the value of $k$.

In this section, we explain why the expressions given in \eqref{eq:approximation_mean_Bellman-Harris}-\eqref{eq:gaver_stehfest_mean_Bellman-Harris} can be used to approximate the function $t\mapsto \mathbb{E}[N_t]$. When $k\in\mathbb{N}^*$, we show that $\mathbb{E}[N_t] = \overline{\mathbb{E}[N_t]}$ for all~$t\geq0$. When $k\notin \mathbb{N}^*$,  we explain what the function $t\mapsto \widehat{\mathbb{E}[N_t]}$ represents and why this is a good approximation of $t\mapsto \mathbb{E}[N_t]$.

\subsubsection{Approximation when $k\in \mathbb{N}^*$}

For all $t\geq 0$, we denote $m(t) = \mathbb{E}[N_t]$ and $\overline{G}(t) = 1- G(t)$. Our aim is to prove that $m$ is equal to the right-hand side of~\eqref{eq:approximation_mean_Bellman-Harris}, by doing an inverse Laplace transform. Using the integral equation, we have
\begin{equation}\label{eq:integral_equation_mean}
m(t) = \overline{G}(t) + 2\int_0^t  m(t-s)g(s) ds = \overline{G}(t) + 2\int_0^t  m(s)g(t-s) ds .
\end{equation}
Then, taking the Laplace transform of $m$, we obtain that for all $p\in \mathbb{C}$ such that $Re(p) > \alpha$, 
$$
\mathcal{L}m(p) = \mathcal{L}\overline{G}(p) + 2\mathcal{L}g(p)\mathcal{L}m(p),
$$
implying that
$$
\mathcal{L}m(p) = \frac{\mathcal{L}\overline{G}(p)}{1 - 2\mathcal{L}g(p)}.
$$
We already know the value of $\mathcal{L}g(p)$ from \eqref{eq:laplace_transform_gamma}. In addition, with usual computations related to Laplace transforms, see \cite[p.~$256$]{folland_2009},$$\mathcal{L}\overline{G}(p) = \frac{\overline{G}(0) - \mathcal{L}g(p)}{p} = \frac{1 - \mathcal{L}g(p)}{p}.$$ Combining these, we obtain for all $p\in \mathbb{C}$ such that $Re(p) > \alpha$,
\begin{equation}\label{eq:relation_mean_laplace_transform}
\mathcal{L}m(p) = \frac{1}{p} \frac{(1+p\theta)^k - 1}{(1+p\theta)^k - 2}.
\end{equation}%In addition, the polynomial at the denominator in~\eqref{eq:laplace_transform_mean} has for degree $k$. 
We now derive an expression of $\mathcal{L}m$ that simplifies the inversion of the Laplace transform.  Function $\mathcal{L}(m)$ is a rational fraction since $p\mapsto ((1+p\theta)^k - 1)/p$ is   polynomial. In addition, in view of~\eqref{eq:expression_eigenvalues_gamma}, the complex numbers $\beta_l= 2^{\frac{1}{k}}\exp\left(\frac{2\pi l}{k}i\right)/\theta$, for all $l \in \llbracket-\lceil \frac{k}{2} \rceil+1, \lfloor \frac{k}{2} \rfloor\rrbracket$ are the roots of the polynomial at the denominator. By doing a partial fraction decomposition of~$\mathcal{L}m$, we then obtain that for all $p\in \mathbb{C}$ verifying $Re(p) > \alpha$ 
\begin{equation}\label{eq:laplace_transform_mean}
\mathcal{L}(m)(p) = \sum_{l = -\left\lceil \frac{k}{2} \right\rceil+1}^{\left\lfloor \frac{k}{2} \right\rfloor}\frac{a_l}{\left(p-\beta_l\right)},
\end{equation}
where for all $l\in \llbracket-\lceil \frac{k}{2} \rceil+1, \lfloor \frac{k}{2} \rfloor\rrbracket$, using that $(1+\beta_l\theta)^{k} = 2$, it holds
\begin{equation}\label{eq:coeffs_laplace_transform_mean}
a_l = \lim_{p \rightarrow \beta_l} \left(p-\beta_l\right)\mathcal{L}(m)(p) = \frac{1}{\beta_l} \lim_{p \rightarrow \beta_l} \frac{\left(p-\beta_l\right)}{(1+p\theta)^k - 2} = \frac{1}{\beta_lk\theta(1+\beta_l\theta)^{k-1}} = \frac{1+\beta_l\theta}{2\beta_lk\theta}.
\end{equation}
%\prod_{l =-\left\lceil \frac{k}{2} \right\rceil+1}^{\left\lfloor \frac{k}{2} \right\rfloor}
%\overset{{\left\lfloor \frac{k}{2} \right\rfloor}}{\underset{l =-\left\lceil \frac{k}{2} \right\rceil+1}{\prod}}(p-\beta_l)
% = \frac{1}{2k}\sum_{l = -\left\lceil \frac{k}{2} \right\rceil+1}^{\left\lfloor \frac{k}{2} \right\rfloor}\frac{2^{\frac{1}{k}}\exp\left(\frac{2\pi\,l}{k}i\right)}{2^{\frac{1}{k}}\exp\left(\frac{2\pi\,l}{k}i\right) - 1}\frac{1}{p-\frac{2^{\frac{1}{k}}\exp\left(\frac{2\pi\,l}{k}i\right) - 1}{\theta}}
%$\left(2^{\frac{1}{k}}e^{\frac{2\pi\,l}{k}i}/(2^{\frac{1}{k}}e^{\frac{2\pi\,l}{k}i} - 1)\right)_{l\in\left\llbracket -\left\lceil \frac{k}{2} \right\rceil+1,\left\lfloor \frac{k}{2} \right\rfloor\right\rrbracket}$
%\ju{In the above, the values of the coefficients before the terms $\frac{1}{p-\beta_l}$, where $\beta_l = \frac{2^{\frac{1}{k}}\exp\left(\frac{2\pi\,l}{k}i\right) - 1}{\theta}$, comes by computing the value of }
%$\ju{\lim_{p \rightarrow \beta_l} =  \frac{1}{p} \frac{(1+p\theta)^k - 1}{(1+p\theta)^k - 2}\left(p-\beta_l\right),}$
%\ju{ in view of the fact that $$
%\lim_{p \rightarrow \beta_l} \frac{(1+p\theta)^k - 2}{p - \beta_l}= \frac{d}{d p}\left((1+p\theta)^k - 2\right)\big|_{p = \beta_l} = k\theta(1+\beta_l\theta)^{k-1}.$$} 
Let us use these two equations to compute the value of $m$. First, we do an analytic continuation, and use that for any $b\in\mathbb{C}$ the inverse Laplace transform of $p \mapsto 1/(p-b)$ is $t\mapsto e^{bt}$, to invert the Laplace transform in~\eqref{eq:laplace_transform_mean}. Then, for all $l\in \llbracket-\lceil \frac{k}{2} \rceil+1, \lfloor \frac{k}{2} \rfloor\rrbracket$, we plug the definition of $\beta_l$ in~\eqref{eq:coeffs_laplace_transform_mean}, to have the value of the coefficient $a_l$. We obtain that for all $t\geq 0$
$$
m(t) = \frac{1}{2k}\sum_{l = -\left\lceil \frac{k}{2} \right\rceil+1}^{\left\lfloor \frac{k}{2} \right\rfloor}\frac{2^{\frac{1}{k}}\exp\left(\frac{2\pi\,l}{k}i\right)}{2^{\frac{1}{k}}\exp\left(\frac{2\pi\,l}{k}i\right) - 1}\exp\left(\frac{2^{\frac{1}{k}}\exp\left(\frac{2\pi\,l}{k}i\right) - 1}{\theta} t\right).
$$
Thus, $m(t)$ is exactly the right-hand side of~\eqref{eq:approximation_mean_Bellman-Harris}.
%As this corresponds exactly to the right-hand side of $\overline{\mathbb{E}[N_t]}$, we have that .
\subsubsection{Approximation when $k\notin \mathbb{N}^*$} 
When $k\notin \mathbb{N}^*$ (and $k\geq 1$), we can not do the partial fraction decomposition as done in the previous case, and therefore can not obtain an explicit value for $t \mapsto \mathbb{E}[N_t]$. Nevertheless, \eqref{eq:relation_mean_laplace_transform} is still true. We can thus use numerical methods for inverting Laplace transforms to approximate $t \mapsto \mathbb{E}[N_t]$. The expression given in~\eqref{eq:gaver_stehfest_mean_Bellman-Harris} comes from one of these methods. Precisely, it comes from the Gaver-Stehfest algorithm~\cite{gaver_observing_1966,stehfest_algorithm_1970,stehfest_remark_1970}. This algorithm allows to approximate a smooth function $f : \mathbb{R}_+^* \rightarrow \mathbb{R}$ when its Laplace transform $\mathcal{L}(f)$ is known. It has the advantage of being fast, easy to implement, and relatively accurate when the exact value of the Laplace transform is known. Its analytical expression is given for all $t\geq0$ as
$$
f_n(t) = \frac{\ln(2)}{t} \sum_{l=1}^{n} C_l(n) \mathcal{L}(f)\left(\frac{l \ln(2)}{t}\right),
$$
where $n\in\mathbb{N}^*$ is an even parameter, and for all $l\in\llbracket1,n\rrbracket$
$$
a_l(n) = \frac{(-1)^{l+n/2}}{(n/2)!} \sum_{j = \left\lfloor\frac{l+1}{2}\right\rfloor}^{\min(l,n/2)} j^{n/2+1} \binom{n/2}{j}\binom{2j}{j}\binom{j}{l-j}.
$$
In view of~\eqref{eq:gaver_stehfest_mean_Bellman-Harris} and~\eqref{eq:relation_mean_laplace_transform}, one can easily see that in fact,~$t\mapsto \widehat{\mathbb{E}\left[N_t\right]}$ corresponds to the Gaver-Stehfest algorithm with parameter $n = 200$.

%This result has been refined in~\cite{kuznetsov_rate_2020} in which the author proved that this convergence holds at an exponential rate when $f$ is analytic.
%This algorithm  
%Let us briefly present this algorithm, and how we derive~\eqref{eq:gaver_stehfest_mean_Bellman-Harris} from it.

In~\cite{kuznetsov_convergence_2013}, it has been proved that if $f$ is continuous at a point $t_0 >0$ and has bounded variation in one of its neighbourhood, then it holds $\lim_{n\rightarrow+\infty}f_n(t_0) = f(t_0)$. This result has been refined in~\cite{kuznetsov_rate_2020} in which the authors have obtained the rate of convergence under additional assumptions. Thus, the parameter $n$ determines the accuracy of the approximation, and the larger it is, the better the approximation of $f$ is. In our case, we have chosen $n = 200$ to define $t\mapsto \widehat{\mathbb{E}\left[N_t\right]}$, which is a good trade-off between having a large parameter and having too long computations. In addition, the function \( t \mapsto \mathbb{E}\left[N_t\right] \) is continuous on~\( \mathbb{R}_+^* \) and has bounded variation on compact sets. This is because, by~\eqref{eq:integral_equation_mean}, $t \mapsto \mathbb{E}\left[N_t\right]$ is the fixed point of the operator $L(h) = \left(1 - G\right) + 2g * h$, which maps locally bounded functions $h: \mathbb{R}_+^* \rightarrow \mathbb{R}$ to continuous functions with bounded variation on compact sets, as~$g$ and~$G$ have these two properties. Then, in view of the theoretical result obtained in~\hbox{\cite{kuznetsov_convergence_2013}}, we have that~\hbox{$t\mapsto \widehat{\mathbb{E}\left[N_t\right]}$} is a good approximation of~$t\mapsto\mathbb{E}\left[N_t\right]$.%, and which allows a relatively fast calculation time} 

We emphasize that as the Gaver-Stehfest algorithm is computationally expensive, when we use \hbox{$t\mapsto \widehat{\mathbb{E}\left[N_t\right]}$} to approximate $t\mapsto\mathbb{E}\left[N_t\right]$ on an interval $(0,M)$, where $M>0$, we do not compute its values for every~$t\in(0,M)$. Instead, we compute its values for a discretization of the interval $\left(0,M\right)$ with~$1000$ points, and use interpolation techniques to obtain an approximation on the full interval. It is sufficient to be able to approximate $t\mapsto\mathbb{E}\left[N_t\right]$ on every interval $(0,M)$, where $M>0$, instead of the full interval $(0,+\infty)$. This is because this approximation is used to compute the integral of the function 
$$
x\in(0,+\infty) \mapsto \Var\left((\mathbb{E}_x\left[N_{\delta}\right] - e^{\alpha\delta})1_{[0,\zeta[}(x) + 2 h^{(\delta)}(x - \zeta)1_{[\zeta, +\infty[}(x)\right) e^{-\alpha\,x},
$$
where~$h^{(\delta)}(x) = \mathbb{E}\left[N_{x+\delta}\right]- e^{\alpha\delta}\mathbb{E}\left[N_x\right]$ for all $x\geq0$, and that the method used to compute this integral (rectangle method) only requires to have the values of the above function on a large subset of $(0,+\infty)$.

\subsection{Variance  of \texorpdfstring{$N_t$}{N\_t} for Gamma law and sensitivity}
\label{subsect:variancegamma_sensitivity}\
In view of \cite[$p.152-153$]{athreya_branching_1972}, there exist $n_1 >0$, $n_2 >0$ such that

$$
\begin{aligned}
\mathbb{E}[N_t] \underset{t\longrightarrow+\infty}{\sim} n_1e^{\alpha t}, \text{ and }\Var(N_t) \underset{t\longrightarrow+\infty}{\sim} n_2e^{2\alpha t},
\end{aligned}
$$
and such that 
$$
\frac{n_2}{(n_1)^2} = \frac{4\mathcal{L}(g)(2\alpha) - 1}{1- 2\mathcal{L}(g)(2\alpha)}.
$$
Using the expression of $\mathcal{L}(g)$ given in \eqref{eq:laplace_transform_gamma} to develop this ratio, and then using the fact that $k = \left(\mu/\sigma\right)^2$, we obtain \eqref{eq:ratio_var_esp}, where for all $x\geq 0$

\begin{equation}\label{eq:q_function}
q(x) = \frac{4\left(2^{x^2+1} - 1\right)^{-\frac{1}{x^2}} - 1}{1 - 2\left(2^{x^2+1} - 1\right)^{-\frac{1}{x^2}}}.
\end{equation}

Let $\alpha = 1$. We show now that $\phi(\sigma/\mu) = \sigma_{\delta}^{2}\left(1/(\sigma/\mu)^2,\alpha\right)$, with $\delta = \log(2)/\alpha$, is more sensible to the variation of $\sigma/\mu$ than the quantity $q(\sigma/\mu)$ for $\sigma/\mu \leq 0.65$. In Figure~\ref{fig:explication_sensibility_variance_coeffvar}, we first plot $\phi(\sigma/\mu,\alpha)$ versus $\sigma/\mu$, using the approximator $\overline{\sigma}_{\delta}^2$ of $\sigma_{\delta}^2$. We observe that the curve is almost linear, with a slope of $1.303$ (this is obviously not the case, but we will use this approximation). We have compared this value with the derivative of $q$ in Figure \ref{fig:explication_sensibility_comparison_derivative}, and we see that the slope is always larger than the derivative of $q$ for $\sigma/\mu \in[1/(57.2)^{1/2},0.65] \simeq  [0.1322,0.65]$.

\begin{figure}[!ht]
    \centering
    \begin{subfigure}[t]{0.475\textwidth}
        \centering
        \includegraphics[scale = 0.31]{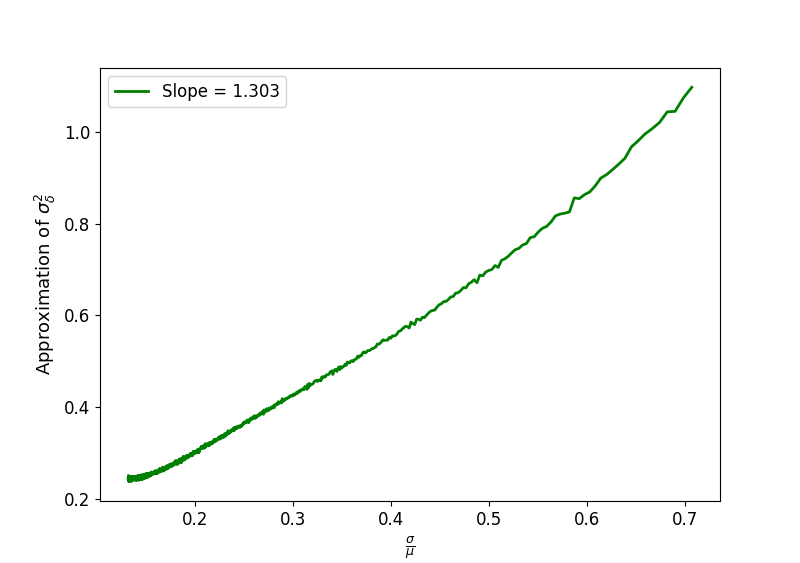}
        \caption{$\phi(\sigma/\mu)$ versus $\sigma/\mu$ for $\sigma/\mu \in[1/(57.2)^{1/2},1/2^{1/2}]$.}\it{The coefficient of determination of the linear regression of $\phi$ versus the coefficient of variation is $R^2 = 0.9885$.}
        \label{fig:explication_sensibility_variance_coeffvar}
    \end{subfigure}
    \hfill
    \begin{subfigure}[t]{0.475\textwidth}
        \centering
        \includegraphics[scale = 0.31]{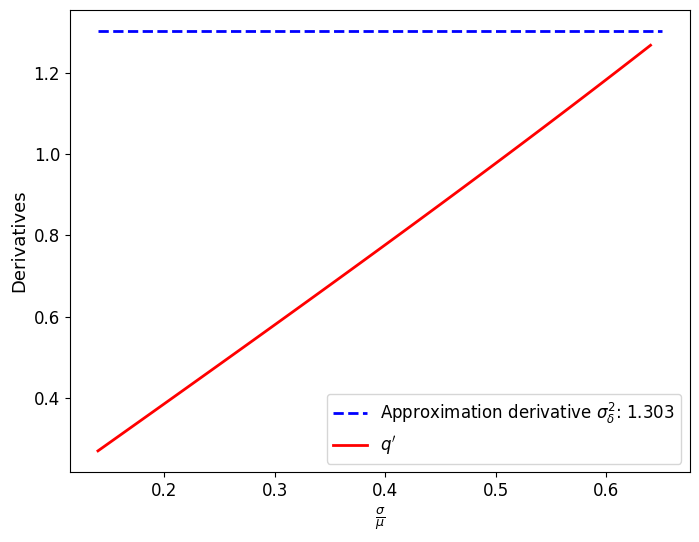}
        \caption{$q'(\sigma/\mu)$ versus $\sigma/\mu$ for $\sigma/\mu \in[1/(57.2)^{1/2},0.65]$ (red). Comparison with the slope of $\phi(\sigma/\mu,\theta)$ (blue).}
        \label{fig:explication_sensibility_comparison_derivative}
    \end{subfigure}
    \caption{Illustration that $\phi(\sigma/\mu)$ is more sensible to the variation of $\sigma/\mu$ than $q(\sigma/\mu)$. \it{$\sigma_{\delta}^2$ has been approximated by $\overline{\sigma}_{\delta}^2$ to plot these curves, using the grid of parameters $\mathbb{G}_{1/20}$ to compute it}}
    \label{fig:explication_sensibility}
\end{figure}

%\begin{figure}[!ht]
%    \centering
%    \includegraphics[scale = 0.27]{images/divers/explication_sensibility.jpg}
%    \caption{At the left: $\overline{\sigma}_{\log(2)/\alpha(\sigma/\mu,\theta)}^{2}(\sigma/\mu,\theta)$ versus $\sigma/\mu$. At the right: $\frac{\partial}{\partial (\sigma/\mu)}q(\sigma/\mu)$ versus $\sigma/\mu$ (red) and comparison with the slope of $\overline{\sigma}_{\log(2)/\alpha(\sigma/\mu,\theta)}^{2}(\sigma/\mu,\theta)$ versus $\sigma/\mu$ (blue).}
%    \label{fig:explication_sensibility}
%\end{figure}

Then, in view of the mean value inequality and using the approximation that $\phi$ is linear, we can conclude that for all $(a,b)\in[1/(57.2)^{1/2},0.65]$ 
$$
\frac{|q(b) - q(a)|}{|b-a|} \leq \max_{x\in[a,b]} |q'(x)| \leq 1.303 \approx \frac{\left|\phi(b,\theta)-\phi(a,\theta)\right|}{|b-a|}.
$$
Thus, $\phi(\sigma/\mu)^{2}(\sigma/\mu,\theta)$ is more sensible to the variation of $\sigma/\mu$ than the quantity $q(\sigma/\mu)$, when $\sigma/\mu\in [1/(57.2)^{1/2},0.65]$. This interval contains most of the biologically-relevant coefficient of variations in the Gaussian regime. Then, $\sigma_{\delta}^2$ seems more relevant to use than~$q$ to obtain information on the variability of division times. 

For coefficient of variations greater than $0.65$, the coefficient of variation seems too large compared to what is biologically true. If we ignore this, even if $q$ is more sensible, there is still difficulties to use $q$ in practice, due to the fact that there are a lot of source of randomness coming from the early stages of the experiment~\cite{stukalin_age-dependent_2013}. Our method of estimation do not have these problems, that is why we think it is still more relevant to use it rather than $q$.

%This interval of values seems to contain the more biologically-relevant values for which we are in the Gaussian regime.

\subsection{Extension of the asymptotic fluctuations to Bellman-Harris processes}\label{subsect:extension_other_bellman_harris} %\ju{need $p_0 = 0$ ? probleme que $W = 0$ sinon ?}
Theorem \ref{thm:main_result} can be extended to Bellman-Harris processes that satisfy the following assumptions:
\begin{enumerate}[(A)]
    \item The reproduction law has a finite third  and fourth moment, and  the probability to give no offspring is $0$. In addition, the first moment of this law, that we denote $m_1$, is strictly greater than $1$. 
    \item Lifetime distribution is given by a density $g$, %\ju{l'appeler $g$ aussi ?}
    that is bounded continuous on $\mathbb{R}_+$, and continuously differentiable on $(0,+\infty)$. In addition, $g$ has a finite first moment.
    \item The following holds
    $$
    \sup_{s\geq0} \left|\frac{g(s)}{1 - G(s)}\right| < +\infty.
    $$
    \item For all  $\delta >0$, there exists $C > 0$, $l > -1$, such that for all $a\geq0$, $x\in(0,\delta]$ (we denote $G$ the cumulative distribution associated to $g$),
    $$
    \left|\frac{\frac{\partial\,g(a+x)}{\partial\,a}(a,x)}{1 - G(a)}\right| \leq C\left(1+x^{l}\right).
    $$
    \item For all $\lambda \in\mathbb{C}$ such that $\mathcal{L}(g)(\lambda) = 1/m_1$, $\left(\mathcal{L}\right)'(g)(\lambda) \neq 0$.
\end{enumerate}

Indeed, conditions on the moments of the reproduction law, given by Assumption $(A)$, imply Lemma~\ref{lemma:preliminaries_asymptotic_Xt}~$(i)$ . Adding the boundedness and the continuity of $g$, given by $(B)$, allow to obtain Lemma~\ref{lemma:preliminaries_asymptotic_Xt}~$ii)-iii)$. Then, we get Proposition \ref{prop:asymptotic_Xt} from Lemma \ref{lemma:preliminaries_asymptotic_Xt}.

Now, using the inequalities presented in $(C)$ and $(D)$, we obtain an analog to Lemma~\ref{lemma:preliminaries_asymptotic_Yt}: $h_1$ is continuously differentiable, and there exists $C_{h_1} >0$ such that $h_1(a) \leq C_{h_1}$ for all~$a\in\mathbb{R}_+$. Then, we apply \cite[Theorem $2.9$]{iksanov_asymptotic_2024}, for which we need to have that $m_1 > 1$ and that $g$ has a finite first moment, to obtain Propositions \ref{prop:application_Iksanov} and \ref{prop:asymptotic_Yt}. We finally conclude the proof, by doing similar steps as those presented in the proof of Theorem~\ref{thm:main_result}.

We require to have a probability to give no offspring of $0$, to be sure that we have $W > 0$ a.s., and then that we can divide $R_{t,\delta}$ by $W$. This assumption can be relaxed by considering events where the population survives. Assumption $(E)$ can also be relaxed, but the analytic expression presented in Theorem \ref{thm:main_result} $ii)-iii)$ must be adapted, adding polynomial terms. We refer to \cite{iksanov_asymptotic_2024} for that.

As for the Gamma distribution, two different regimes exist for the asymptotic fluctuations when $(A-E)$ are verified: Gaussian or oscillating. Again, this depends on the spectral gap (i.e. the gap between the first and second eigenvalue). More precisely, denoting $\alpha$ the positive real number verifying~$\mathcal{L}(g)(\alpha) = \frac{1}{m_1}$ (unique by the intermediate value theorem), if the set
\begin{equation}\label{eq:set_spectral_gap_general}
\left\{\rho\in\mathbb{C}\backslash\{\alpha\}\,:\,\text{Re}(\rho) \geq \frac{\alpha}{2} ,\,\mathcal{L}(g)(\rho) = \frac{1}{m_1}\right\},
\end{equation}
is empty, then we are in the Gaussian regime. If it is equal to~$\left\{\frac{\alpha}{2}\right\}$, then we are in the critical regime. Otherwise, we are in the oscillating regime. However, in comparison with the Gamma distribution, the Laplace transform is, in most of the case, not explicit in function of the coefficient of variation of the law. It it thus in general not possible to obtain a relation between the eigenvalues of the process and the coefficient of variation of the lifetimes as the one we have obtained in~\eqref{eq:expression_eigenvalues_gamma}. The eigenvalues can however still be approximated with numerical methods (Newton's method, bisection method etc...).

For unimodal distributions, we strongly believe that the variability of the lifespan has still a significant influence on the value of the spectral gap. Indeed, the lower the variability is, the more the cells will divide “at the same time”, which will generate oscillations in the cell number over time, see~Figures~\ref{fig:dynamics_oscillatory} and~\ref{fig:mean_evolution}. For other types of distribution, we expect that the factors implying that the set presented in~\eqref{eq:set_spectral_gap_general} is non-empty are more complex. However, the motivations of this paper concern the estimation of unimodal lifetime distributions.
\subsection{The exponential case}
The case of exponential lifetimes seems not relevant for modeling many biological phenomenons such as the time between two divisions. In that case, the memory less property allows to simplify the analysis. The behavior of the fluctuations 
$$R_t^{\delta}=N_{t+\delta} -e^{\delta  \alpha} \, N_t,$$
where~$t\geq0$, $\delta>0$, is simply obtained by a classical central limit theorem. Indeed,
\begin{align*}
R_t^{\delta}=\sum_{i=1}^{N_t} \left(N_{t,\delta}^i-e^{\delta  \alpha}\right),
\end{align*}
and $(N_{t,\delta}^i-e^{\delta  \alpha} : 1\leq i \leq N_t)$ are i.i.d. variables independent of $N_t$. Then, if we denote $\nu >0$ the parameter of the exponential law, as for all $i \in 1 \leq i \leq N_t$, $\text{Var}\left(N_{t,\delta}^i\right) = e^{2\nu\delta} - e^{\nu \delta}$, we have 
$$
R_{t}^{\delta} \overset{\mathcal{L}}{\underset{t\longrightarrow+\infty}{\Longrightarrow}} \mathcal{N}\left(0,e^{2\nu\delta} - e^{\nu\delta}\right).
$$
In particular,  there is one single regime, the Gaussian one, with a simple interpretation of the limiting variance. 

As an exponential law is a gamma law with parameter $k = 1$, this result can also be obtained with Theorem \ref{thm:main_result}. 

$\newline$
$\newline$

\noindent {\bf Acknowledgement.} We warmly thank Marc Hoffmann for stimulating discussions on this topic and this work. We also thank the anonymous referee for his/her insightful comments that help us to improve the overall quality of this manuscript. This work  was partially funded by the Chair “Mod\'elisation Math\'ematique et Biodiversit\'e" of VEOLIA-Ecole Polytechnique-MNHN-F.X., by the European Union (ERC, SINGER, 101054787), and by the Fondation Mathématique Jacques Hadamard. Views and opinions expressed are however those of the author(s) only and do not necessarily reflect those of the European Union or the European Research Council. Neither the European Union nor the granting authority can be held responsible for them. 
\medskip

\noindent {\bf Authors' contributions.} Jules Olayé has conducted this work for the mathematical part, the simulations and the estimations from data. Vincent Bansaye and Charles Baroud  have supervised this work. Antoine Barizien  and Vincent Bansaye have made the first investigations. Jules Olayé and Hala Bouzidi have realized a project during their Master~2, which have given the first theoretical and numerical results of this paper. Andrey Aristov and  Salomé Gutiérrez Ramos have realized the microfuidics experiments. Jules Olayé and Vincent Bansaye have written the paper based on  all these contributions.
\medskip

\noindent {\bf Data availability statement.} All data supporting the findings of this study are available in the Supplementary Material.

\emergencystretch=3em
\printbibliography
\end{document}